\documentclass[11pt]{article}

\usepackage[margin=1in]{geometry}
\usepackage{amsthm,amsmath,amsfonts,amssymb}
\usepackage[numbers]{natbib}
\usepackage{graphicx}
\usepackage{subfigure,multirow}
\usepackage{algorithm,algorithmicx,algpseudocode}
\usepackage[colorlinks,citecolor=blue,urlcolor=blue,linkcolor=blue]{hyperref}

\theoremstyle{plain}

\newtheorem{theorem}{Theorem}[section]
\newtheorem{lemma}[theorem]{Lemma}
\newtheorem{corollary}[theorem]{Corollary}
\newtheorem{proposition}[theorem]{Proposition}

\theoremstyle{definition}

\newtheorem{assumption}{Assumption}

\theoremstyle{remark}

\newtheorem{remark}[theorem]{Remark}

\newcommand{\E}{\mathbb{E}}
\newcommand{\X}{\mathbb{X}}
\newcommand{\Real}{{\mathbb{R}}}
\newcommand{\bs}[1]{\boldsymbol{#1}}
\title{Variance Reduction for Independent Metropolis}
\author{Siran Liu\thanks{Department of Statistical Science, University College London, UK. Email: \href{mailto:siran.liu.21@ucl.ac.uk}{siran.liu.21@ucl.ac.uk}}
\and Petros Dellaportas\thanks{Department of Statistical Science, University College London, UK, and Department of Statistics, Athens University of Economics and Business, Greece. Email: \href{mailto:p.dellaportas@ucl.ac.uk}{p.dellaportas@ucl.ac.uk}}
\and Michalis K. Titsias\thanks{Google DeepMind, UK. Email: \href{mailto:mtitsias@gmail.com}{mtitsias@gmail.com}}}
\date{}

\begin{document}

\maketitle

\begin{abstract}
Assume that we would like to estimate the expected value of a function $F$ with respect to an intractable  density $\pi$, which is specified up to some unknown normalising constant.   We prove that if $\pi$ is close enough under KL divergence to another density $q$, an independent Metropolis sampler estimator that obtains samples from $\pi$ with proposal density $q$, enriched with a variance reduction computational strategy based on control variates, achieves smaller asymptotic variance than i.i.d.\ sampling from $\pi$. The control variates construction requires no extra computational effort but assumes that the expected value of $F$ under $q$ is analytically available. We illustrate this  result by calculating the marginal likelihood in a linear regression model with  prior-likelihood conflict and a non-conjugate prior. Furthermore, we propose an adaptive independent Metropolis algorithm that adapts the proposal density such that its KL divergence with the target is being reduced.    We demonstrate its applicability in a Bayesian logistic and Gaussian process regression problems and we rigorously justify our asymptotic arguments under  easily verifiable and essentially minimal conditions.


\smallskip
\noindent\textbf{MSC2020 subject classifications:} Primary 65C05, 60J22; secondary 65C40, 62F15.
\end{abstract}

\section{Introduction}
\subsection{The general problem}
 We are given a probability distribution $\pi(x) = \frac{1}{\mathcal{Z}} \exp\{f(x)\}$ on some state space $\X \subset \Real ^d$ where the normalising constant
 $\mathcal{Z}$ is intractable, and  we want to estimate the expectation of a function $F: \X \rightarrow \Real$ with respect to $\pi(\cdot)$ given by
\begin{equation}
\E_{\pi} [F] = \int_\X F(x) \pi(x) dx.
\label{Expectation}
\end{equation}
A Monte Carlo (MC) estimation of this expectation simulates, typically using Markov chain Monte Carlo (MCMC), a set of correlated random variables $X_1,X_2,\ldots,X_n \sim \pi$ and then produces the {consistent} estimator
\begin{equation}
\label{cmcest}
        \mu_{n,MC} = \frac{1}{n}\sum_{i=1}^n F(X_i),
\end{equation}
with variance estimator $\sigma^2_{n,MC}$. If the MCMC samples are non-negatively correlated, the optimal Monte Carlo 
estimator $\mu_{n,MC^*}$ of the form \eqref{cmcest} that achieves the smallest variance $\sigma^2_{n,MC^*}$ 
is obtained when $X_1,X_2,\ldots,X_n$ are i.i.d.\ 
from $\pi(\cdot)$. Then it holds 
\begin{equation}
n\sigma^2_{n,MC} \geq n \sigma^2_{n,MC^*} = \E_{\pi} [(F(x)- \E_{\pi}[F])^2] = \sigma_F^2,
\label{true_var}
\end{equation}
where $\sigma_F^2$ denotes the true variance of $F$. Often  
the variance $\sigma^2_{n,MC}$, 
or its lower bound  $\sigma^2_{n,MC^*}$, 
is very large and variance reduction techniques are needed. 
One general method for variance reduction can be based on control variates \citep[e.g.,][]{owen2013monte}:  if there exist functions (control variates)
$\{U_j\}_{j=1}^k$ such that $\E_\pi[U_j]$ are analytically available,
then the estimator
$\mu_{n,MC}+\sum_{j=1}^{k}c_j\,\frac{1}{n}\sum_{i=1}^{n}
\big(U_j(X_i)-\E_\pi[U_j]\big)$
is unbiased, and for appropriate values of the scalars $c_j$, has no
larger variance than that of $\mu_{n,MC}$. However, the requirement
to know the expected value $\E_\pi[U_j]$ under the target $\pi$ is
restrictive since even for the simplest functions, such as
$U_j(x)=x$, this expectation is not analytically available. While
there are few special cases such as the score function with known
expectation, in order to be flexible when selecting $U_j$ it would be
useful to remove the dependence of the estimator on $\E_\pi[U_j]$. To
this end, this article introduces control variate estimators that
avoid $\E_\pi[U_j]$ through the use of the independent Metropolis
sampler.

An independent Metropolis estimator can be derived by constructing a Markov chain on $\X$ which has $\pi(\cdot)$ as a stationary distribution by letting a density $q: \X \rightarrow \Real$ be the transition function, or proposal density, of the chain. The independent Metropolis algorithm moves from any state $x \in \X$ by proposing a new location $y\in \X$ from $q$  and accepting it with probability
\begin{equation}
\label{alpha}
\alpha(x,y) :=   
\min \left(1,\frac{\exp\{f(y)\}q(x)}{\exp\{f(x)\}q(y)} \right).
\end{equation}
By collecting, at stationarity, dependent samples $X_1,X_2,\ldots,X_n$ from the Markov chain, we construct the ergodic estimator 
\begin{equation}
\label{mcmcest}
        \mu_{n,IM} = \frac{1}{n}\sum_{i=1}^n F(X_i),
\end{equation}
which satisfies, under appropriate conditions (see section 2), a central limit theorem of the form
$
\sqrt n ( \mu_{n,IM} - \E_\pi[F] ) = n^{-1/2} \sum_{i=1}^{n} 
( F(X_i) - \E_\pi[F] )
\overset{D}{\to} N(0,\sigma^2_{IM})
$
with the asymptotic variance given by 
\begin{equation}
\label{sigmaIM}
\sigma^2_{IM} := \lim_{n \to \infty} n \E_\pi \left[ \left( \mu_{n,IM} - \E_\pi[F]\right)^2 \right].
\end{equation}

\subsection{Outline of the method}

Since $\mu_{n,IM}$ is based on dependent samples, it is an instance of the
estimator $\mu_{n,MC}$ from \eqref{cmcest}. 
{Moreover, the independent
Metropolis kernel is a positive operator, so the autocovariances of \(F(X_i)\) are non-negative at all lags  (see Proposition~\ref{prop1}, part~4)}, and
therefore $\sigma^2_{n,IM} \geq \sigma^2_{n,MC^*}$, where recall that
$\sigma^2_{n,MC^*}$ denotes the variance of i.i.d.\ sampling.
Indeed $\sigma^2_{IM} = \sigma^2_{n,MC^*}$ only when $q$ equals  $\pi$. Our key idea  is that we can exploit the Markovian structure of the independent Metropolis samples to build, with negligible extra computational cost, new estimators with variance smaller than $\sigma^2_{n,MC^*}$.  We derive these new estimators based on approximate solutions to the Poisson equation of the Markov chain that allow the construction of control variates for $\mu_{n,IM}$. 
There are two requirements for these estimators to achieve smaller variance than 
$\mu_{n,MC^*}$. The first is the ability to obtain analytically $\E_q[F]$ or $\E_q[{\tilde F}]$ for some function ${\tilde F}$ close to $F$.  The second is that the proposal density $q$ is close to $\pi$.  This is rigorously justified  in our Theorem \ref{thmbound} which states that when $q$ is close enough to  $\pi$ under KL-divergence, then the variance of the new estimator is smaller than $\sigma^2_{n,MC^*}$.

In the first part of our work, 
we derive the new ergodic estimator for independent Metropolis algorithms for any real valued functions $F$ together with a theoretical result that proves that such an estimator is a good generic candidate for estimating expectations $\E_\pi [F]$ if the proposal density $q$ is close enough to the target density. 
To illustrate our theoretical results and show that the new estimator can have smaller variance
than the i.i.d.\ sampling 
variance $\sigma^2_{n,MC^*}$
we consider a simple Bayesian linear regression problem of computing marginal densities of the form $\int F(x) \pi(x) d x$ where 
$F(x)$ is a Gaussian likelihood and   $\pi(x)$ is a non-conjugate prior.

 In the second part, 
 we build an adaptive independent Metropolis algorithm that updates every $B>1$ iterations the proposal density $q$ such that the KL-divergence $ \mathbb{KL}(q(x) \vert \vert \pi(x))$ is being reduced after every adaptation.  This update is based on a recent strand of research that exploits stochastic gradient-based optimisation techniques 
 for KL-divergence minimisation 
in variational inference 
\citep{Ranganath14}
and adaptive MCMC \citep{Gabrie2022AdaptiveFlows}.  
We provide rigorously justified asymptotic arguments for (i) ergodicity of the resulting independent Metropolis algorithm, (ii) convergence to distribution of the estimators, (iii) a weak law of large numbers when both $B$ 
and the MCMC sampling size go to infinity and (iv) convergence of the scalar estimators that are used to construct our control variates.   
To experimentally evaluate the estimator based on the adaptive independent Metropolis scheme, 
we considered Bayesian  inference problems in logistic regression 
 and Gaussian process regression.
For several functions of interest $F$, including non-linear functions of parameters  such as the odds ratio, 
the variance reduction we achieved by adopting our estimator compared with the simple ergodic estimator of the adaptive independent Metropolis 
varied between $2.7$ and $57.2$.  

\subsection{Related work}\label{sec:related-work}

There is a vast literature of variance reduction in MCMC via control variates; for a long list of references and some critical discussion, see \cite{alexopoulos2023variance, south2026handbook}. {A large family of methods builds the correction from a Stein operator of the target, so that the correction integrates to zero under $\pi$.} A line of research originates from the work of \cite{Assaraf1999Zero-varianceAlgorithms}, who observed that a Hamiltonian operator together with a trial function is sufficient to construct an estimator with zero asymptotic variance. They considered a Hamiltonian operator of Schr{\" o}dinger-type leading to a series of zero-variance estimators studied by \cite{valle2010new}, \cite{mira2013zero} and \cite{papamarkou2014zero}. The estimation of the optimal parameters of the trial function is conducted by ignoring the Markov chain sample dependency, so the theoretical justification of these methods is based on the assumption that the Markov chain samples are i.i.d. This issue was addressed by \cite{belomestny2020variance}, \cite{alexopoulos2023variance} and \cite{brosse2024diffusion}, and was empirically investigated in Section 3.5 of \cite{alexopoulos2023variance}. {The number of parameters of the polynomial trial functions grows with dimension $d$, see \cite{south2018regularised}. The nonparametric analogs of these methods led to the development of control functionals \citep{Oates2017ControlIntegration,Barp2018AMethod,South2022Semi-exactMethod}. Their computational cost does not grow with $d$, but it grows cubically in the number of samples \citep{south2018regularised,Oates2019ConvergenceMethod}, and as nonparametric methods, they can face a statistical curse of dimensionality \citep{2006AllStatistics}. Iterative solvers and score-free variants \citep{fisher2023gradient} mitigate these requirements, so that the practical balance between cost and variance reduction depends on the dimension, the sample size and the availability of the score. Stein operators also underlie a recent line of work on post-processing by
reweighting, which establishes that sampling from a distribution other than
the target and correcting afterwards can dominate standard baselines. This
line includes Stein importance sampling \citep{liu2017blackbox}, its
extension to the output of biased Markov chains
\citep{hodgkinson2020reproducing}, its $\Pi$-invariant refinement
\citep{wang2023stein}, and Stein thinning \citep{riabiz2022thinning}. In
particular, \citet{hodgkinson2020reproducing} prove that Stein importance
sampling remains consistent for samples from a geometrically ergodic Markov
chain whose invariant law differs from the target, as is the case for
unadjusted samplers, and \citet{wang2023stein} show that a Markov chain
invariant to a suitably chosen distribution different from the target,
followed by Stein importance sampling, can outperform the corresponding
target-invariant chain. Our estimator shares this principle, since the IM chain is driven by a proposal close to the target and the ergodic average is corrected afterwards, but it realises the principle within the ergodic-averaging framework rather than through weighted quadrature. The essential structural difference lies in the reference measure under which the correction has zero mean, namely the target for the Stein family and the product law $\pi \otimes q$ of the chain for ours. A detailed comparison covering assumptions, bias, derivative requirements, adaptation and cost, together with the reasons for retaining a Metropolis correction rather than reweighting draws from $q$ directly, is given in Section~\ref{sec:discussion}.}

Another approach, which is closely related to our proposed methodology, attempts to minimise the asymptotic variance of the ergodic estimator of the Markov chain.  One way to achieve this is based on the observation that the solution of the Poisson equation for $F$ (also called the fundamental equation) automatically leads to a zero-variance  estimator for $F$.
Interestingly, a solution of this equation produces the zero-variance estimators suggested by \cite{Assaraf1999Zero-varianceAlgorithms} for a specific choice of Hamiltonian operator. 
One of the rare examples where the Poisson equation has been solved exactly for discrete time Markov chains
is the random-scan Gibbs sampler where the target is a multivariate Gaussian and the goal is to estimate the mean of each component, see \cite{dellaportas2012control,dellaportas2009notes}.  Since a direct solution of the Poisson equation is not available, approximate solutions have been  suggested 
by  \cite{Andradottir1993VarianceSimulations}, \cite{Atchade2005ImprovingAlgorithm},  \cite{henderson1997variance}, \cite{meyn2008control}.  Recently, \cite{alexopoulos2023variance} provided a general framework that constructs estimators with reduced variance for random walk Metropolis and Metropolis-adjusted Langevin algorithms for means of each coordinate of a $d$-dimensional target density.

Our independent Metropolis estimator connects to a Rao-Blackwellized estimator \citep{Robertetal2018}
and a coupling estimator
\citep{pinto2001improving}. 
In Section \ref{sec:previousestimators}
we fully describe both estimators and in Section \ref{sec:adaptiveNumerical}
we compare them against our proposed
estimator. The experimental results show that our new estimator, based on the Poisson equation, has a markedly smaller variance than the
previous methods. 

Adaptive MCMC algorithms with certain ergodic properties were proposed by \citet{Haario2001AnAlgorithm} and \citet{Roberts2009ExamplesMCMC}. An adaptation scheme of independent Metropolis together with  its convergence properties is provided by \citet{Holden2009AdaptiveMetropolis-Hastings}. \citet{Andrieu2006OnAlgorithms} and \citet{Roberts2007CouplingMCMC} studied  the ergodicity properties of adaptive MCMC algorithms and provided central limit theorems and weak laws of large numbers. 
Recently, \citet{Gabrie2022AdaptiveFlows} and \citet{brofos2022adaptation} suggested applying adaptive MCMC by adopting normalizing flows as independent Metropolis proposal densities, which are optimised 
by KL-divergence minimisation using 
the history of states.  In our adaptive MCMC algorithm 
we also use KL-divergence minimisation to optimise 
the proposal density $q$ but unlike 
\citet{Gabrie2022AdaptiveFlows,brofos2022adaptation} we make use of all candidate states (rejected or accepted) 
to adapt $q$. Specifically, we apply
stochastic gradient methods
based on the reprametrisation trick that are widely used for variational inference in machine learning 
\citep{Titsias2014DoublyInference, Rezende2014, Kingma2014, Roeder17, Gianniotis2015, Kucukelbir17, Salimans_2013, Tan18}. 

Finally, a further parallel research path is based on adaptive importance sampling \citep[see, for example,][]{Richard2007EfficientSampling,douc2007convergence,cappe2008adaptive,Cornebise2008AdaptiveModels, cornuet2012adaptive,portier2018asymptotic,dellaportas2019importance, Paananen2021ImplicitlySampling}.  We will discuss and compare our methodology with corresponding adaptive importance sampling strategies in Section \ref{sec:IS}.  The theoretical properties of adaptive importance sampling have been rigorously investigated in the literature; see, for example, \cite{douc2007convergence} and \cite{portier2018asymptotic}.

\subsection{Outline of the paper}
The rest of the paper is organised as follows.  Section $2$ provides the methodological arguments and their rigorous justification for the construction of control variables for the ergodic estimator of the independent Metropolis algorithm, together with illustrations with simulated and real data.  Section $3$ deals with the construction of an adaptive independent Metropolis algorithm that adapts the proposal density such that its KL-divergence with the target is being reduced and illustrates its applicability with synthetic and real data. Section 4 outlines the connections between the proposed method and other related estimators. We conclude with a brief discussion in Section $5$. The proofs of theorems and more detailed discussions can be found in the Supplementary Material, which is appended after the references. The Python code to construct estimators and experiments is provided at https://github.com/cliusir123/VRforIM.
\section{Control variates for Independent Metropolis algorithms}
\subsection{A new estimator based on control variates}
We denote by $IM(P,\pi,q)$ the independent Metropolis algorithm defined on the state space $\X \subset \Real^d$ with transition kernel $P$ defined via (\ref{alpha}) and  two probability measures $\pi(\cdot)$ and $q(\cdot)$ with corresponding target and proposal densities $\pi(x)$  and  $q(y)$. 
The transition kernel $P$ that generates the Markov chain based on $IM(P,\pi,q)$ is expressed by 
\begin{equation}
P(x,dy) := \alpha(x,y)q(y)dy + \left( 1 - \int
\alpha(x,z)q(z)dz \right) \delta_x(dy),
\nonumber
\end{equation}
where $\delta_x$ is Dirac's measure centred at $x$. 
Then, for any function $G: \X \rightarrow \Real$ we have 
\begin{align}
PG(x) & := E_{\pi} [ G(X_1)] = E_{\pi} [ G(X_1)| X_0 = x] = \int P(x, dy)G(y)dy 
\nonumber \\
 & = \int \alpha(x, y)q(y)G(y)dy + \left(1- \int \alpha(x,z)q(z)dz \right) G(x) \nonumber \\
    & = G(x) + \int \alpha(x,y)(G(y) - G(x))q(y)dy. \label{PG}
\end{align}
\citet{henderson1997variance} observed that since the function $PG(x)-G(x)$ has zero expectation with respect to $\pi(\cdot)$, we can use it as a control variate to reduce the variance of $\mu_{n,IM}$ by constructing, using (\ref{PG}), the control-variate based estimator which is based on 
samples $X_1,X_2,\ldots,X_n$ and is given by 
\begin{equation}
    \mu_{n,IMCV,G} = \frac{1}{n}\sum_{i=1}^{n}\left\{F(X_i) +  \int \alpha(X_i,y)(G(y)-G(X_i))q(y)dy \right\}.
    \label{mu_CV}
\end{equation}
The optimum choice of $G$ is the solution, $\hat{F}(x)$, to the Poisson equation for $F$:
\begin{equation}
\label{PE}
    \int \alpha(x,y)(\hat{F}(y) - \hat{F}(x)) q(y) dy = -F(x) + \E_{\pi}[F], \text{ for every } x \sim \pi(\cdot),
\end{equation}
that  achieves zero variance for the estimator $\mu_{n,IMCV,{\hat F}}$. The solution $\hat{F}(x)$ of \eqref{PE} is translation invariant based on the Proposition 17.4.1 of \cite{Meyn2009MarkovEdition}. {Thus, in the following we denote by $\hat F(x)$ the centred
solution, normalised so that $\E_\pi[\hat F]=0$.}

A research strategy that was used by 
\cite{dellaportas2012control} for random scan Gibbs samplers and by \cite{alexopoulos2023variance} for random walk and Metropolis adjusted Langevin algorithms,  is to find an estimator based on some approximation $G \approx {\hat F}$.  We follow the same research avenue for $IM(P,\pi,q)$ by first proving the following Theorem:
{\begin{theorem} (Proof in Section~\ref{thmp1} of the Supplementary Material)
\label{thmsolution}
Assume that for a target density $\pi(x)$ there exists a sequence of
proposal densities $\{q_i(x)\}_{i=1}^\infty$ such that
$q_i(x)\rightarrow\pi(x)$ as $i\rightarrow\infty$ pointwise for all
$x\in\X$, and for each $q_i$ the chain $IM(P_i,\pi,q_i)$ satisfies
(A\ref{asu1}), (A\ref{asu2}) and (A\ref{asu3}) of
Section~\ref{sec:thm}, together with the uniform conditions
\begin{equation*}
    \sup_{i\geq1} w^\star_i \leq M < \infty ,
    \qquad
    \sup_{i\geq1}\,\E_{q_i}\big[F^2\big] < \infty .
\end{equation*}
Let $\hat F_i$ denote the solution of the Poisson equation of
$IM(P_i,\pi,q_i)$ for a function $F:\X\rightarrow\Real$, normalised
so that $\E_\pi[\hat F_i]=0$. Then
\[
    \int_\X\Big\{\hat F_i(x)-F(x)+\E_\pi[F]\Big\}^2\,\pi(x)\,dx
    \;\longrightarrow\;0
    \qquad\text{as } i\rightarrow\infty ,
\]
that is, $\hat F_i$ converges to the centred function
$F-\E_\pi[F]$ in mean square with respect to $\pi$.
\end{theorem}}

Theorem \ref{thmsolution} guarantees that as the proposal 
converges to the target density, the solution to the Poisson equation
converges to the function $F$, centred at its mean $\E_\pi[F]$. It is
straightforward to observe that at the limit
$\lim_{i\rightarrow\infty}q_i(x)=\pi(x)$ the acceptance ratio
$\alpha(x,y)$ in (\ref{PE}) becomes one and the solution of the
Poisson equation is just $\hat F(x)=F(x)-\E_\pi[F]$. Motivated by
this limiting case, for any independent Metropolis sampler (where
$q(x)\neq\pi(x)$) we propose to use the function $G(x)=F(x)$ in the
general control variate based estimator in \eqref{mu_CV}, which leads
to the estimator

\begin{equation}
    \mu_{n,IMCV,F} = \frac{1}{n}\sum_{i=1}^{n}\left\{F(X_i) + \int \alpha(X_i,y) (F(y)-F(X_i))q(y)dy \right\}.
    \label{mu_CV_F}
\end{equation}
Unlike the approach of \cite{alexopoulos2023variance}, the above estimator does not require any approximation of function $G$ because it simply assumes that $G(x)=F(x)$ is a good approximation to ${\hat F(x)}$.  However, (\ref{mu_CV_F}) still requires the evaluation of the  intractable integral 
$\int \alpha(X_i,y) (F(y)-F(X_i))q(y)dy$
for each sampled point $X_i$. 
Since the transition kernel of the $IM(P,\pi,q)$ algorithm required one sample $Y_i \sim q$ every time the chain was at state $X_i$, a Monte Carlo estimator of this integral 
can be based on the (sample size one!) estimator $\alpha(X_i,Y_i)(F(Y_i) - F(X_i))$, where 
all acceptance ratios $\alpha(X_i,Y_i)$ are available at no additional cost. This suggests that (\ref{mu_CV_F}) can be approximated by 
\begin{equation}
    \mu_{n,IMCV,F} \approx \frac{1}{n} \sum_{i = 1}^{n} \left\{ F(X_i) + \alpha(X_i,Y_i)(F(Y_i) - F(X_i)) \right\}. \label{mu_CV_F_one}
\end{equation} 
A final trick to achieve our final estimator is to use a further control variate to reduce the variance of the estimator (\ref{mu_CV_F_one}), by reducing the variance of the sample size one estimate $\alpha(X_i,Y_i)(F(Y_i) - F(X_i))$ of the integral  $\int \alpha(X_i,y) (F(y)-F(X_i))q(y)dy$. If the expected value 
$\E_q [F]$ is analytically available then an immediate choice of control variate is the zero-mean  random variable $F(Y)-\E_q [F]$ resulting in our proposed final estimator 
\begin{equation}\label{imcv}
\mu_{n,IMCV} = \frac{1}{n} \sum_{i = 1}^{n} \left\{ F(X_i) +  \left\{ \alpha(X_i,Y_i)(F(Y_i) - F(X_i))   
- ( F(Y_i) - \E_q [F] ) \right\} \right\}.
\end{equation}
If the expected value 
$\E_q [F]$ is analytically intractable, but there exists an approximation ${\tilde F}$ of $F$ such that $\E_q[{\tilde F}]$ is analytically available, then we can replace the control variate $F(y)-\E_q [F]$ by the control variate ${\tilde F}(y)-\E_q [{\tilde F}]$ and obtain 
\begin{equation} \label{approximcv}
\mu_{n,IMCV,\tilde F} = \frac{1}{n} \sum_{i = 1}^{n} \left\{ F(X_i) + \left\{ \alpha(X_i,Y_i)(F(Y_i) - F(X_i))   
- ({\tilde F}(Y_i)-\E_q [{\tilde F}]) \right\} \right\}. 
\end{equation}
It also possible to further 
improve the estimators \eqref{imcv} and 
\eqref{approximcv} 
by estimating a couple of regression coefficients in front of the control variates 
to further minimize variance, as
we detail in Section~\ref{sec:coef} of the Supplementary Material. 

{Notice that since the
measure $\pi(x)q(y)\alpha(x,y)\,dx\,dy$ has a density that is
symmetric in $(x,y)$ (Lemma~\ref{lem:aug} of the Supplementary Material), the term $\alpha(X,Y)\{F(Y)-F(X)\}$ has
exactly zero mean in stationarity, as does $F(Y)-\E_q[F]$, any linear combination of the two is therefore a valid
control variate. Rearranging the summand as
\[
    \varphi(x,y) = \E_q[F]+\{1-\alpha(x,y)\}\{F(x)-F(y)\},
\]
the estimator reads as the analytic expectation $\E_q[F]$ plus a
correction supported entirely on rejections, which vanishes as
$q\to\pi$. The Poisson equation perspective of this section explains why this
particular combination is the natural one. The term
$\alpha(X_i,Y_i)\{F(Y_i)-F(X_i)\}$ is the one-sample estimate of
$PF(X_i)-F(X_i)$, which by Theorem~\ref{thmsolution} approximates
the ideal correction $P\hat F-\hat F$ as $q\to\pi$, and the term
$F(Y_i)-\mathbb{E}_q[F]$ reduces the variance of this one-sample
estimate, the resulting coefficients $(1,-1)$ are the default values
of the regression coefficients of Section~\ref{sec:coef} of the Supplementary Material. The unbiasedness of the correction, however,
requires only the symmetry of $\pi(x)q(y)\alpha(x,y)$.
}

\subsection{Theoretical justifications}\label{sec:thm}
The asymptotic variance of $\mu_{n,IMCV}$ is given by
\begin{equation}\label{asydef}
    \sigma^2_{IMCV}
    = \lim_{n\rightarrow\infty}
    \mathrm{Var}\big[\sqrt{n}\,\mu_{n,IMCV}\big],
\end{equation}
and the question that arises is when $\sigma^2_{IMCV}$ is smaller
than $\sigma^2_F$. {Since the summands $\varphi(X_i,Y_i)$ of
$\mu_{n,IMCV}$ depend on the state $X_i$ and the proposal $Y_i$
jointly, the natural state process for their analysis is the pair
$Z_i:=(X_i,Y_i)$, which forms a Markov chain on the product space
$\X\times\X$: given $(X_i,Y_i)$, the next state is
$X_{i+1}=Y_i$ with probability $\alpha(X_i,Y_i)$ and
$X_{i+1}=X_i$ otherwise, while $Y_{i+1}\sim q$ is drawn afresh, so
that the transition kernel is
\begin{equation}\label{eq:augP}
    \widetilde P\big((x,y),\,d(x',y')\big)
    = \big[\alpha(x,y)\,\delta_y(dx')
    +\{1-\alpha(x,y)\}\,\delta_x(dx')\big]\otimes q(y')\,dy' ,
\end{equation} with invariant distribution $\pi\otimes q$, the product
measure with density $\pi(x)q(y)$, expectations and variances with
respect to it are denoted $\E_{\pi\otimes q}$ and
$\mathrm{Var}_{\pi\otimes q}$.} Let $W\geq1$ be a real valued function on $\X$; we
define a Banach space
$L^W_{\infty} := \big\{f:\X\rightarrow\Real \,\big\vert\,
\Vert f\Vert_W := \sup_{x\in\X}|f(x)|/W(x)<\infty\big\}$,
and also define
$L_{n}^{\pi} := \big\{f:\X\rightarrow\Real \,\big\vert\,
\int_\X|f(x)|^n\pi(x)\,dx<\infty\}$, $L_{n}^{q} := \big\{f:\X\rightarrow\Real \,\big\vert\,
\int_\X|f(x)|^n q(x)\,dx<\infty\}$ and $L_{n}^{\pi,q}:=L_{n}^{\pi} \cap L_{n}^q $ where $\pi$ and $q$ are the
corresponding densities of $IM(P,\pi,q)$, and
$w^\star:=\sup_{x\in\X}w_x$ where $w_x=\pi(x)/q(x)$. We state the
following assumptions.

\begin{assumption}[A1]\label{asu1}
The Markov chain produced by the $IM(P,\pi,q)$ sampler is $\psi$-irreducible and aperiodic with a unique invariant measure $\pi$.
\end{assumption}

\begin{assumption}[A2]\label{asu2}
    $F \in L^W_\infty$ and $W \in L_4^{\pi,q}$. 
\end{assumption}
\begin{assumption}[A3]\label{asu3}
 $1<w^{\star}<\infty$. 
\end{assumption}

Assumptions (A\ref{asu1}), (A\ref{asu2}) and (A\ref{asu3}), guarantee standard theoretical MCMC results as shown in the following proposition. Note that (A\ref{asu3}) also ensures uniform ergodicity. {The restriction $w^\star>1$ excludes only the trivial case $q=\pi$, in
which the acceptance probability is identically one and the estimator
is exact at every finite $n$.}
{\begin{proposition}(Proof in Section~\ref{propp1} of the Supplementary Material)\label{prop1}
    Under (A\ref{asu1}), (A\ref{asu2}) and (A\ref{asu3}):
    \begin{enumerate}
    \item $(Ergodicity)$ The chain $\{(X_i,Y_i)\}$ of states and
    proposals produced by the $IM(P,\pi,q)$ sampler is uniformly
    ergodic with invariant law $\pi\otimes q$; in particular, the
    chain $\{X_i\}$ is uniformly ergodic with invariant law $\pi$.
    \item $(LLN)$ $\mu_{n,IMCV}\rightarrow\E_{\pi}[F]$ as
    $n\rightarrow\infty$ a.s., from any initial distribution.
    \item $(CLT)$ $\sqrt{n}\left\{\mu_{n,IMCV}-\E_{\pi}[F]\right\}$
    converges in distribution to
    $\mathcal{N}(0,\sigma^2_{IMCV})$ as $n\rightarrow\infty$, and
    $\sigma^2_{IMCV}<\infty$.
    \item $(Positivity)$ The IM kernel is a positive operator on
    $L^\pi_2$: $\langle P^kf,f\rangle_\pi:=\int f P^kf \pi\geq0$ for every
    $f\in L^\pi_2$ and $k\geq1$
    \citep{Liu1996MetropolizedSampling,Gasemyr2006Spectrum}.
    Consequently, at stationarity, all autocovariances of $\{F(X_i)\}$ are nonnegative
    and $\sigma^2_F\leq\sigma^2_{IM}$.
\end{enumerate}
\end{proposition}
}
We now state our key result.  By denoting the Kullback–Leibler(KL) divergence as 
$
\mathbb{KL}(q(x) \vert \vert \pi(x)) := \int q(x) \log \frac{q(x)}{\pi(x)}dx ,
$
we  provide a general result that connects $\sigma^2_{IMCV}$ in (\ref{asydef}) with $\sigma^2_{F}$  in (\ref{true_var}) and $\sigma^2_{IM}$ in (\ref{sigmaIM}):

{\begin{theorem}\label{thmbound} (Proof in Section~\ref{thmp2} of the Supplementary Material)
Under (A\ref{asu1}), (A\ref{asu2}) and (A\ref{asu3}),
\[
    \sigma^2_{IMCV}
    \;\leq\;
    4\,(1+2w^\star)\,C_2^{1/2}\,\{1+\Lambda(w^\star)\}^{1/2}\,
    \mathbb{KL}(q\Vert\pi)^{1/2},
\]
where $C_2:=\max\{\E_\pi[F^4],\E_q[F^4]\}$ and
\[
    \Lambda(w^\star)
    := \frac{w^\star\log w^\star-w^\star+1}{w^\star-1-\log w^\star}
    \quad\text{for } w^\star>1,
    \qquad \Lambda(1):=1 .
\]
In particular, if
\[
    \mathbb{KL}(q\Vert\pi)
    \;\leq\;
    \left(\frac{\sigma^2_F}
    {4\,(1+2w^\star)\,C_2^{1/2}\,\{1+\Lambda(w^\star)\}^{1/2}}
    \right)^{2},
\]
then $\sigma^2_{IMCV}\leq\sigma^2_F\leq\sigma^2_{IM}$.
\end{theorem}
}

Theorem \ref{thmbound} states that under fairly general assumptions
on the function $F$ for which we need to calculate $\E_\pi[F]$, we
can construct an $IM(P,\pi,q)$ sampler so that the ergodic estimator
$\mu_{n,IMCV}$ of (\ref{imcv}) achieves smaller asymptotic variance
than both the i.i.d.\ (crude Monte Carlo) estimator
$\mu_{n,MC^\star}$ and the standard ergodic estimator
$\mu_{n,IM}$ of (\ref{mcmcest}). This immediately suggests a
methodological strategy for estimating $\E_\pi[F]$: use an
$IM(P,\pi,q)$ sampler with a proposal density $q$ that has two
properties: it approximates $\pi$ well, and $\E_q[F]$ is analytically
available.

\subsection{Numerical illustrations}
\label{illustrations}
In this section we illustrate Theorem \ref{thmbound} in simple synthetic data examples and a more challenging Bayesian model selection example. We examine the relative performance of two estimators by the variance reduction factor (VRF) which is defined as the ratio of their variances. For example, the measure of performance of our estimator $\mu_{n,IMCV}$ against the i.i.d.\ sampling or crude MC estimator $\mu_{n,MC^*}$
is based on  the ratio 
$\sigma^2_{F} / \sigma^2_{IMCV}$.
In our experiments we fix $n$ and produce $T$ independent replications $\{\mu_{n,IMCV}^{(i)}\}_{i=1}^{T}$ and  $\{\mu_{n,MC^*}^{(i)}\}_{i=1}^{T}$ that are used to estimate the variance reduction factor by
\begin{equation}\label{VRF}
    VRF = \frac{\sum_{i=1}^T\{ \mu_{n,MC^*}^{(i)}-T^{-1} \sum_{i=1}^T \mu_{n,MC^*}^{(i)}\}^2}{\sum_{i=1}^T\{ \mu_{n,IMCV}^{(i)}-T^{-1}\sum_{i=1}^T\mu_{n,IMCV}^{(i)}\}^2}.
\end{equation}

\subsubsection{Synthetic data examples}\label{sec:1d-n}
We first construct a synthetic data example in which we assume that the target distribution $\pi(x)$ is a standard normal distribution $\mathcal{N}(x|0,1)$ and the quantity of interest is the expected value of $F(x) = x$.  We choose two different proposals $q(y)$, a normal distribution $\mathcal{N}(y|0,\sigma^2)$ and a zero-mean student-t distribution with $\nu$ degrees of freedom  $t_{\nu}(y)$. 
{Independent Metropolis algorithms were initialised at a draw from the
proposal distribution, followed by a burn-in of 1000 iterations.}  $T=100$ independent replications of size $n=5000$ samples were used to obtain the estimators $\{\mu_{n,IMCV}^{(i)}\}_{i=1}^{T}$ and  
$\{\mu_{n,MC^*}^{(i)}\}_{i=1}^{T}$ that were used to construct the box-plots of Figures \ref{fig:1a} and \ref{fig:1c}.  It is evident that as the proposal densities approach the target, the sample variance of the estimators decreases.  The corresponding VRFs 
are depicted (shown in log space) in Figures \ref{fig:1b} and \ref{fig:1d}.  Notice that when the proposal density is far away from the target, 
$\mu_{n,IMCV}$ has larger variance than that of   
$\mu_{n,MC^*}$.
Furthermore, Figures \ref{fig:1b} and \ref{fig:1d} depict the 
theoretical lower bounds for the variance reduction factor \eqref{VRF} which are used in the proof of Theorem \ref{thmbound} in Section~\ref{thmp2} of the Supplementary Material. For these particular examples these bounds can be analytically derived, see  \eqref{1dnormalv} and \eqref{1dtv} in Section~\ref{ap:1dn} of the Supplementary Material.

\begin{figure}[ht]
\centering

\subfigure[]{
    \includegraphics[width=0.47\textwidth]{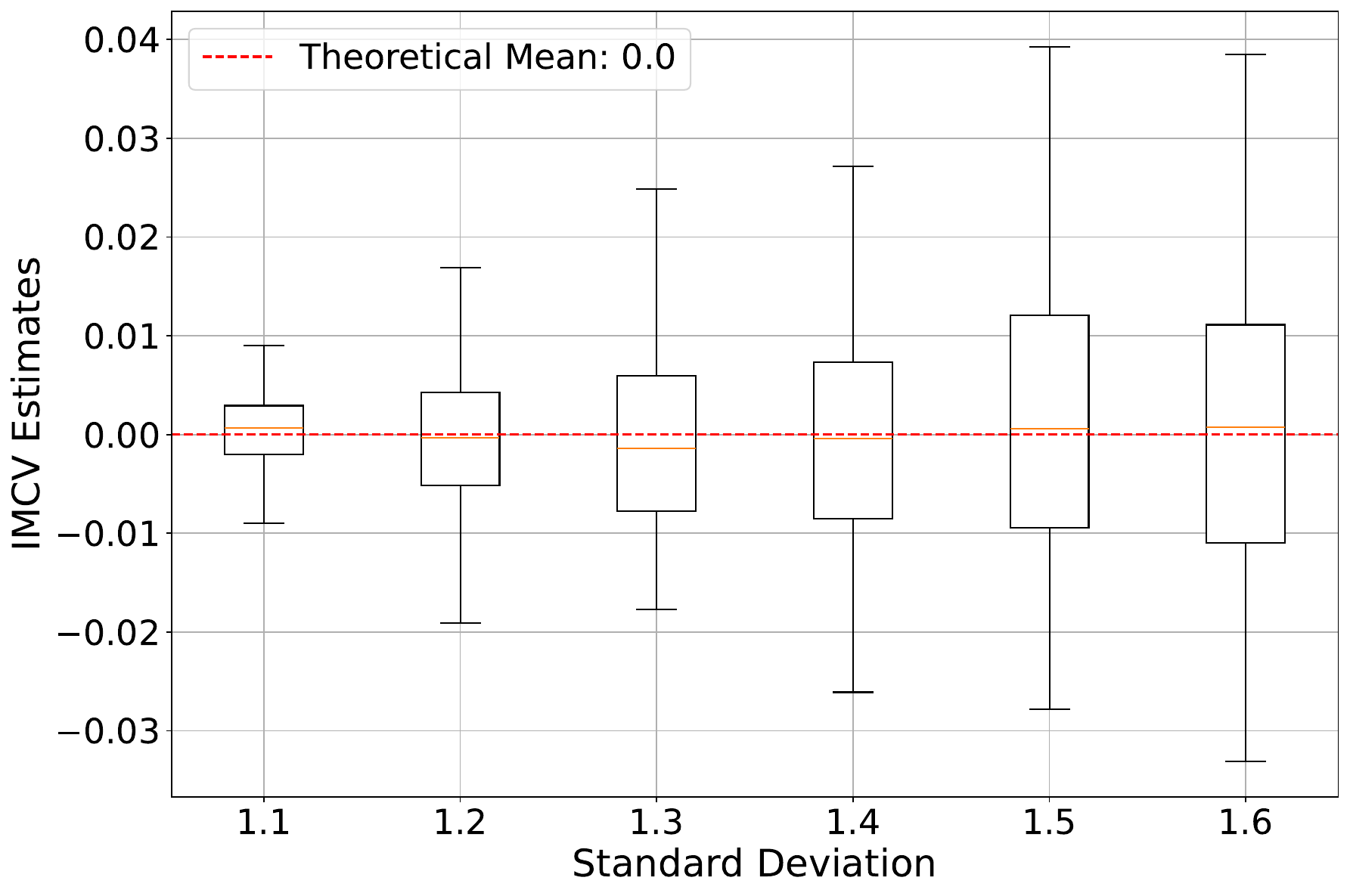}
    \label{fig:1a}
}
\subfigure[]{
    \includegraphics[width=0.47\textwidth]{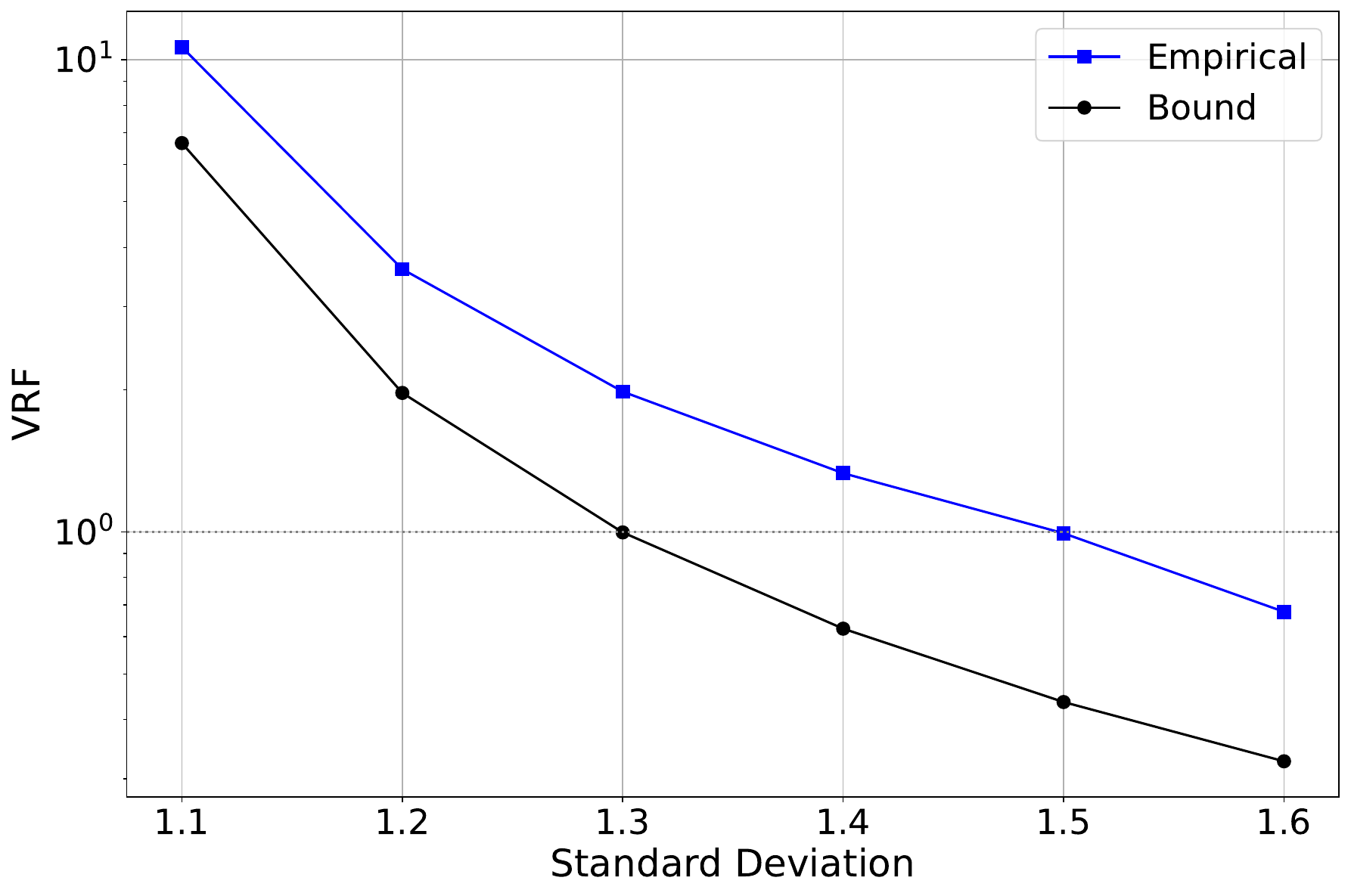}
    \label{fig:1b}
}

\subfigure[]{
    \includegraphics[width=0.47\textwidth]{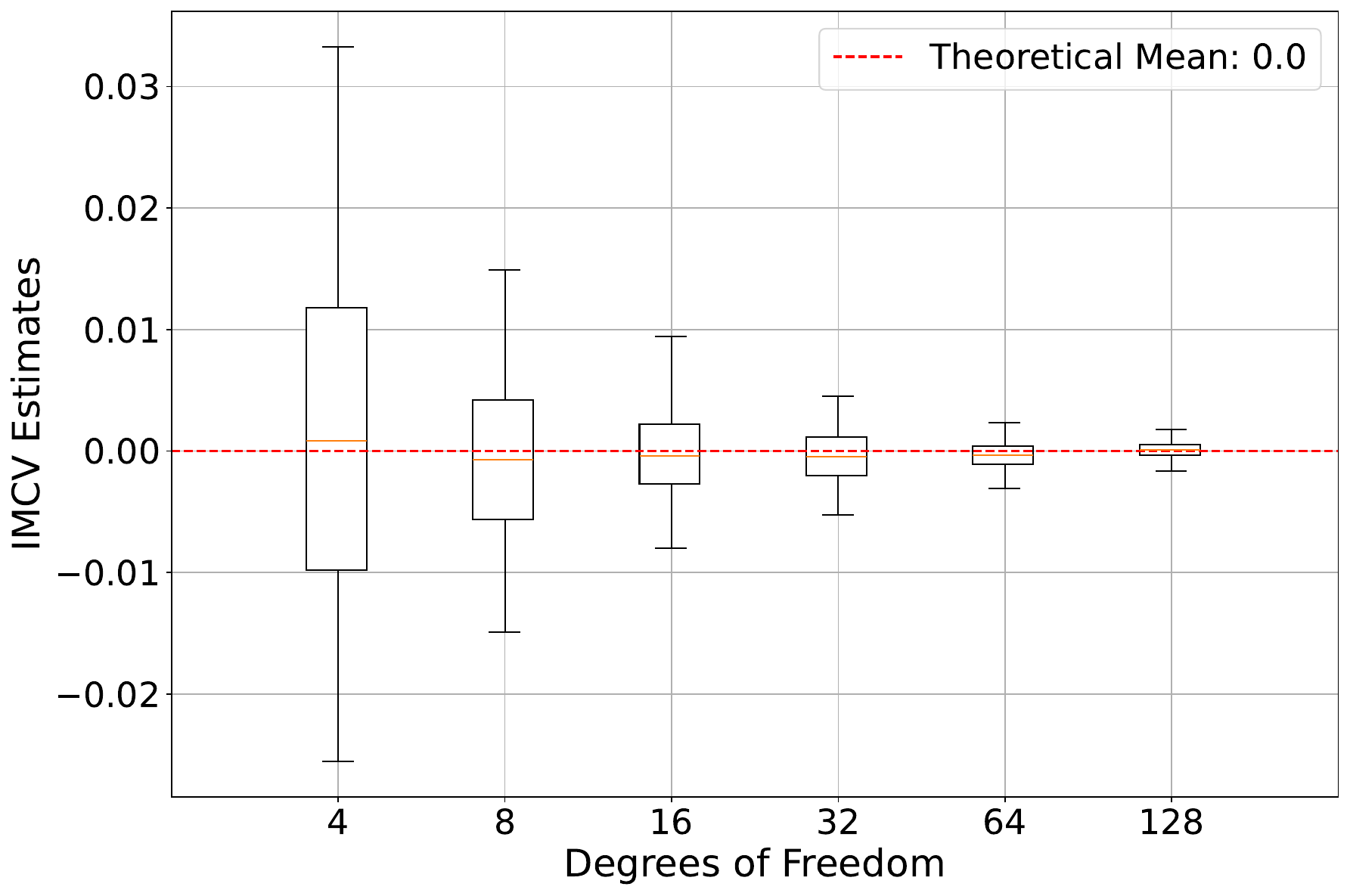}
    \label{fig:1c}
}
\subfigure[]{
    \includegraphics[width=0.47\textwidth]{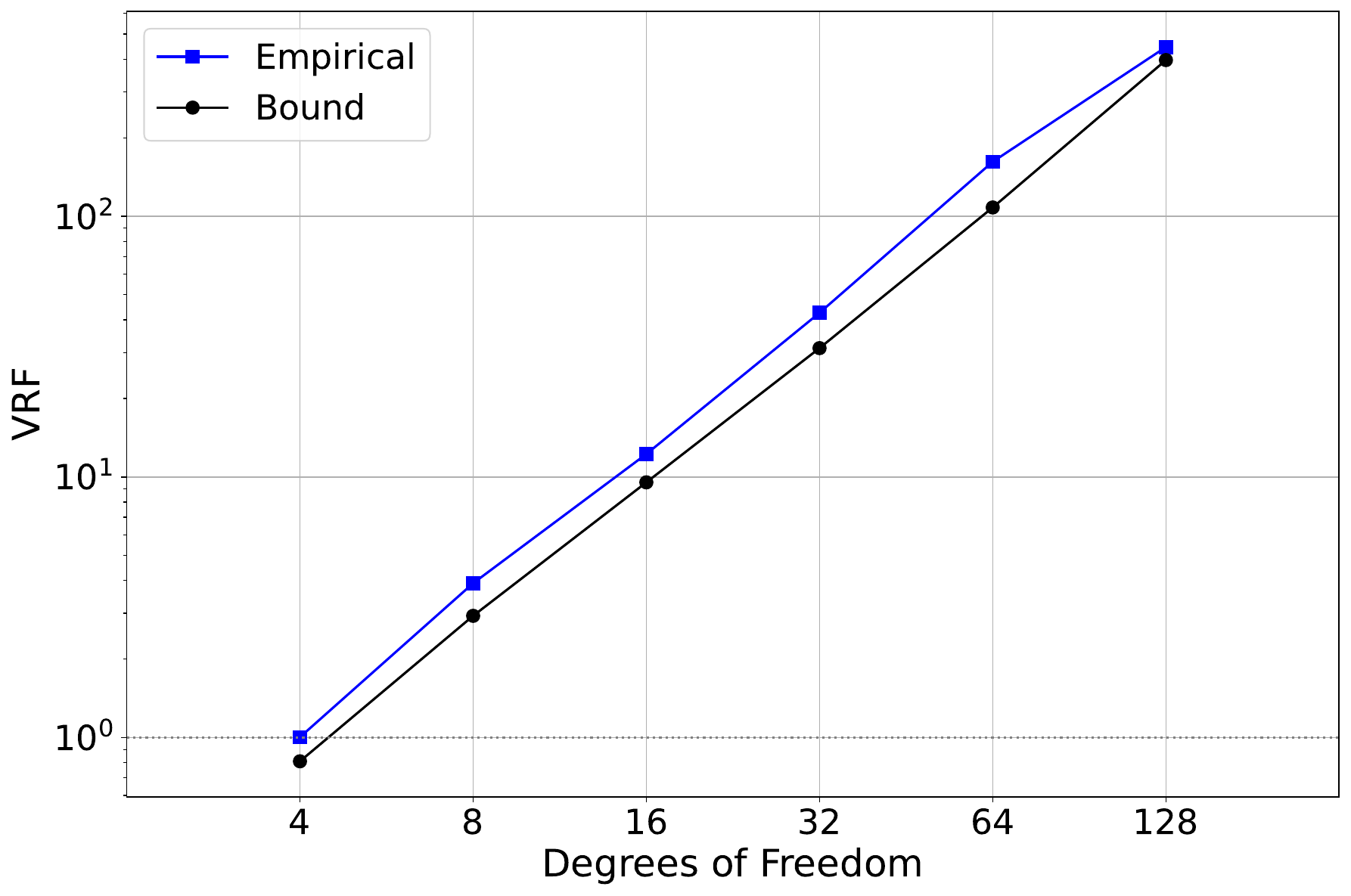}
    \label{fig:1d}
}
\caption{Comparison of $\mu_{n,IMCV}$ and $\mu_{n,MC^*}$ estimators. Top row: $\mathcal{N}(x|0,1)$ target and $\mathcal{N}(x|0,\sigma^2)$ proposal. Bottom row: $\mathcal{N}(x|0,1)$ target and $t_{\nu}(y)$ proposal. (a)-(c): Boxplots of $\mu_{n,IMCV}$ based on 100 replications for different values of $\sigma$ and $\nu$.  (b)-(d): The logarithm of VRFs and corresponding theoretical bounds for different values of $\sigma$ and $\nu$.}
\label{fig:1dexp}

\end{figure}

\subsubsection{Bayesian model selection in non-conjugate linear regression}\label{sec:le}
We consider a standard linear regression $y = \boldsymbol{X} \beta + \epsilon$ where  $y =  (y_1,y_2,\ldots,y_N)$ is a column vector of observations, $\boldsymbol{X}$ is an $N \times p$ design matrix, $\beta =  (\beta_1,\beta_2,\ldots,\beta_p)$ is a column vector of regression coefficients and $\epsilon$ is an error vector distributed as $\epsilon \sim N(0, \sigma^2 \boldsymbol{I_N})$.   The Bayesian approach to model selection  requires the calculation of posterior model probabilities $f(m | y) \propto f(m)f( y | m)$
for all models $m$ specified by all $\beta_m$, the parameter vectors with elements the $2^p$ subsets of the set of $p$ elements of $\beta$, and the corresponding design matrices $\boldsymbol{X}_m$.  In turn, this requires a specification of prior model probabilities $f(m)$ and the calculation of the marginal likelihood $f( y | m)$ of model $m$ which is given by
\begin{equation}
    f(y | m) = \int f(y | m, \beta_m) 
    f(\beta_m | m) d \beta_m \label{Marginal likelihood}
\end{equation}
where $f(y | m, \beta_m)$ and $f(\beta_m | m)$
are the likelihood functions and parameter prior densities of model $m$.  
It is well-known that the choice of the prior $f(\beta_m | m)$ is of paramount importance in calculating posterior model probabilities and a careful specification is both necessary and an issue that has attracted a lot of interest in the literature \citep[see, for example,][]{dellaportas2012joint}.  Perhaps the most popular and widely accepted default choice is the mixture of g-priors of \cite{liang2008mixtures}. An example of these priors that has been suggested by \cite{liang2008mixtures} is given by 
\begin{equation}
    \beta_m | g, m \sim \mathcal{N}(0,
    g(\boldsymbol{X}_m^\top \boldsymbol{X}_m)^{-1}),~~
    p(g) = (1+g)^{-2}, ~~g>0.
\label{prior}
\end{equation}
This and all other choices of mixtures of $g$-priors result in making integration of (\ref{Marginal likelihood}) intractable.  This is a typical situation in Bayesian model choice literature where huge efforts have been made to propose ways to approximate such intractable marginal likelihoods.

It has often been noted \citep[see, for example,][]{newton1994approximate} that an obvious solution would be to adopt a crude MC estimator by sampling from the prior and estimating (\ref{Marginal likelihood}) as the expected likelihood function.  However, since the posterior is often concentrated relative to the prior, Monte Carlo estimators will be very inefficient; see \cite{mcculloch1992bayes} for a discussion. Although many strategies for increasing the efficiency of the crude MC are available, it seems that there is nothing available to aid the estimation of the marginal likelihood of a linear regression model with the possibly most popular default prior density.  Note that we cannot, at least readily, find functions of $\beta_m$ such that their analytical expectations with respect to the prior (\ref{prior}) are available, so is not easy to construct control variables to reduce the variance of  the crude MC estimator.

We test our $\mu_{n,IMCV}$ estimator by estimating (\ref{Marginal likelihood}) in synthetic data where there is a clear prior-posterior conflict.  We generate $N=50$ 
data points with predictors $X_i \sim \mathcal{N}(0,1)$, $i = 1,\ldots,4$ and response $Y \sim \mathcal{N}(4X_3 + 4X_4, 2.5^2)$.  The $IM(P,\pi,q)$ algorithm  with
$\pi$ being the prior (\ref{prior}) requires a proposal density $q$. Of course, in this setting one could integrate out the beta parameters analytically; however, this is not feasible in other scenarios such as logistic regression. We use this simple example here to illustrate our method in comparison with highly accurate estimates of Bayes factors. Our estimator  $\mu_{n,IMCV}$ performs well when this proposal is both close to the target density (\ref{prior}) and the expectation $\E_q[f(y | m,\beta_m)]$ is analytically available.  We choose to use a proposal of the form 
\begin{equation}\label{mixture}
    q(\beta_m) = \sum_{i=1}^K w_i\mathcal{N}(\beta_m | 0, g_i (\boldsymbol{X}_m^\top \boldsymbol{X}_m)^{-1}), \text{ such that }\sum_{i=1}^K w_i=1,~ 0<w_i<1.
\end{equation}
For fixed $K$ and model $m$ the parameters $\left\{w_i, g_i\right\}_{i=1}^K$ can be estimated by an Expectation-Maximisation (EM) algorithm as follows. Let $\boldsymbol{I}_{(m)}$ be an identity matrix with dimension equal to the dimension of $\beta_m$. We  fit a $\sum_{i=1}^{K}w_i \mathcal{N}(0,g_{i} \boldsymbol{I}_{(m)})$ model based on 1000 data points generated from the generic prior \eqref{prior} with $\boldsymbol{X}_m = \boldsymbol{I}_{(m)}$
and we obtain estimates ${\hat{w}}_i$ and ${\hat{g}}_i$. Note that this algorithm is fitted only once and converges rapidly because all normal components have zero means.
Then we sample $\beta^{'}$ from the density 
\begin{equation}\label{identitypro}
    q(\beta')=\sum_{i=1}^{K}{\hat{w}}_i \mathcal{N}(\beta^{'} | 0,{\hat{g}}_{i} \boldsymbol{I}_{(m)})
\end{equation}
and we exploit the invariance of the KL divergence under linear transformation to obtain $\beta_m = (\boldsymbol{X}_m^\top \boldsymbol{X}_m)^{-\frac{1}{2}}\beta^{'}$ as samples generated from the density \eqref{mixture} that matches very well \eqref{prior}.

We chose $K=4$ by visual inspecting a Q-Q plot of the samples of 
(\ref{prior}) with $\boldsymbol{X}_m = \boldsymbol{I}_{(m)}$ and (\ref{identitypro}).  The $IM(P,\pi,q)$ algorithm ran for $1000$ iterations and compared with the corresponding crude MC through VRFs based on $100$ independent replications of the experiment.  Results for all $15$ models are shown in Table \ref{tab:Marginal} where it is evident that our estimator approximates well the marginal likelihood in all models.  For comparative purposes we have included precise estimations of the marginal likelihoods obtained by first integrating $\beta_m$ and then performing  numerical integration over $g$. Notice that the difference of marginal likelihoods in $\log_{10}$-scale is used in \cite{kass1995bayes} as a way to interpret the weighted evidence against a null hypothesis with Bayes factors. 
It is also clear that as the parameter dimension increases, the variance of the crude MC estimator and consequently the size of the variance reduction factor increases.

\begin{table}
\centering
\caption{Marginal Likelihood estimation for different models $X_m$ where $m$ is specified by which covariates are in the model. Results are reported as negative log-estimates and compared with a numerical integration estimation. VRF denotes the variance reduction factor.}
\label{tab:Marginal}
\begin{tabular}{@{}lcccc@{}}
\hline
Model & $-\log_{10}(\mu_{1000,IMCV})$ & $-\log_{10}(\mu_{1000,MC^*})$ & Numerical integration & VRF$^{a}$ \\
\hline
$X_{1}$    & 93.19 & 93.19 & 93.19 & 4.29 \\
$X_{2}$    & 93.77 & 93.77 & 93.77 & 3.74 \\
$X_{3}$    & 77.23 & 77.08 & 77.16 & 8.71 \\
$X_{4}$    & 72.78 & 72.80 & 72.77 & 6.30 \\
$X_{12}$   & 93.35 & 93.34 & 93.35 & 5.01 \\
$X_{13}$   & 77.95 & 77.81 & 77.94 & 8.10 \\
$X_{14}$   & 73.52 & 73.66 & 73.51 & 4.02 \\
$X_{23}$   & 77.82 & 77.66 & 77.83 & 5.39 \\
$X_{24}$   & 73.57 & 73.36 & 73.55 & 16.74 \\
$X_{34}$   & 55.31 & 55.08 & 55.29 & 37.25 \\
$X_{123}$  & 78.53 & 78.30 & 78.53 & 30.34 \\
$X_{124}$  & 74.20 & 73.97 & 74.23 & 25.36 \\
$X_{134}$  & 55.39 & 55.49 & 55.34 & 49.12 \\
$X_{234}$  & 56.14 & 55.56 & 56.00 & 106.15 \\
$X_{1234}$ & 55.91 & 57.71 & 55.93 & 4763.42 \\ 
\hline
\end{tabular}
\par\smallskip
{\footnotesize $^{a}$The Variance Reduction Factor (VRF) is calculated relative to the standard i.i.d. sampling variance.}
\end{table}

\section{Adaptive independent Metropolis}
\label{sec: aim}

In general, we consider the sequence of independent Metropolis adaptations $\{IM(P_i,\pi,q_{\theta_i})\}_{i=1}^{\infty}$ identified by the proposal densities $\{q_{\theta_i}\}_{i=1}^{\infty}$, indexed by the corresponding parameter vectors $\theta_i \in \Theta$. For each step $i$, we have an update function $H_{\theta_i}(x)$ and its step-size $\xi_i$. Then, the adaptive independent Metropolis  can be implemented as shown in Algorithm \ref{alg:adp}.

\begin{algorithm}[ht]
\raggedright
\caption{General adaptive independent Metropolis}\label{alg:adp}
\textbf{Input}: Parameterised proposal distribution $q_\theta$, (intractable) target distribution $\pi$, update function $H_\theta(x)$, stepsize $\xi$, total number of adaptation iterations $T$.
\textbf{Initialisation} Set $i \leftarrow 1$ and initialise \{$X$, $q_{\theta}$, $H_{\theta}(x)$, $\xi$\} by \{$X_0$, $q_{\theta_0}$, $H_{\theta_0}(X_0)$, $\xi_0$\}.
\begin{algorithmic}[1]
\While{$i \leq T$}
 \State Save sample $(X_i, Y_i)$ from independent Metropolis $IM(P_i,\pi,q_{\theta_i})$.
 \State Adaptation update $\theta_{i+1} \leftarrow \theta_i - \xi_i H_{\theta_i}(Y_i)$.
 \State Set $i \leftarrow i+1$.
\EndWhile
\end{algorithmic}
\textbf{Return}: Proposals and samples $\{q_{\theta_i}, X_i ,Y_i\}_{i=1}^{T}$.
\end{algorithm}

The ergodicity of the adaptive independent Metropolis algorithm is guaranteed by assumptions on the specific procedure of adaptation. Throughout this section we will be assuming that all independent Metropolis sampling chains have a state space $\X$ and satisfy the following Assumption:
{\begin{assumption}[A4]\label{asu4}
There is a deterministic $\theta^\star\in\Theta$ to which the
adaptation converges, uniformly in the number of samples used per
adaptation step: writing $\theta^{(B)}_i$ for the parameter after $i$
adaptation steps when each step uses $B\geq1$ samples, so that $B=1$
is Algorithm~1, for every $\epsilon>0$
\[
    \lim_{i\rightarrow\infty}\ \sup_{B\geq1}\
    \Pr\big(\Vert\theta^{(B)}_i-\theta^\star\Vert>\epsilon\big)\;=\;0 .
\]
Elsewhere we write $\theta_i$ for $\theta^{(B)}_i$, the sample size per
step being clear from the context.
\end{assumption}
}
While (A\ref{asu4}) can be difficult to verify in practice, it plays a crucial role in optimisation theory in settings involving non-convex loss functions where various conditions and convergence results have been studied \citep{eon1998online,mertikopoulos2020almost}. In empirical implementations, the use of a learning rate schedule has been found effective in promoting convergence \citep[e.g.][]{Andrieu:2008,Roberts2009ExamplesMCMC,brofos2022adaptation}.

Moreover, in order to explore the properties of the adaptive MCMC, we
need to rewrite (A\ref{asu2}) and (A\ref{asu3}) in their adaptation
version:
\begin{assumption}[A5]\label{asu21}
    The Markov chain produced by the $IM(P_i,\pi,q_{\theta_i})$
    sampler is $\psi$-irreducible and aperiodic with unique invariant
    measure $\pi$ for all $i=1,2,\ldots$.
\end{assumption}
\begin{assumption}[A6]\label{asu31}
    There exists a constant $M>1$, such that
    $w^{\star}_i:=\sup_{x\in\X}\,\pi(x)/q_{\theta_i}(x)\leq M$ for
    all $\theta_i\in\Theta$.
\end{assumption}
{Assumption (A\ref{asu31}) is the adaptation version of
(A\ref{asu3}). It requires the importance weights to be bounded
uniformly over the proposal family, so that all chains in the family
contract at a common geometric rate. The adaptation version of
(A\ref{asu2}) requires, in the same spirit, that the moment bounds on
$F$ hold uniformly over the family. This is conveniently expressed
through a bounding function $W$, since $F\in L^W_\infty$ means that
$|F|\leq\Vert F\Vert_W\,W$ pointwise.
\begin{assumption}[A7]\label{asu2p}
    There exists $W\geq1$ with $F\in L^W_\infty$ such that
    $\E_\pi[W^4]<\infty$ and
    $\sup_{\theta\in\Theta}\int W(x)^4\,q_\theta(x)\,dx<\infty$.
\end{assumption}
Assumptions (A\ref{asu21})--(A\ref{asu2p}) suffice for the law of
large numbers of the adaptive estimator and for
Corollary~\ref{cor1}. The central limit theorem additionally
requires the proposal family to be continuous at the limiting
parameter $\theta^\star$ of (A\ref{asu4}), and the adaptation to
settle down along the run. Both requirements are collected in the
following assumption.
\begin{assumption}[A8]\label{asu8}
The maps $\theta\mapsto q_\theta(x)$ are continuous at $\theta^\star$
for almost every $x$ and the map $(\theta,x)\mapsto q_\theta(x)$ is jointly measurable. Also, the adaptation is diminishing in the sense
of \citet{Roberts2007CouplingMCMC}, that is,
$\sup_x\big\Vert P_{\theta_{i+1}}(x,\cdot)-P_{\theta_i}(x,\cdot)\big\Vert_{TV}
\rightarrow0 \quad\text{in probability as } i\rightarrow\infty$.
\end{assumption}
The first condition of (A\ref{asu8}) controls the regularity of the
family $\{q_\theta\}$. Together with the moment bounds of
(A\ref{asu2p}), it is what makes the asymptotic variance
$\sigma^2_{IMCV}(\theta)$ continuous at $\theta^\star$. The
diminishing condition concerns the adaptation scheme rather than the
proposal family. It holds for stochastic-gradient updates with
decreasing step sizes and bounded gradient estimates, by the argument
of \citet[Theorem~4.1]{brofos2022adaptation}.}

\subsection{Adaptation based on KL divergence}
 Theorem \ref{thmbound} and the numerical illustrations of the previous section indicate that the key element of a good estimator of an $IM(P,\pi,q)$ sampler is the ability to obtain a good proposal density $q$ that is as close as possible, in terms of KL divergence, to the target $\pi$.  Therefore, in this section we define the update function $H_{\theta}$ in
 Algorithm \ref{alg:adp} 
 as a gradient direction towards minimising the KL divergence $\mathbb{KL}(q_{\theta}(x) \vert \vert \pi(x))$. To further justify this choice, we present the following result from  Theorem \ref{thmbound}.

\begin{corollary}(Proof in Section~\ref{corp1} of the Supplementary Material)\label{cor1}
Suppose there exists a sequence of constants
$\{\epsilon_i\}_{i=1}^{\infty}$ and a sequence of proposals
$\{q_{\theta_i}\}_{i=1}^{\infty}$ such that
$\mathbb{KL}(q_{\theta_i}\Vert\pi)\leq\epsilon_i$ for every $i$, and
$\lim_{i\rightarrow\infty}\epsilon_i=0$. Under (A\ref{asu21}),
(A\ref{asu31}) and (A\ref{asu2p}), define $\sigma^2_{IMCV,(i)}$ to be
the asymptotic variance of \eqref{imcv} with the
$IM(P_i,\pi,q_{\theta_i})$ sampler. Then
$\lim_{i\rightarrow\infty}\sigma^2_{IMCV,(i)}=0$.
\end{corollary}

Corollary \ref{cor1} indicates that if we obtain a sequence of adaptations 
$\{IM(P_i,\pi,q_{\theta_i})\}_{i=1}^{\infty}$  such that the proposal densities approach the target density as $i$ increases, we can achieve zero variance estimators. The requirement is that the adaptation satisfies 
$$
\mathbb{KL}(q_{\theta_i}(x) \vert \vert \pi(x)) = \int q_{\theta_i}(x) \log \frac{q_{\theta_i}(x)}{\pi(x)}dx \leq \epsilon_i
$$
for a sequence of $\{\epsilon_i \}_{i=1}^{\infty} $ that converges to zero. 

{Whether such a sequence exists depends on the proposal family. For a fixed
parametric family, the limit of the adaptation is the KL projection of
$\pi$ onto the family, and $\mathrm{KL}(q_{\theta_i}\|\pi)\to 0$ requires
the family to be flexible enough, for instance mixtures with a growing number
of components or normalizing flows
\citep{Gabrie2022AdaptiveFlows,brofos2022adaptation}, for which universal
approximation results are available. In practice the adequacy of the
family for the target at hand can be monitored online: the acceptance
rate and the empirical weight bound
$\hat w^\star_n := \max_{t\le n}\, \pi(Y_t)/q_{\theta_{i(t)}}(Y_t)$
(computable up to the normalising constant of $\pi$) measure the quality
of the fitted proposal, and a persistently low acceptance rate or a
$\hat w^\star_n$ that keeps growing with $n$ indicates that (A6) is at risk
and the family should be enlarged.
}

In practice, a way to obtain 
the sequence of the adapted proposal distributions 
$q_{\theta_i}$ is by applying 
gradient descent updates 
on the parameters in order to minimise the above KL divergence. The gradient descent update takes the form 
\begin{align}\label{KL-update}
&\theta_{i+1} = \theta_i - \xi_i \nabla_{\theta_i} \mathbb{KL}(q_{\theta_i}(x)||\pi(x)), \ \text{where } \nabla_{\theta_i} \E(q_{\theta_i}||\pi) = \nabla_{\theta_i}\E_{q_{\theta_i}} \log \frac{q_{\theta_i}(x)}{\pi(x)}
\end{align}
and $\xi_i>0$ is a step size parameter. This exact optimisation procedure 
is not feasible since the gradient is not available in  closed-form, and instead we can apply stochastic optimisation \citep{robbinsmonro51} where we use  
unbiased estimates of the gradients. More precisely, we can follow the recent literature in machine learning where such stochastic gradient optimisation methods 
are widely used for 
KL-based variational inference   \citep{Ranganath14, Titsias2014DoublyInference, Kingma2014, Kucukelbir17}.
We will assume that the independent Metropolis proposal 
$q_\theta(x)$ is reparametrisable 
from a simpler distribution $p(z)$, which enables the use of efficient 
reparametrisation gradient
methods \citep{Titsias2014DoublyInference, Rezende2014, Kingma2014, Gianniotis2015, Salimans_2013, Roeder17, Tan18}. For instance, for  the standard multivariate Gaussian proposal $q_\theta(x) = \mathcal{N}(\mu, \boldsymbol{\Sigma}) = \mathcal{N}(\mu, \boldsymbol{L} \boldsymbol{L}^\top)$ where $\theta = (\mu, \boldsymbol{L})$ and $\boldsymbol{L}$ is the Cholesky factor of the covariance matrix, we can reparametrise $x \sim \mathcal{N}(\mu, \boldsymbol{L} \boldsymbol{L}^\top)$ as $\mu + \boldsymbol{L} z, z \sim \mathcal{N}(0,\boldsymbol{I}_d)$, and then consider
as the update function $H_{\theta}(Y)$ in Algorithm \ref{alg:adp}
the unbiased gradient 
\begin{equation}
H_{\theta}(Y) = \nabla_{\theta}\E_{q_{\theta} (x)} [\log q_{\theta}(x)]
- \nabla_\theta \log \pi( Y=\mu + \bs{L} z), \ \  z \sim \mathcal{N}(0, \boldsymbol{I}_d),
\label{eq:unbiasedgradient}
\end{equation}
where $\E_{q_{\theta} (x)} [\log q_{\theta}(
x)]$ is the negative entropy of the Gaussian distribution.
The simplest form of adaptation is to apply stochastic optimisation updates at each iteration of independent Metropolis, i.e.\ 
 similarly to standard adaptive MCMC   \citep{Haario2001AnAlgorithm, Roberts2009ExamplesMCMC, Andrieu:2008}. To perform the adaptation step at the $i$-th sampling iteration we use the proposed state $Y_i = \mu_i + \bs{L}_i z_i \sim q_{\theta_i}$ to evaluate the 
unbiased gradient in  \eqref{eq:unbiasedgradient}, and then  perform a gradient step to obtain the new parameters $\theta_{i+1}$ by applying the update step in Algorithm \ref{alg:adp} (line 3). 

Further implementation details of the above unbiased gradient estimation procedure towards minimising the KL divergence are given in Section~\ref{sec:adp-algo} of the Supplementary Material.  

\subsection{Theory and implementation of the adaptive independent Metropolis}

\subsubsection{Sample collection after adaptation}
\label{subopt}
A trivial sub-optimal way to implement the general adaptation  Algorithm \ref{alg:adp} is to adapt $q_{\theta}$ until it reaches a final proposal density $q_{\theta_{T}}$ and ignore the samples generated by the adaptations  $\{ IM(P_i,\pi,q_{\theta_i}) \}_{i=1}^{T-1}$.   Then, the ergodic estimator (\ref{imcv}) can be evaluated by using the samples from  the non-adaptive algorithm $IM(P_T,\pi,q_{\theta_{T}})$ and the Proposition \ref{prop1} provides the necessary theoretical guarantees for such scheme. See Section \ref{sec:adaptiveNumerical} for the experimental results.

\subsubsection{Sample collection during adaptation}\label{sec:opt} 

Clearly the previous sampling scheme can be inefficient since it produces an estimator that ignores all samples from the adaptation phase. In this section, we wish to construct estimators that make use of all samples, and to that end we will study the properties of the whole sequence of the independent Metropolis adaptations $\{IM(P_i,\pi,q_{\theta_i})\}_{i=1}^{\infty}$
and the corresponding sequence of generated  samples. 

{Laws of large numbers (LLN) and central limit theorems (CLT) have been
proved for various cases of $\{IM(P_i,\pi,q_{\theta_i})\}_{i=1}^{\infty}$
sequences, see \citet{Andrieu2006OnAlgorithms},
\citet{Roberts2007CouplingMCMC} and \citet{Fort2014CLTAdaptive}.
However, some of the conditions required for these proofs are restrictive or hard to
verify in practice, see for example \citet{Andrieu2006OnAlgorithms}.
\citet{Roberts2007CouplingMCMC} prove a LLN that assumes that the
function $F$ is bounded, while the CLT of
\citet[Theorem~2.2]{Fort2014CLTAdaptive} is obtained under drift and
small-set conditions together with a regularity condition on the
solutions of the Poisson equations of consecutive kernels.

We propose an adaptive $\{IM(P_i,\pi,q_{\theta_i})\}_{i=1}^{\infty}$
algorithm in which adaptation takes place every $B$ iterations, where
$B$ is the number of samples per adaptation step appearing in (A4).
This has three key advantages. First, the update (\ref{KL-update}) is
based on $B$ samples so the estimate \eqref{eq:unbiasedgradient} is
based on $B$ samplers. Second, our estimator is now based on $\ell$
batches of size $B$, $n=\ell B$, and this allows us to prove a LLN
relaxing the assumption of boundedness of $F$. Third, holding the
proposal fixed within a batch simplifies the martingale approximation
underlying the CLT, which we obtain under the uniform minorisation
implied by (A6) and the condition $\ell/B\rightarrow0$ on the batch
length, in place of the drift and regularity conditions of
\citet{Fort2014CLTAdaptive}.
}

Assume that we collect samples $\left \{ \left\{(X_{ij}, Y_{ij}) \right\}_{j=1}^B \right\}_{i=1}^\ell$ 
where $X_{ij}$ denote the accepted samples and $Y_{ij}$ denote the corresponding proposal values from the proposal densities $\{ q_{\theta_{i}} \}_{i=1}^\ell$. Our estimator is  
\begin{equation}\label{batchest}
\scalebox{0.85}{$
    \mu_{\ell,B,IMCV} = \frac{1}{\ell B}\sum_{i = 1}^{\ell}  \sum_{j = 1}^{B} \left\{F(X_{ij}) + \left[\alpha(X_{ij}, Y_{ij})(F(Y_{ij}) -  F(X_{ij})) -
    (F(Y_{ij}) - \E_{q_{\theta_{i}}}[F(y)])\right]  \right\}.
$}
\end{equation}
A more elaborate version of our estimator with a couple of coefficients to minimise the variance can be found in Section~\ref{sec:coef} of the Supplementary Material.

\begin{algorithm}[ht]
\raggedright
\caption{Batch Adaptive Independent Metropolis with control variate}\label{alg:BAIM}
\textbf{Inputs}: Parameterised proposal distribution $q_\theta$, (intractable) target distribution $\pi$, update function $H_\theta(x)$, stepsize $\xi$, batch size $B$, the number of batches $\ell$, objective function $F$.
\\
\textbf{Initialisation} Set $i \leftarrow 1, j \leftarrow 1$ and initialise \{$X$, $q_{\theta}$, $H_{\theta}(x)$, $\xi$\} by \{$X_0$, $q_{\theta_0}$, $H_{\theta_0}(X_0)$, $\xi_0$\}.
\begin{algorithmic}[1]
        \While{$i \leq \ell$}
        \While{$j \leq B$}
        \State Save $\{\alpha(X_{ij},Y_{ij}), X_{ij}, Y_{ij}\}$ from independent Metropolis $IM(P_i,\pi,q_{\theta_i})$.
        \State Set $j \leftarrow j + 1$.
        \EndWhile
        \State Adaptation update $\theta_{i+1} \leftarrow \theta_i - \xi_i \sum_{j=1}^{B}H_{\theta_i}(Y_{ij})/B$.
        \State Calculate analytical result for the batch $E_{q_{\theta_i}}[F]$.
        \State Update $i \leftarrow i + 1, j \leftarrow 1.$
        \EndWhile

    \item Calculate the estimator \eqref{batchest}.
\end{algorithmic}
\textbf{Returns}: Estimator \eqref{batchest}.
\end{algorithm} 

The resulting Algorithm \ref{alg:BAIM} describes our adaptive independent Metropolis scheme and the following Theorem gives rigorous theoretical arguments.

{\begin{theorem}(Proof in Section~\ref{thmp3} of the Supplementary Material)\label{thmconv}
    \begin{enumerate}
        \item $(Ergodicity)$ Under (A\ref{asu31}) and (A\ref{asu8}), the Markov chain generated by
        Algorithm~\ref{alg:BAIM} is ergodic in the sense of \citet{Roberts2007CouplingMCMC}.
        \item $(LLN)$ Under (A\ref{asu21}), (A\ref{asu31}) and
    (A\ref{asu2p}),
        $\mu_{\ell,B,IMCV}\rightarrow\E_{\pi}[F]$
        in probability as $B\rightarrow\infty$, for any $\ell \geq 1$.
        \item $(CLT)$ Under (A\ref{asu4}), (A\ref{asu21}), (A\ref{asu31}),
    (A\ref{asu2p}) and (A\ref{asu8}), if in addition $\ell\rightarrow\infty$
    and $\ell/B\rightarrow0$, then, with $n=\ell B$,
        \[
            \sqrt{n}\,\big\{\mu_{\ell,B,IMCV}-\E_{\pi}[F]\big\}
            \;\xrightarrow{\;d\;}\;
            \mathcal{N}\big(0,\,\sigma^2_{IMCV}(\theta^\star)\big),
        \]
        where $\sigma^2_{IMCV}(\theta^\star)$ is the asymptotic
        variance \eqref{asydef} of the IMCV estimator with the fixed
        proposal $q_{\theta^\star}$, and $\theta^\star$ is the limiting
        parameter of (A\ref{asu4}).
    \end{enumerate}
\end{theorem}
}

\subsection{Numerical illustrations}\label{sec:adaptiveNumerical}
We evaluate the performance of the adaptive independent Metropolis algorithm with real data applications.  Given the estimator sample size $n$, our interest lies in comparing: (1) given a converged adaptation proposal density $q_{\theta_T}$, the variance of the standard IM estimator $\mu_{n,IM}$ in \eqref{mcmcest}, IMCV estimator $\mu_{n,IMCV}$ in \eqref{imcv}, IMCV estimator with coefficients $\mu_{n,IMCV}^{(\hat{c}_1,\hat{c}_2)}$ in \eqref{imcv_coef} of the Supplementary Material, Rao-Blackwellisation estimator $\mu_{n,RB}$ in \eqref{rbest} and IM with coupling estimator $\mu_{n,CIM}^{(\hat{c})}$ in \eqref{est:cimcoeff} of the Supplementary Material, we set sample size $n=5000$ for this case;  (2) the variance of the standard IM estimator $\mu_{n,IM}$ in \eqref{mcmcest} and the variance of our IMCV estimator $\mu_{\ell,B,IMCV}$ in  \eqref{batchest} where $n=\ell B$ after a certain number of quick adaptation burning-in steps.  These comparisons will be based on the ratios of the variance of the standard IM estimator to the other estimators variances estimated by the VRF in exactly the same way as in Section \ref{illustrations}, all VRFs are estimated by 50 independent replications.

\subsubsection{Logistic regression}
\label{sec:log}
We consider binary classification based on binary labels $y= \{y_i\}_{i=1}^N$ and input data $x= \{x_i\}_{i=1}^N$ where $x_i$ are $d$-dimensional input vectors. We assume a  logistic regression log-likelihood 
$p(y \vert x, \beta) = \sum_{i=1}^N \{y_i \log s(x_i,\beta) + (1-y_i) \log (1-s(x_i, \beta))\}$
where $s(x_i, \beta) = 1/(1 + \exp(-\beta^\top x_i))$ with $\beta$ being a $d$-dimensional parameter vector. We place a
$d$-dimensional Gaussian prior  $\mathcal{N}(\beta \vert 0, \boldsymbol{I}_d)$ on $\beta$ and we 
illustrate the performance of the adaptive independent Metropolis algorithm equipped with our control variate estimators in two datasets that have been commonly used in MCMC applications, see e.g.\ \citet{Girolami2011RiemannMethods} and \citet{Titsias2019Gradient-basedCarlo}.  See Table \ref{table:info} for the names of the datasets and details on the specific samples sizes and dimensions.
\begin{table}
\centering
\caption{Summary of datasets for logistic regression.}
\label{table:info}
\begin{tabular}{@{}lcccc@{}}
\hline
Dataset & Ripley & Pima Indian & Heart & German \\
\hline
$d$ & 3 & 8 & 14 & 25  \\
$N$ & 250 & 532 & 270 & 1000 \\
\hline
\end{tabular}
\end{table}

We estimate the  posterior expected values of all parameters $\beta$ and of the  odd ratio $r(\overline{x}, \beta) = \exp (\beta^\top  \overline{x})$
where $\overline{x}$ is the average input vector.  The  proposal distribution is $q_{\mu,\bs{L}}(\beta) = \mathcal{N}(\beta \vert \mu, \bs{L} \bs{L}^\top)$ with $\mu$ and $\bs{L}$ initially being a zero vector and the identity matrix respectively and being adapted with Algorithm~\ref{alg:sl} in Section~\ref{sec:adp-algo} of the Supplementary Material.  Note that the required analytical expectations with respect to the proposal density are readily available, for example 
$
\E_{q_{\mu,\bs{L}}}[r(\overline{x}, \beta)] = \exp (\mu^\top\overline{x} + \frac{1}{2}\overline{x}^\top \bs{L} \bs{L}^\top \overline{x} ).$
The variance reduction factors of estimators against the standard IM estimator $\mu_{n,IM}$ are shown in Table \ref{t:logreg2} and Table \ref{t:logreg}.
The Ripley dataset achieves the best VRFs because the lower dimension of the posterior implies a better approximation of the proposal to the target.  Moreover, as the batch and sample sizes increase, we obtain higher variance reduction.

\begin{table}
\centering
\caption{Sample collection after adaptation: range of VRFs (computed by the min and max across all dimensions of $\beta$) for the parameters of Logistic Regression datasets.}
\label{t:logreg2}
\begin{tabular}{@{}llcccc@{}}
\hline
   & $F(\beta)$ & $\mu_{n,IMCV}$ &  $\mu_{n,IMCV}^{(\hat{c}_1,\hat{c}_2)}$ & $\mu_{n,RB}$ & $\mu_{n,CIM}^{(\hat{c})}$  \\
\hline

\multirow{6}{*}{\shortstack{Ripley \\ AP: 0.97}} & \multirow{2}{*}{$\beta$} & 46.5 -- 65.6 & 47.6 -- 66.9 & 1.1 -- 1.2 & 9.6 -- 13.7  \\
 &  & -- & (1.046, 0.961) & -- & (1.047)  \\ 
 & \multirow{2}{*}{$\beta^2$} & 39.7 -- 54.0 & 41.8 -- 55.6 & 1.1 -- 1.3 & 7.3 -- 12.3  \\
 & & -- & (1.056, 0.965) & -- & (1.056)\\
  & \multirow{2}{*}{$r(\overline{x}, \beta)$} & 71.4 & 74.4 & 1.2 & 12.1 \\
 & & -- & (1.040, 0.964) & -- & (1.043)\\
 
\hline

\multirow{6}{*}{\shortstack{Pima \\ AP: 0.89}} & \multirow{2}{*}{$\beta$}  & 8.4 -- 20.0 & 9.5 -- 22.1 & 1.0 -- 1.4 & 3.5 -- 7.1 \\ 
 &   & -- & (1.093, 0.886) & -- & (1.094)  \\ 
  & \multirow{2}{*}{$\beta^2$} & 6.3 -- 16.0 & 9.8 -- 18.3 & 1.1 -- 1.3 & 2.6 -- 5.2  \\
 & & -- & (1.101, 0.871) & -- & (1.102)\\
  & \multirow{2}{*}{$r(\overline{x}, \beta)$} & 14.1 & 15.0 & 1.2 & 3.2 \\
 & & -- & (1.107, 0.890) & -- & (1.110)\\
 
\hline

\multirow{6}{*}{\shortstack{Heart \\ AP: 0.89}} & \multirow{2}{*}{$\beta$}  & 9.2 -- 24.7 & 11.4 -- 27.1 & 1.0 -- 1.4 & 2.5 -- 6.3  \\ 
&  & -- & (1.141, 0.854) & -- & (1.142)  \\ 
 & \multirow{2}{*}{$\beta^2$} & 10.0 -- 20.0 & 10.6 -- 21.9 & 1.0 -- 1.3 & 2.9 -- 5.5   \\
 & & -- & (1.111, 0.878) & -- & (1.111)\\
  & \multirow{2}{*}{$r(\overline{x}, \beta)$} & 15.8 & 19.4 & 1.3 & 2.89 \\
 & & -- & (1.108, 0.880) & -- & (1.102)\\
 
\hline

\multirow{6}{*}{\shortstack{German \\ AP: 0.81}} & \multirow{2}{*}{$\beta$}  & 4.6 -- 11.8 & 4.5 -- 11.5 & 0.9 -- 1.6 & 1.9 -- 4.1  \\ 
&  & -- & (1.153, 0.827) & -- & (1.150)  \\ 
 & \multirow{2}{*}{$\beta^2$} & 3.4 -- 10.1 & 3.6 -- 9.7 & 1.0 -- 1.6 & 1.6 -- 3.2  \\
 & & -- & (1.184, 0.802) & -- & (1.186)\\
  & \multirow{2}{*}{$r(\overline{x}, \beta)$} & 6.3 & 6.5 & 1.2 & 2.6 \\
 & & -- & (1.180, 0.810) & -- & (1.174)\\
\hline
\end{tabular}
\end{table}

\begin{table}
\centering
\caption{Sample collection during adaptation: range of VRFs (computed by the min and max across all $d$-dimensions of $\beta$) for the parameters of Logistic Regression datasets.}
\label{t:logreg}
\begin{tabular}{@{}lcrrrcrrr@{}}
\hline
& \multicolumn{3}{c}{$B=50, \ell=20$} & & \multicolumn{3}{c}{$B=500, \ell=200$} \\
\cline{2-4} \cline{6-8}
Dataset & $F(\beta)=\beta$ & $F(\beta)=\beta^2$ & $r(\overline{x}, \beta)$ & & $F(\beta)=\beta$ & $F(\beta)=\beta^2$ & $r(\overline{x}, \beta)$ \\
\hline
Ripley & 42.4 -- 50.5 & 40.7 -- 50.1 & 43.6 & & 41.4 -- 57.2 & 42.6 -- 54.5 & 52.5 \\
Pima   & 5.1 -- 14.3  & 3.0 -- 13.6  & 10.9 & & 9.7 -- 17.4  & 9.6 -- 17.1  & 16.4\\
Heart  & 6.5 -- 18.5  & 3.7 -- 16.3  & 12.2 & & 12.5 -- 22.1 & 10.3 -- 19.7 & 17.4\\
German & 3.2 -- 7.2   & 3.0 -- 6.8   & 5.5  & & 8.1 -- 11.5  & 6.2 -- 10.9  & 8.2\\
\hline
\end{tabular}
\end{table}

\subsubsection{Gaussian process hyperparameters
\label{sec:gpreg}
}

In our final example we consider Bayesian 
estimation of kernel hyperparameters in 
Gaussian process (GP) models \citep{MurrayAdams10, FilipponeG14}. We assume standard regression observed data $\left\{x_i, y_i \right\}_{i=1}^N$, where each $x_i \in \mathbb{R}^D$ is an input vector and $y_i \in \mathbb{R}$ is the corresponding scalar response. We model each response using 
a Gaussian likelihood, i.e.\ $y_i \sim \mathcal{N}(f(x_i), \sigma^2)$ and assign to the latent function $f(x)$ a zero-mean  GP distribution so that $f(x) \sim \mathcal{GP}(0, k(x,x'))$ and $k(x,x')$ is the kernel or covariance function.  
In our experiment, we use a GP having the squared-exponential kernel 
$$k(x,x') = \sigma_f^2 \exp\left\{-\frac{1}{2l^2}\Vert x-x'\Vert^2 \right\},$$
so that overall the GP regression model depends on three 
hyperparameters $\theta = (\sigma^2, \sigma_f^2, \ell^2)$, while  
the latent function variables $\{ f(x_i)\}_{i=1}^N$ can be analytically  
integrated out. To do Bayesian inference
over the hyperparameters we place a Gaussian prior on log space, i.e.\ 
$\log \theta \sim \mathcal{N}(\log\theta \vert 0,10\boldsymbol{I})$ and then 
apply independent Metropolis 
sampling by using a multivariate normal density $\mathcal{N}(\log \theta |\mu,\Sigma)$ as the proposal distribution.  
This proposal was adapted to approximate the exact GP posterior $p(\log \theta|y)$. We applied the sampler and our estimator to two data sets: Boston 
housing \citep{harrub78} 
and Pendulum \citep{lazaro-gredilla10a}. The Boston housing dataset has $N=455$ samples and $d=13$ number of covariates, whereas the Pendulum dataset has $N=315$ and $d=9$. As the function 
$F(x)$ we used both $F(x) = x$ and  $F(x) = \exp\{x\}$ where
$x \in \{\log \sigma^2, \log \sigma_f^2, \log \ell^2\}$ so that for the second case the expected value of $F(x)$ gives an estimate of the posterior mean of each hyperparameter in 
the linear space. Note that the 
expectation of $F(x)=\exp\left\{x\right\}$ under the proposal density is analytic by using the formula 
$\E_{\mathcal{N}(\mu,\sigma^2)} [F(x)] = \exp\left\{\mu + \sigma^2/2\right\}$. 

In this experiment, we set the number of batches to $\ell = 20$ and the batch size to $B=50$. The VRF scores are given in Table \ref{table:greg2} shows that the variances of our estimators are smaller than those of others. Again, there is considerable variance reduction for all estimators in both log space and linear space.

The performance of variance reduction is notably influenced by the dimensionality of the problem. In this section, we employ uniform hyperparameters across all test cases. This consistency in configuration ensures a uniform stochastic approximation error across all dimensions, providing a baseline for comparison and scaling analysis. In practice, however, higher-dimensional problems typically require more extensive training to achieve convergence. A more rigorous investigation of how the variance reduction scales with dimension, under the assumption of constant training error, is provided in Section~\ref{apd:nd} of the Supplementary Material.
\begin{table}
\centering
\caption{VRF for the parameters of Gaussian Process datasets.}
\label{table:greg2}
\begin{tabular}{@{}llcccc@{}}
\hline
   &  &  $\mu_{n,IMCV}$ &  $\mu_{n,IMCV}^{(\hat{c}_1,\hat{c}_2)}$ & $\mu_{n,RB}$ & $\mu_{n,CIM}^{(\hat{c})}$  \\
\hline
\multirow{8}{*}{\shortstack{Boston \\ AP: 0.92}} & $\log(\ell^2)$  & 3.0 & 3.2 & 1.2 & 1.7  \\
 & $\log(\sigma_f^2)$ & 2.1 & 2.5 & 1.1 & 1.6 \\
& $\log(\sigma^2)$  & 7.8 & 10.2 & 1.1 & 3.7 \\
&  & -- & (1.182, 0.901) & -- & (1.182)  \\
\cline{2-6}
&$\ell^2$  & 2.9 & 4.7 & 1.1 & 1.2  \\
 & $\sigma_f^2$ & 3.2 & 4.5 & 1.0 & 1.4 \\
& $\sigma^2$  & 7.0 & 7.5 & 1.0 & 2.2 \\
&  & -- & (1.193, 0.828) & -- & (1.194)  \\
\hline
\multirow{8}{*}{\shortstack{Pendulum \\ AP: 0.89}} & $\log(\ell^2)$ & 7.5 & 8.6 & 0.9 & 3.5 \\
 & $\log(\sigma_f^2)$  & 9.4 & 13.1 & 1.1 & 4.3  \\
& $\log(\sigma^2)$ & 1.7 & 3.0 & 0.9 & 1.2 \\
&  & -- & (1.149, 0.897) & -- & (1.145) \\
\cline{2-6}
& $\ell^2$ & 6.6 & 9.2 & 1.1 & 2.8 \\
 & $\sigma_f^2$  & 3.7 & 5.4 & 1.0 & 1.7  \\
& $\sigma^2$ & 3.3 & 5.0 & 0.9 & 1.3 \\
&  & -- & (1.164, 0.851) & -- & (1.164) \\
\hline
\end{tabular}
\end{table}

\section{Connection with other estimators
\label{sec:previousestimators}
} 

\subsection{Connection to Rao–Blackwellisation by integrating out decision step}

Our proposed estimator  connects with a basic Rao-Blackwellised estimator that integrates out the decision
step at each Metropolis–Hastings iteration; see for instance 
Section 6.1 in \cite{Robertetal2018}.
Given that $X_i$ is the  current state, $Y_i$ is the proposal and $u_i \sim U(0,1)$ is the uniform random variable used to accept or reject $Y_i$, the  IM estimator (or any other MCMC estimator)
can be written as\footnote{Note that we apply some re-indexing to
re-write $\frac{1}{n} \sum_{i=1}^n F(X_i)$ as $\frac{1}{n} \sum_{i=1}^n F(X_{i+1})$ so that the counting of the samples starts at $i=2$.} 
$$
\mu_{n,IM} = \frac{1}{n} \sum_{i=1}^{n} 
F(X_{i+1}) 
= \frac{1}{n} \sum_{i=1}^n I( u_i < \alpha(X_i, Y_i) ) F(Y_i) 
+ (1 - I( u_i < \alpha(X_i, Y_i) ))
F(X_i), 
$$
where we used that the next state $X_{i+1}$ is equal to either $Y_i$ when the 
indicator function $I( u_i < \alpha(X_i, Y_i) )=1$ or $X_i$ otherwise. By taking the expectation under the uniform distribution $U(0,1)$ to integrate out $u_i$ we obtain 
\begin{align}\label{rbest}
\mu_{n,RB} & = \frac{1}{n} \sum_{i=1}^n \alpha(X_i, Y_i)  F(Y_i) 
+ (1  - \alpha(X_i, Y_i) )
F(X_i) \nonumber \\
&= \frac{1}{n} 
\sum_{i=1}^n 
\left\{F(X_i)  + \alpha(X_i, Y_i)
(F(Y_i) - F(X_i)) 
\right\},
\end{align}
which is the same with \eqref{mu_CV_F_one}. Therefore, 
an interpretation of our final proposed estimator $\mu_{n,IMCV}$
in \eqref{imcv} is that it adds a
control variate $F(Y_i) - \mathbb{E}_q[F]$ to further 
reduce the variance associated 
with the proposed state $Y_i$ 
in the above Rao–Blackwellised estimator.

In practice, as we observe 
in our experiments in Section \ref{sec:adaptiveNumerical}, the Rao-Blackwellised estimator for IM it is not effective and it can lead to only very minor variance reduction. For example, in the ideal scenario where the proposal $q$ matches exactly the target $\pi$ and $\alpha(X_i,Y_i)=1$ for all samples, the Rao-Blackwellised estimator reduces
to $\frac{1}{n}\sum_{i=1}^n F(Y_i) = \frac{1}{n}\sum_{i=1}^n F(X_{i+1})$ which means that there is no variance reduction at all. In contrast, in such case our estimator from 
\eqref{imcv} will have zero variance.

\subsection{Connection with 
a coupling estimator}

\cite{pinto2001improving} proposed a control variate via coupling Markov chains to reduce the variance. By adopting this framework to independent Metropolis chain $\left\{X_i\right\}$ produced by the $IM(P,\pi,q)$ sampler, we construct a coupled Markov chain $\left\{Y_i\right\}$ where the target and proposal are the same, i.e., both are $q$. In other words, the coupled chain always accepts the proposed samples.  Therefore, we form the following estimator with control variate based on the coupled chain:
\begin{equation}\label{couplingIM}
    \mu_{n,CIM} = \frac{1}{n} \sum_{i = 1}^{n} \left\{ F(X_i)-(F(Y_{i-1}) - \E_q [F] ) \right\}.
\end{equation}
Note that when we use the coupled chain as control variate, the index of the sample $Y_{i-1}$ is shifted back one position to achieve a high correlation with the original chain, since whenever 
$IM(P, \pi,q)$ accepts the proposed sample $Y_{i-1}$ it holds $X_i=Y_{i-1}$.
We can also add a coefficient $\hat{c}$
in front of the control variate to further reduce the variance, see Section~\ref{sec:coef} of the Supplementary Material. The coupling estimator has zero variance in the ideal case. To see this, when $q$ matches exactly $\pi$, the $IM(P,\pi,q)$ sampler always accepts the proposed sample, i.e., $X_i = Y_{i-1}$ for all $i$. Then, the coupling estimator becomes $\E_q[F]$ which is the exact expectation. However, our experiments show that the coupling estimator, while it can achieve variance reduction, is still one order of magnitude worse than our proposed estimator $\mu_{n,IMCV}$ from \eqref{imcv}.

\subsection{Connection with an importance sampling estimator}\label{sec:IS}

Importance sampling (IS) estimators are based on i.i.d. samples  $X_1,X_2,\ldots,X_n$ from some  probability density (importance function) $q^*$ on $\X$ with respect to a measure $q^*(\cdot)$ that produce the self-normalised estimator 
\begin{equation}
\label{ISest}
        \mu_{n,IS} = \frac {\sum_{i=1}^n F(X_i) \omega(X_i)}{\sum_{i=1}^n \omega(X_i)},~~\omega (X_i) = \pi(X_i) / q^*(X_i)
\end{equation}
which can achieve a variance smaller than $\sigma^2_{n,MC^*}$ obtaining zero variance when $q^*(x) \propto |F(x) - \E_\pi[F] | \pi(x)$ for all $x \in \X$, see
\cite{Kahn1953MethodsComputations}.  
A direct comparison to our proposed estimator can be achieved by constructing a similar control variate for (\ref{ISest}) as follows.  Assuming that $\E_q[F]$ is analytically available, a control variate of the form $F - \E_q[F]$  leads to a self-normalised importance sampling with control variate (ISCV) estimator, defined as
\begin{equation}\label{eq:ISCV}
    \mu_{n,ISCV} = \frac{\sum_{i=1}^n F(X_i) \omega(X_i)}{\sum_{i=1}^n \omega(X_i)} - \frac{1}{n} \sum_{i=1}^n \big(F(X_i) - \mathbb{E}_q[F] \big).
\end{equation}
Although both ISCV and IMCV are consistent estimators,  ISCV is inherently biased due to its self-normalisation, whereas the bias of IMCV can be mitigated through a suitable burn-in period.  Even without burn-in, under the condition that $\mathbb{E}_q[\omega^{p}] < \infty$ for some $p \geq 2$, the bias of the IM estimator \eqref{mcmcest} is asymptotically lower than that of the IS estimator \eqref{ISest}  \citep[see Theorem 2.2 and 4.1 of][]{deligiannidis2024importance}.
Nevertheless, estimator \eqref{ISest} is preferable to \eqref{mcmcest}, as its use of independent samples results in lower variance and, consequently, smaller mean squared error. Furthermore, one may argue that IS is fully parallelisable, whereas IM offers only partial parallelisation.

However, a different picture emerges when comparing the control variate estimators. The IMCV estimator achieves lower MSE, see Figure~\ref{fig:mse} in Section~\ref{sec:iscv} of the Supplementary Material. More importantly, because practical applications typically rely on adaptive IS or IM algorithms, we conducted a comparative study in Section~\ref{sec:aiscv} of the Supplementary Material, where both estimators were evaluated using the exact same set of samples drawn from a sequence of adaptive proposals.  Since the IM-based estimator typically exhibits lower bias within each batch compared to the IS-based estimator, the batch estimator \eqref{batchest} has consequently substantially smaller bias than the adaptive IS estimator  
\begin{equation}\label{eq:aiscv}
        \mu_{\ell,B,ISCV} = \frac{1}{\ell}\sum_{i=1}^{\ell}\Bigg\{\frac{\sum_{j=1}^B F(X_{ij}) \omega_i(X_{ij})}{\sum_{j=1}^B \omega_i(X_{ij})} - \frac{1}{B} \sum_{j=1}^B \big(F(X_{ij}) - \mathbb{E}_{q_{\theta_i}}[F] \big)\Bigg\}.
\end{equation}
This leads to a striking reduction in MSE; see Figure~\ref{fig:aimcv_vs_aiscv} in Section~\ref{sec:aiscv} of the Supplementary Material.  Thus, although the IS is preferable than the IM, this preference is reversed when control variates are employed.

\section{Discussion}\label{sec:discussion}

Independent Metropolis algorithms have not attracted much interest in the MCMC community because of their poor performance when compared with other Metropolis-based sampling strategies.  We have attempted to revisit their applicability domain by pointing out that they can produce, without any extra computational effort,  ergodic estimators of any function of interest equipped with control variates that have reduced asymptotic variance.  
Our main result indicates that when the proposal density is close enough to the target density, these estimators can outperform the i.i.d. estimators based on crude MC.  Although this seems counter-intuitive since MCMC samples are not independent as those from crude MC, the key idea is that the dependency of MCMC samples provides the ability to construct control variates through the theory of solutions to the corresponding Poisson equation.  {This provides a methodological avenue for constructing reduced variance estimators that complements the usual tools of importance sampling and standard control variates, as well as the Stein-based post-processing methods discussed in Section \ref{sec:related-work}, to which we return below.}
Incorporation of our key idea to an adaptive independent Metropolis algorithm based on reducing the KL-divergence between the target and the proposal density has produced a very efficient sampling strategy as shown in a series of synthetic and real data examples.  
To illustrate our methodology we revisited some common Bayesian statistics problems such as approximation of marginal likelihoods and sampling from Bayesian posterior distributions in logistic regression and Gaussian processes.  We have shown that the applicability of independent Metropolis algorithms may be extended to produce solutions to many interesting popular statistical problems.

{Compared with reweighting draws from $q$, or the output of a $q$-invariant chain, by Stein importance sampling \citep{liu2017blackbox,wang2023stein}, retaining a Metropolis correction has two advantages. The first is adaptation. The optimal $\Pi$ of \citet{wang2023stein} is determined by the kernel and the target, so it can only be changed by changing the kernel, for which no adaptation theory currently available. In our framework the proposal $q_\theta$ is updated by stochastic minimisation of $KL(q_\theta \| \pi)$ using all candidate states, the resulting algorithm is ergodic with valid limit theorems, and the objective being minimised is exactly the quantity that controls the variance guarantee of Theorem 2.3. The acceptance rate and the empirical bound $\hat w^\star$ further indicate online how closely $q_\theta$ approaches $\pi$. The second advantage is structural. The accept-reject step creates the stationary law $\pi \otimes q$ under which our correction has zero mean; without this step the samples follow $q$, and the gap between $q$ and $\pi$ must be closed by information about $\pi$ external to the chain, through exact density ratios as in importance sampling, which we compare with in Section \ref{sec:IS}, or through the score as in Stein reweighting.

The two routes also differ in what they require of the target. Stein-based post-processing needs the score at every sample, the construction of $\Pi$ and its consistency theory involve second derivatives and tail conditions \citep[Thm.~1]{wang2023stein}, and score-free variants exist but need additional tuning \citep{fisher2023gradient}. Our estimator needs no derivatives of $\pi$. Its inputs are the analytic expectation $\E_q[F]$, available for Gaussian or mixture proposals with polynomial, exponential or indicator-type $F$, together with $w^\star < \infty$ and a proposal close to $\pi$ in KL divergence. The adaptive scheme uses the score of $\pi$ through the
reparametrisation gradient, but this is one of the choices, $q$ can be
fitted by methods that do not use the score, for example the EM
algorithm of Section~\ref{sec:le}. Regarding bias, weighted quadrature is biased at every finite $n$,
and the optimal $\Pi$ of \citep{wang2023stein} reduces the constant in
the $n^{-1/2}$ bound on its kernel Stein discrepancy but does not
remove it. The initialisation bias of IMCV decays
at the geometric rate of the chain, with a constant proportional to
the rejection residual, and vanishes at $q=\pi$, where the estimator
is exact at every finite $n$. Regarding cost, Stein constructions require the kernel matrix and a quadratic program in the number of samples in addition to score evaluations, while our estimator is linear in the number of samples and reuses the quantities $\alpha(X_i,Y_i)$ and $Y_i$ already computed during sampling.

Stein control variates can also be added to the output of the independent Metropolis chain, since they have zero mean under $\pi$ and the chain is $\pi$-invariant, and the two families can even be combined. The difference lies in the fitting criterion. Stein and zero-variance control variates are fitted by minimising the marginal variance, or a surrogate for the asymptotic variance of the ergodic average \citep{belomestny2020variance} rather than the asymptotic variance
itself, which lead to different optimisers \citep{dellaportas2012control}. Our correction is derived from the Poisson equation of the chain, so it targets the asymptotic variance directly, and the control variate coefficients can be estimated consistently via batch means \citep{Flegal2010BatchCarlo}. Moreover,  IMCV correction is a function of the current state and the proposal jointly, and when the proposal equals the target, it is exact for every function of interest with a tractable proposal expectation. Conversely, when the score is available and $F - \E_\pi[F]$ is well approximated within the Stein reproducing kernel Hilbert space, control functionals achieve super-root-$n$ rates \cite{Oates2017ControlIntegration,Oates2019ConvergenceMethod} that our estimator does not. The two routes are complementary, and the choice depends on whether the score of $\pi$ or analytic expectations under $q$ are at hand. Finally, assumption (A3) requires $q$ to be overdispersed relative to $\pi$, in the same direction as the optimal $\Pi$ of \citet{wang2023stein} derived from a different criterion.
}

Further improvements 
of the adaptive independent Metropolis algorithm can be related to the stochastic gradient optimisation scheme that we currently use to fit the proposal distribution to the target distribution by minimising the KL-divergence. 
One direction is to reduce 
the variance of the stochastic gradients in the optimisation, a topic that has attracted a lot of attention by the machine learning community 
\citep{Roeder17,Miller17, Geffner20} 
and it deserves further investigation. 
A second direction is to exploit 
parallelisation in the computations of the independent Metropolis algorithm, which 
thanks to the independent form 
of the proposal distribution, and unlike other MCMC algorithms, can be largely parallelised \citep{ParallelIM}. The use of parallelisation is an interesting topic for future research since it could lead to faster adaptation and further  improvements of the estimators.   {A third direction is to combine our correction with Stein control variates within a single estimator, exploiting the fact that the two families have zero mean under different reference measures and therefore act on different sources of variability.}


\section*{Acknowledgments}
We would like to thank Sam Livingstone, Max Hird, Arnaud Doucet, {two reviewers and the
editor} for very constructive comments.

\bibliographystyle{plainnat}
\bibliography{reference}

\clearpage
\appendix
\setcounter{section}{0}
\setcounter{equation}{0}
\setcounter{figure}{0}
\setcounter{table}{0}
\setcounter{algorithm}{0}
\renewcommand{\thesection}{S\arabic{section}}
\renewcommand{\theequation}{S\arabic{equation}}
\renewcommand{\thefigure}{S\arabic{figure}}
\renewcommand{\thetable}{S\arabic{table}}
\renewcommand{\thealgorithm}{S\arabic{algorithm}}

\begin{center}
{\LARGE Supplementary Material to ``Variance Reduction for Independent Metropolis''}
\end{center}
\addcontentsline{toc}{section}{Supplementary Material}

\section{Proof of theorems in the main paper}

\subsection{Preliminary lemmas}\label{sec:prelim}

This subsection collects the technical lemmas used repeatedly in the proofs
below. Lemma~\ref{lem:contr} provides a minorisation and an explicit
$L^2$ contraction rate for the independence sampler; it is used in the
proofs of Theorems~2.1, 2.3 and~3.2. 

Throughout, $w_x:=\pi(x)/q(x)$, $w^\star:=\sup_{x\in\X} w_x$ and
$\sigma_F:=\{\mathrm{Var}_\pi(F)\}^{1/2}$. In addition to the spaces
$L^{\pi,q}_n$ in the main paper, for a probability density $p$ on a space
$\mathcal{Y}$ we write
\[
    L^2(p) := \Big\{f:\mathcal{Y}\rightarrow\Real \;\Big\vert\;
    \int_{\mathcal Y} f(y)^2\,p(y)\,dy < \infty\Big\},
    \qquad
    \langle f,g\rangle_p := \int_{\mathcal Y} f(y)g(y)\,p(y)\,dy ,
\]
with norm $\Vert f\Vert_{L^2(p)}:=\langle f,f\rangle_p^{1/2}$; only the
cases $p=\pi$ on $\mathcal Y=\X$ and $p=\pi\otimes q$ on
$\mathcal Y=\X\times\X$ occur below, and
$L^{\pi,q}_2=L^2(\pi)\cap L^2(q)$. We also set
$L^2_0(\pi):=\{f\in L^2(\pi):\E_\pi[f]=0\}$. The summand of
$\mu_{n,IMCV}$ is denoted
\begin{equation}\label{eq:phi}
    \varphi(x,y)
    := F(x)+\alpha(x,y)\{F(y)-F(x)\}-\{F(y)-\E_q[F]\} ,
\end{equation}
which can be rearranged as
\begin{equation}\label{eq:phialt}
    \varphi(x,y) = \E_q[F]+\{1-\alpha(x,y)\}\{F(x)-F(y)\} ,
\end{equation}
and its centred version is
$\varphi_0:=\varphi-\E_{\pi\otimes q}[\varphi]$.

Since $X_{i+1}\in\{X_i,Y_i\}$, the proposal $Y_i$ is not independent of the
future of the $X$-chain, and functions of the pair $(X_i,Y_i)$ are not
functionals of the $X$-chain alone. The appropriate state process is the
augmented chain $Z_i=(X_i,Y_i)$ on $\X\times\X$, which is Markov with kernel
\begin{equation}\label{eq:augkernel}
    \widetilde{P}\big((x,y),\,d(x',y')\big)
    = K\big((x,y),\,dx'\big)\otimes Q(dy'),
\end{equation}
where $Q(dy)=q(y)\,dy$ and
$K\big((x,y),\cdot\big):=\alpha(x,y)\,\delta_y+\{1-\alpha(x,y)\}\,\delta_x$.
For $g:\X\to\Real$ we also write
$Kg(x,y)=\alpha(x,y)g(y)+\{1-\alpha(x,y)\}g(x)$, so that
$Kg(x,y)=\E[g(X_1)\mid X_0=x,\,Y_0=y]$. Two identities follow directly from the product form of the kernel.
First, integrating the proposal out of $K$ recovers one IM step: for
every $x$,
\begin{equation}\label{eq:KtoP}
    \int Q(dy)\,K\big((x,y),\,dx'\big)
    = \alpha(x,x')\,q(x')\,dx' + r(x)\,\delta_x(dx')
    = P(x,dx'),
\end{equation}
where $r(x)=\int q(y)\{1-\alpha(x,y)\}\,dy$ is the rejection
probability. Second, for $f:\X\times\X\to\Real$,
\begin{equation}\label{eq:Pbar}
    \widetilde{P}f(x,y)
    = \int K\big((x,y),dx'\big)\int Q(dy')\,f(x',y') ,
\end{equation}
that is, $\widetilde P$ averages out the proposal coordinate and
then applies $K$. In particular, for $g:\X\to\Real$, the inner
integral in \eqref{eq:Pbar} applied to $f=Kg$ equals
$\int Q(dy')\,Kg(x',y')=Pg(x')$ by \eqref{eq:KtoP}, and therefore
\begin{equation}\label{eq:PK}
    \widetilde{P}(Kg)=K(Pg).
\end{equation}

\begin{lemma}\label{lem:contr}
Assume (A1) and (A3). Then:
\begin{enumerate}
\item[(i)] The IM kernel satisfies the minorisation
\begin{equation}\label{eq:doeblin}
    P(x,\cdot)\;\geq\;\frac{1}{w^\star}\,\Pi(\cdot)
    \qquad\text{for all } x\in\X .
\end{equation}
\item[(ii)] For every $f\in L^2_0(\pi)$,
\begin{equation}\label{eq:L2contr}
    \Vert P^k f\Vert_{L^2(\pi)}
    \;\leq\;
    \Big(1-\frac{1}{w^\star}\Big)^{k}\Vert f\Vert_{L^2(\pi)},
    \qquad k\geq 0 .
\end{equation}
\end{enumerate}
\end{lemma}

\begin{proof}
(i) Since $\alpha(x,y)=\min\{1,w_y/w_x\}$ and $w_yq(y)=\pi(y)$, the
acceptance part of $P(x,\cdot)$ has density
\[
    \alpha(x,y)\,q(y)
    = \min\Big\{q(y),\,\frac{\pi(y)}{w_x}\Big\}
    \;\geq\; \frac{\pi(y)}{w^\star},
\]
where the inequality holds for both terms in the minimum:
$\pi(y)/w_x\geq\pi(y)/w^\star$ and
$q(y)=\pi(y)/w_y\geq\pi(y)/w^\star$. Adding the
nonnegative rejection mass gives, for every measurable $A$,
\[
    P(x,A)\;\geq\;\int_A \alpha(x,y)q(y)\,dy
    \;\geq\;\frac{1}{w^\star}\,\Pi(A).
\]

(ii) By \eqref{eq:doeblin}, define
\[
    R \;:=\; \Big(1-\frac{1}{w^\star}\Big)^{-1}
             \Big(P-\frac{1}{w^\star}\,\Pi\Big)
\]
is a Markov kernel: nonnegativity is exactly \eqref{eq:doeblin}, and
$R(x,\X)=(1-1/w^\star)^{-1}\{P(x,\X)-1/w^\star\}=1$ for every $x$.
Moreover $\pi R=\pi$, since
$\pi R=(1-1/w^\star)^{-1}(\pi P-\tfrac{1}{w^\star}\pi\Pi)
=(1-1/w^\star)^{-1}(\pi-\tfrac{1}{w^\star}\pi)=\pi$ by 
invariance $\pi P=\pi$. For $f\in L^2_0(\pi)$ we have
$\Pi f=\E_\pi[f]=0$, hence
\[
    Pf \;=\; \Big(1-\frac{1}{w^\star}\Big)\,Rf ,
\]
and $Rf$ is again centred:
$\E_\pi[Rf]=(\pi R)f=\pi f=0$. Furthermore, $R$ is a contraction on
$L^2(\pi)$, which can be proved via Jensen's inequality
\[
    \big(Rf(x)\big)^2
    = \Big(\int f(y)\,R(x,dy)\Big)^{\!2}
    \;\leq\; \int f(y)^2\,R(x,dy)
    = R(f^2)(x),
\]
 integrating against $\pi$ and using $\pi R=\pi$,
\[
    \Vert Rf\Vert_{L^2(\pi)}^2
    = \int (Rf)^2\,d\pi
    \;\leq\; \int R(f^2)\,d\pi
    = (\pi R)(f^2)
    = \pi(f^2)
    = \Vert f\Vert_{L^2(\pi)}^2 .
\]
Since $Rf$ is centred whenever $f$ is, iterating
$Pf=(1-1/w^\star)Rf$ gives
$P^kf=(1-1/w^\star)^kR^kf$, and \eqref{eq:L2contr} follows from the
contraction bound applied $k$ times.
\end{proof}

\begin{lemma}\label{lem:aug}
Assume (A1) - (A3). Then:
\begin{enumerate}
\item[(i)] $\pi\otimes q$ is invariant for $\widetilde{P}$.
\item[(ii)] $\widetilde{P}^2\big((x,y),\cdot\big)\geq
\frac{1}{w^\star}\,(\Pi\otimes Q)(\cdot)$ for all $(x,y)$, and the augmented
chain is uniformly ergodic, with
\begin{equation}\label{eq:augTV}
    \big\Vert \widetilde{P}^k\big((x,y),\cdot\big)-\pi\otimes q
    \big\Vert_{\mathrm{TV}}
    \;\leq\;
    \Big(1-\frac{1}{w^\star}\Big)^{k-1}
    \qquad\text{for all } (x,y),\; k \geq 1 .
\end{equation}
\item[(iii)] The measure $\pi(x)q(y)\alpha(x,y)\,dx\,dy$ has the symmetric
density
\[
    \pi(x)q(y)\,\alpha(x,y)
    \;=\; q(x)q(y)\,\min\{w_x,w_y\} ,
\]
and consequently, in stationarity,
$\E_{\pi\otimes q}\big[\alpha(X,Y)\{F(Y)-F(X)\}\big]=0$ 
and \\
$\E_{\pi\otimes q}[\varphi(X,Y)]=\E_\pi[F]$.
\end{enumerate}
\end{lemma}

\begin{proof}
(i) Because \eqref{eq:augkernel} is a product kernel, the second coordinate $Y'\sim q$ is drawn
independently of everything, so its marginal is $q$ and it is independent
of $X'$. For the first coordinate, by \eqref{eq:KtoP}, if $(X,Y)\sim\pi\otimes q$, then for every measurable
$A\subseteq\X$,
\[
    \Pr(X'\in A)
    = \iint \Pi(dx)\,Q(dy)\,K\big((x,y),A\big)
    = \int \Pi(dx)\,P(x,A)
    = \Pi(A)
\]
by the $\pi$-invariance of $P$, so $X'\sim\pi$.

(ii) For the total variation bound, the product structure of
\eqref{eq:augkernel} gives, for every $k\geq1$,
\[
    \widetilde{P}^k\big(z,\cdot\big)
    = \big(K(z,\cdot)P^{k-1}\big)\otimes Q ,
\]
since the proposal $Y_k$ is drawn from $Q$ afresh and independently of
$X_k$, whose conditional law is that of $k-1$ steps of $P$ started from
$K(z,\cdot)$. Let $k=2$, we have $\widetilde{P}^2\big(z,\cdot\big)
    = \big(K(z,\cdot)P\big)\otimes Q$ ,with measure
$\mu_z=K(z,\cdot)P$, then the inequality holds. For product measures sharing the second factor,
$\vert\mu\otimes Q-\nu\otimes Q\vert=\vert\mu-\nu\vert\otimes Q$ gives
$\Vert\mu\otimes Q-\nu\otimes Q\Vert_{\mathrm{TV}}
=\Vert\mu-\nu\Vert_{\mathrm{TV}}$. Writing $\pi\otimes q=\pi\otimes Q$ and
using that $K(z,\cdot)$ is a probability measure together with
\eqref{eq:doeblin},
\[
\big\Vert\widetilde{P}^k\big(z,\cdot\big)-\pi\otimes q\big\Vert_{\mathrm{TV}}
    = \big\Vert K(z,\cdot)P^{k-1}-\pi\big\Vert_{\mathrm{TV}}
    \;\leq\;\sup_{x'\in\X}\big\Vert P^{k-1}\big(x',\cdot\big)-\pi\big\Vert_{\mathrm{TV}}
    \;\leq\;\Big(1-\frac{1}{w^\star}\Big)^{k-1},
\]
the last step by Theorem 8 of \cite{Roberts2004GeneralAlgorithms} applied to
the one-step minorisation \eqref{eq:doeblin}.

(iii) Using $\pi(x)=w_xq(x)$,
\[
    \pi(x)q(y)\,\alpha(x,y)
    = q(x)q(y)\,w_x\min\Big\{1,\frac{w_y}{w_x}\Big\}
    = q(x)q(y)\min\{w_x,w_y\},
\]
which is symmetric in $(x,y)$. Since $F(y)-F(x)$ is antisymmetric and
integrable under (A2), the integral of their product vanishes. The last
claim then follows from \eqref{eq:phi} together with
$\E_{\pi\otimes q}[F(Y)]=\E_q[F]$.
\end{proof}

\begin{lemma}\label{lem:Kcontr}
Let $(X_0,Y_0)\sim\pi\otimes q$. For every $g\in L^2(\pi)$,
$\Vert Kg\Vert_{L^2(\pi\otimes q)}\leq\Vert g\Vert_{L^2(\pi)}$, and if
$\E_\pi[g]=0$ then $\E_{\pi\otimes q}[Kg]=0$.
\end{lemma}

\begin{proof}
Since $Kg(x,y)=\E[g(X_1)\mid Z_0=(x,y)]$, conditional Jensen gives
$(Kg)^2\leq K(g^2)$ pointwise, and by Lemma~\ref{lem:aug}(i) the
$X$-marginal is stationary, so
\[
    \E_{\pi\otimes q}\big[K(g^2)\big]
    = \E\big[g^2(X_1)\big]
    = \E_\pi[g^2].
\]
Hence $\Vert Kg\Vert_{L^2(\pi\otimes q)}^2\leq\Vert g\Vert_{L^2(\pi)}^2$.
By the law of total expectation
$\E_{\pi\otimes q}[Kg]=\E[g(X_1)]=\E_\pi[g]=0$.
\end{proof}

\subsection{Proof of Theorem 2.1}\label{thmp1}
\begin{proof}
    Write $f := F - \E_\pi[F] \in L^2_0(\pi)$, so that
    $\Vert f\Vert_{L^2(\pi)}=\sigma_F$, and set
    $M := \sup_i w^\star_i$ and $\rho := 1-1/M < 1$. By
    Lemma~\ref{lem:contr}(ii) applied to $P_i$, whose contraction rate
    $1-1/w^\star_i$ is bounded by $\rho$ uniformly in $i$,
    \begin{equation}\label{eq:unif-tail}
        \Vert P^n_i f\Vert_{L^2(\pi)} \;\leq\; \rho^n \sigma_F
        \qquad\text{for all } i\geq 1,\; n\geq 0 .
    \end{equation}
    In particular the series
    \begin{equation}\label{eq:poisson-rep}
        \sum_{n=0}^{\infty} P^n_i f
        \;=\; \sum_{n=0}^{\infty}\big\{P^n_i F - \E_{\pi}[F]\big\}
    \end{equation}
    converges absolutely in $L^2(\pi)$, with $L^2(\pi)$-norm at most
    $M\sigma_F$ uniformly in $i$. It is the solution of the Poisson
    equation of $\mathrm{IM}(P_i,\pi,q_i)$ normalised so that
    $\E_\pi[\hat F_i]=0$. Indeed, if $g\in L^2_0(\pi)$ satisfies
    $P_ig-g=-f$, then iterating gives
    $g=P^k_ig+\sum_{n=0}^{k-1}P^n_if$ for every $k\geq1$, and
    $\Vert P^k_ig\Vert_{L^2(\pi)}\leq\rho^k\Vert g\Vert_{L^2(\pi)}\to0$
    by \eqref{eq:unif-tail} applied to $g$, so
    $g=\sum_{n\geq0}P^n_if$ in $L^2(\pi)$. The centred solution is
    therefore unique in $L^2_0(\pi)$, and we write
    $\hat F_i:=\sum_{n\geq0}P^n_if$.

    We next show $\Vert P_i f\Vert_{L^2(\pi)} \to 0$. Write
    $w_{i,x} := \pi(x)/q_i(x)$. Since the acceptance part of $P_i$ has
    density $\alpha_i(x,y)q_i(y)=\min\{q_i(y),\pi(y)/w_{i,x}\}$,
    \[
        P_i f(x)
        \;=\; \int f(y)\,
              \min\Big\{q_i(y),\,\frac{\pi(y)}{w_{i,x}}\Big\}\,dy
        \;+\; r_i(x)f(x),
    \]
    where, using $u-\min\{u,v\}=(u-v)_+$,
    \[
        r_i(x)
        = \int q_i(y)\{1-\alpha_i(x,y)\}\,dy
        = \int\Big(q_i(y)-\frac{\pi(y)}{w_{i,x}}\Big)_{\!+}\,dy .
    \]
    Insert and subtract $\min\{\pi(y),\,\pi(y)/w_{i,x}\}$ inside the
    integral of the acceptance part. The inserted term contributes
    $\min\{1,1/w_{i,x}\}\int f(y)\pi(y)\,dy = 0$, because the factor
    $\min\{1,1/w_{i,x}\}$ depends on $x$ only and $f$ is centred. For
    the difference, $|\min\{a,c\}-\min\{b,c\}| \leq |a-b|$ and
    Cauchy-Schwarz give
    \begin{align*}
        \Big|\int f(y)\,
        \min\Big\{q_i(y),\frac{\pi(y)}{w_{i,x}}\Big\}\,dy\Big|
        &\;\leq\; \int |f|\,|q_i-\pi|\,dy \\
        &\;\leq\;
        \big(\E_\pi[f^2]+\E_{q_i}[f^2]\big)^{1/2}
        \big(2\,\mathrm{TV}(q_i,\pi)\big)^{1/2}
        \;=:\; a_i ,
    \end{align*}
    a bound independent of $x$, where $|q_i-\pi| \leq q_i+\pi$ was used
    on one square-root factor. By Scheff\'e's lemma, the pointwise
    convergence $q_i \to \pi$ implies $\mathrm{TV}(q_i,\pi)\to 0$, and
    $\sup_i \E_{q_i}[f^2] < \infty$ by assumption, so $a_i \to 0$.     For the rejection term, insert and subtract $\pi(y)$ inside the
    positive part and use inequality $(a+b)_+\leq a_++b_+$:
    \[
        \Big(q_i(y)-\frac{\pi(y)}{w_{i,x}}\Big)_{\!+}
        \;\leq\;
        \big(q_i(y)-\pi(y)\big)_+
        + \pi(y)\Big(1-\frac{1}{w_{i,x}}\Big)_{\!+} .
    \]
    Integrating over $y$, and using that $q_i$ and $\pi$ both
    integrate to one, so that
    $\int(q_i-\pi)_+\,dy=\frac12\int|q_i-\pi|\,dy
    =\mathrm{TV}(q_i,\pi)$,
    \[
        r_i(x)
        \;\leq\;
        \mathrm{TV}(q_i,\pi)
        + \Big(1-\frac{1}{w_{i,x}}\Big)_{\!+} .
    \]
    Since $q_i(x)\to\pi(x)$ pointwise, $w_{i,x}\to 1$ and thus
    $r_i(x)\to 0$ for every $x$, while $0\leq r_i\leq 1$. Dominated
    convergence theorem with dominating function $f^2\in L^1(\pi)$ yields
    $\Vert r_i f\Vert_{L^2(\pi)}\to 0$.

    Finally, since $\pi P_i = \pi$, the function $P_i f$ is again
    centred, so Lemma~\ref{lem:contr}(ii) applied to $P_i f$ gives
    $\Vert P^n_i f\Vert_{L^2(\pi)} \leq
    \rho^{\,n-1}\Vert P_i f\Vert_{L^2(\pi)}$ for $n\geq 1$. Let $\epsilon>0$, since $\rho<1$, there exists $N$ such that
    $\sum_{n>N}\rho^n\sigma_F<\epsilon/2$; by \eqref{eq:unif-tail}
    this tail bound holds for every $i$. Since
    $\Vert P_if\Vert_{L^2(\pi)}\to0$, there exists $i_0$ such that
    $\Vert P_if\Vert_{L^2(\pi)}\sum_{n=1}^{N}\rho^{\,n-1}<\epsilon/2$
    for all $i\geq i_0$. Then for all $i \geq i_0$,
    \[
        \big\Vert \hat F_i - f \big\Vert_{L^2(\pi)}
        \;=\; \Big\Vert \sum_{n\geq 1} P^n_i f \Big\Vert_{L^2(\pi)}
        \;\leq\; \sum_{n=1}^{N}\rho^{\,n-1}
                 \Vert P_i f\Vert_{L^2(\pi)}
               + \sum_{n>N}\rho^{n}\sigma_F
        \;<\; \epsilon ,
    \]
    and therefore $\hat F_i \to F - \E_\pi[F]$ in $L^2(\pi)$.
\end{proof}

\subsection{Proof of Proposition 2.2}\label{propp1}
\begin{proof}
    By Lemma~\ref{lem:aug}(i)--(ii), the augmented chain $\{Z_i\}$ is
    uniformly ergodic with invariant density $\pi\otimes q$. Similarly, the
 one-step minorisation \eqref{eq:doeblin} gives
    $\Vert P^n(x,\cdot)-\Pi\Vert_{\mathrm{TV}}\leq(1-1/w^\star)^n$ for
    all $x$ \citep[Theorem~8]{Roberts2004GeneralAlgorithms}, so the IM
    chain $\{X_i\}$ is uniformly ergodic with the invariant density $\pi$. This
    proves the first statement.

    For the law of large numbers, uniform ergodicity implies that
    $\{Z_i\}$ is positive Harris recurrent. Since $0\leq\alpha\leq1$,
    \eqref{eq:phialt} gives
    $|\varphi(x,y)|\leq|\E_q[F]|+|F(x)|+|F(y)|$, so
    $\varphi\in L^1(\pi\otimes q)$ under (A2). The strong law of
    large numbers for Harris chains
    \citep[Theorem~17.0.1(i)]{Meyn2009MarkovEdition} therefore applies, and the limit is
    $\E_{\pi\otimes q}[\varphi]=\E_\pi[F]$ by
    Lemma~\ref{lem:aug}(iii).

    For the central limit theorem, we first prove the covariance series of IMCV estimator is absolutely summable. To do this, Cauchy-Schwarz gives
    $\E_{\pi\otimes q}[\varphi^2]
    \leq 3\big(\E_q[F]^2+\E_\pi[F^2]+\E_q[F^2]\big)<\infty$ under
    (A2), so $\varphi_0=\varphi-\E_\pi[F]\in L^2(\pi\otimes q)$.
    Write $\bar{\varphi}_0(x):=\int\varphi_0(x,y)\,q(y)\,dy$, which
    satisfies $\E_\pi[\bar\varphi_0]=\E_{\pi\otimes q}[\varphi_0]=0$
    and $\Vert\bar\varphi_0\Vert_{L^2(\pi)}\leq
    \Vert\varphi_0\Vert_{L^2(\pi\otimes q)}$ by Jensen. Since
    \eqref{eq:Pbar} and \eqref{eq:PK},
    \begin{equation}\label{eq:Pkid}
        \widetilde{P}^k\varphi_0
        = K\big(P^{k-1}\bar\varphi_0\big),
        \qquad k\geq 1 .
    \end{equation}
    Conditioning on $Z_0$, the stationary lag-$k$ autocovariances
    $\gamma_k:=\E_{\pi\otimes q}[\varphi_0(Z_0)\varphi_0(Z_k)]$ can
    therefore be written as
    \[
        \gamma_k
        = \E_{\pi\otimes q}\big[\varphi_0\,
          \widetilde{P}^k\varphi_0\big]
        = \E_{\pi\otimes q}\big[\varphi_0\,
          K\big(P^{k-1}\bar\varphi_0\big)\big] ,
    \]
    an inner product in $L^2(\pi\otimes q)$ of $\varphi_0$ with
    $K(P^{k-1}\bar\varphi_0)$. By Cauchy--Schwarz, then by
    Lemma~\ref{lem:Kcontr} and
    Lemma~\ref{lem:contr}(ii), for $k\geq1$,
       \begin{align}
        |\gamma_k|
        \;\leq\;
        \Vert\varphi_0\Vert_{L^2(\pi\otimes q)}\,
        \big\Vert K\big(P^{k-1}\bar\varphi_0\big)
        \big\Vert_{L^2(\pi\otimes q)}
        \;&\leq\;
        \Vert\varphi_0\Vert_{L^2(\pi\otimes q)}\,
        \big\Vert P^{k-1}\bar\varphi_0\big\Vert_{L^2(\pi)}
        \label{eq:covdecayA}\\
        \;&\leq\;
        \Big(1-\frac{1}{w^\star}\Big)^{k-1}
        \Vert\varphi_0\Vert_{L^2(\pi\otimes q)}^2 ,
        \label{eq:covdecayB}
    \end{align}
    so the covariance series is absolutely summable. The central limit
    theorem for square-integrable functionals of uniformly ergodic
    chains \citep[Corollary~5]{Jones2004MarkovCLT}, valid for any
    initial distribution, holds
    \[
        \sqrt{n}\,\big\{\mu_{n,IMCV}-\E_\pi[F]\big\}
        \;\xrightarrow{\;d\;}\;
        \mathcal{N}\big(0,\sigma^2_{IMCV}\big),
    \]
    with $\sigma^2_{IMCV}=\gamma_0+2\sum_{k\geq1}\gamma_k<\infty$.

   For the positivity, we first show $\langle Pf,f\rangle_\pi\geq0$. By
Lemma~\ref{lem:aug}(iii), the joint law of $(X_0,X_1)$ is
\[
    \Pi(dx)P(x,dy)
    = q(x)q(y)\min\{w_x,w_y\}\,dx\,dy
    + r(x)\,\Pi(dx)\,\delta_x(dy),
\]
so that
\[
    \langle Pf,f\rangle_\pi
    = \iint f(x)f(y)\,q(x)q(y)\min\{w_x,w_y\}\,dx\,dy
    + \int r(x)f(x)^2\,\pi(x)\,dx ,
\]
and the second term is nonnegative. For the first term, note that
the integrand is absolutely integrable: since
$\min\{w_x,w_y\}^2\leq w_xw_y$, we have
$\min\{w_x,w_y\}\leq\sqrt{w_xw_y}$, and by Cauchy--Schwarz
\[
    \iint|f(x)f(y)|\,q(x)q(y)\sqrt{w_xw_y}\,dx\,dy
    =\Big(\int|f|\sqrt{\pi q}\,\Big)^2
    \leq \int f^2\pi\int q=\Vert f\Vert_{L^2(\pi)}^2<\infty.
\]
Writing
$\min\{a,b\}=\int_0^\infty\mathbb{I}\{a>t\}\,\mathbb{I}\{b>t\}\,dt$
for $a,b\geq0$ and by
Fubini's theorem,
\[
    \iint f(x)f(y)\,q(x)q(y)\min\{w_x,w_y\}\,dx\,dy
    = \int_0^\infty\Big[\int f(x)\,q(x)\,
      \mathbb{I}\{w_x>t\}\,dx\Big]^2 dt \;\geq\; 0 .
\]
Thus, $\langle Pf,f\rangle_\pi\geq0$ for every $f\in L^2(\pi)$. For
higher orders, self-adjointness of $P$ on $L^2(\pi)$ gives
$\langle P^{2m}f,f\rangle_\pi=\Vert P^mf\Vert_{L^2(\pi)}^2\geq0$ and
$\langle P^{2m+1}f,f\rangle_\pi=\langle P(P^mf),P^mf\rangle_\pi
\geq0$ by the one-step case applied to $P^mf$. Finally, writing
$\bar F:=F-\E_\pi[F]$, the asymptotic variance of the IM estimator
is $\sigma^2_{IM}=\mathrm{Var}_\pi(F)
+2\sum_{k\geq1}\langle P^k\bar F,\bar F\rangle_\pi$, and every
summand is nonnegative, so
$\sigma^2_{IM}\geq\mathrm{Var}_\pi(F)=\sigma^2_F$, an alternative
derivation of this comparison is given by
\cite{Tan2006MonteAcceptance-rejection}.
\end{proof}

\subsection{Proof of Theorem 2.3}\label{thmp2}

The proof is based on the following lemmas. The first controls the
asymptotic variance on the augmented chain, on which
$\{\varphi_0(Z_i)\}$ is an actual functional of the state; the
remaining lemmas bound the stationary variance
$\mathrm{Var}_{\pi\otimes q}(\varphi)$ in terms of the
Kullback--Leibler divergence.

\begin{lemma}\label{lem:varbound}
Assume (A1)--(A3) and let $v:=\mathrm{Var}_{\pi\otimes q}(\varphi)$.
Then
\begin{equation}\label{eq:sigbound}
    \sigma^2_{IMCV} \;\leq\; \big(1+2w^\star\big)\,v
\end{equation}
as well as
\begin{equation}\label{eq:sigbound2}
    \sigma^2_{IMCV} \;\leq\; v+2(w^\star-1)\,\sigma_F\sqrt{v} ,
\end{equation}
so that $\sigma^2_{IMCV}$ is bounded by the minimum of the two
right-hand sides.
\end{lemma}

\begin{proof}
Recall from the proof of Proposition~2.2 (Section~\ref{propp1}) that
$\sigma^2_{IMCV}=\gamma_0+2\sum_{k\geq1}\gamma_k$ with the
autocovariances bounded as in
\eqref{eq:covdecayA}--\eqref{eq:covdecayB}. Since
$\gamma_0=\Vert\varphi_0\Vert^2_{L^2(\pi\otimes q)}=v$, summing the
geometric series $\sum_{k\geq1}(1-1/w^\star)^{k-1}=w^\star$ in
\eqref{eq:covdecayB} yields
$\sigma^2_{IMCV}\leq v+2w^\star v$, which is \eqref{eq:sigbound}.

For \eqref{eq:sigbound2}, we first integrate \eqref{eq:phi} over $y$ under $q$, the term
$F(y)-\E_q[F]$ integrates to zero, and therefore
\[
    \int\varphi(x,y)\,q(y)\,dy
    = \int\big[F(x)+\alpha(x,y)\{F(y)-F(x)\}\big]\,q(y)\,dy
    = PF(x) .
\]
Centring both sides then gives $\bar\varphi_0=P\bar F$, and
consequently $P^{k-1}\bar\varphi_0=P^{k}\bar F$ for $k\geq1$.
Combining \eqref{eq:covdecayA} with Lemma~\ref{lem:contr}(ii)
applied to $\bar F$,
\[
    |\gamma_k|
    \;\leq\;
    \Vert\varphi_0\Vert_{L^2(\pi\otimes q)}\,
    \Vert P^{k}\bar F\Vert_{L^2(\pi)}
    \;\leq\;
    \Big(1-\frac{1}{w^\star}\Big)^{k}\sqrt{v}\;\sigma_F ,
    \qquad k\geq1,
\]
 summing the geometric series
$\sum_{k\geq1}(1-1/w^\star)^k=w^\star-1$ yields
\eqref{eq:sigbound2}.
\end{proof}

\begin{remark}\label{rmk:twobounds}
The two bounds are complementary. The bound \eqref{eq:sigbound} is
linear in $v$, which keeps the algebra simple, thus we use this form to
derive the Kullback--Leibler threshold in the proof of
Theorem~2.3. The bound \eqref{eq:sigbound2} is sharper than
\eqref{eq:sigbound} precisely when
$\sqrt{v}>(1-1/w^\star)\,\sigma_F$, it holds
in particular for all proposals considered in
Section~\ref{ap:1dn}, where \eqref{eq:sigbound2} attains the
minimum throughout.
\end{remark}

\begin{lemma}\label{lem:defect}
For any densities $\pi,q$,
\[
    \E_{\pi\otimes q}\big[(1-\alpha(X,Y))^2\big]
    \;\leq\; \mathbb{KL}(q\Vert\pi)+\mathbb{KL}(\pi\Vert q) .
\]
\end{lemma}

\begin{proof}
Write $s_1:=\pi(x)q(y)$ and $s_2:=\pi(y)q(x)$, so that
$\alpha=\min\{1,s_2/s_1\}$ and $1-\alpha=(1-s_2/s_1)_+$. Since
$1-r\leq-\log r$ for $0<r\leq1$,
\[
    (1-\alpha)^2
    = \Big(1-\frac{s_2}{s_1}\Big)_{\!+}^{2}
    \;\leq\; \Big(1-\frac{s_2}{s_1}\Big)_{\!+}
             \Big(\log\frac{s_1}{s_2}\Big)_{\!+} .
\]
Multiplying by the density $s_1$ of $\pi\otimes q$ and integrating,
\begin{align*}
    \E_{\pi\otimes q}\big[(1-\alpha)^2\big]
    &\;\leq\;
    \iint (s_1-s_2)_+\,\big(\log s_1-\log s_2\big)_+\,dx\,dy \\
    &\;=\;
    \frac12\iint (s_1-s_2)\big(\log s_1-\log s_2\big)\,dx\,dy ,
\end{align*}
where the equality holds because the second integrand is
nonnegative, symmetric under the swap $(x,y)\mapsto(y,x)$, and
coincides with the first integrand on $\{s_1>s_2\}$, which the swap
maps onto $\{s_1<s_2\}$. The second integral is the Jeffreys
divergence
$\mathbb{KL}(\pi\otimes q\,\Vert\,q\otimes\pi)
+\mathbb{KL}(q\otimes\pi\,\Vert\,\pi\otimes q)$, and
$\mathbb{KL}$ factorises over product measures:
$\mathbb{KL}(\pi\otimes q\Vert q\otimes\pi)
=\mathbb{KL}(\pi\Vert q)+\mathbb{KL}(q\Vert\pi)$, and likewise for its
reverse. The claim follows.
\end{proof}

\begin{lemma}\label{lem:revkl}
Under (A3),
$\mathbb{KL}(\pi\Vert q)\leq\Lambda(w^\star)\,\mathbb{KL}(q\Vert\pi)$,
where
\[
    \Lambda(w^\star)
    := \frac{w^\star\log w^\star-w^\star+1}{w^\star-1-\log w^\star}
    \qquad\text{for } w^\star>1,
    \qquad \Lambda(1):=1,
\]
and the function $\Lambda$ is nondecreasing.
\end{lemma}

\begin{proof}
Define $\psi(w):=w\log w-w+1$ and $\phi(w):=w-1-\log w$ for $w>0$;
both are nonnegative and vanish only at $w=1$. Writing
$w=w_x=\pi(x)/q(x)$ and using $\E_q[w]=1$, we have
\begin{align*}
    \mathbb{KL}(\pi\Vert q)
    &= \int\pi\log\frac{\pi}{q}\,dx
    = \E_q\big[w\log w\big]
    = \E_q[\psi(w)],\\
    \mathbb{KL}(q\Vert\pi)
    &= \int q\log\frac{q}{\pi}\,dx
    = \E_q\big[-\log w\big]
    = \E_q[\phi(w)],
\end{align*}
so it suffices to show $\psi(w)\leq\Lambda(w^\star)\,\phi(w)$ for
all $0<w\leq w^\star$. We show that this ratio is
nondecreasing on $(0,\infty)$. Since $\psi'(w)=\log w$ and
$\phi'(w)=1-1/w$,
\[
    \Big(\frac{\psi}{\phi}\Big)'(w)
    = \frac{\psi'(w)\phi(w)-\psi(w)\phi'(w)}{\phi(w)^2}
    = \frac{(w-1)^2-w\log^2 w}{w\,\phi(w)^2} ,
\]
and the substitution $w=e^{2t}$ turns the inequality
$(w-1)^2\geq w\log^2w$ into $|\sinh t|\geq|t|$, which holds for all
real $t$. Therefore $(\psi/\phi)'\geq0$ on $(0,1)\cup(1,\infty)$,
and since $\psi/\phi$ extends continuously to $w=1$ with value $1$,
it is nondecreasing on $(0,\infty)$. Consequently
$\psi(w)/\phi(w)\leq\psi(w^\star)/\phi(w^\star)=\Lambda(w^\star)$
for all $0<w\leq w^\star$, which proves the inequality, and the same
monotonicity shows that $\Lambda$ is nondecreasing. Asymptotically,
$\Lambda(w^\star)\to1$ as $w^\star\downarrow1$, and
$\Lambda(w^\star)=\log w^\star-1+o(1)$ as $w^\star\to\infty$.
\end{proof}

\begin{lemma}\label{lem:vKL}
Assume (A3) and $\mathbb{KL}(q\Vert\pi)\leq\epsilon$, and let
$v:=\mathrm{Var}_{\pi\otimes q}(\varphi)$ as in
Lemma~\ref{lem:varbound}. Let $n\geq2$ and
$C_n:=\max\{\E_\pi[F^{2n}],\E_q[F^{2n}]\}<\infty$. Then
\begin{equation}\label{eq:vKL}
    v \;\leq\; 4\,C_n^{1/n}
    \Big[\big\{1+\Lambda(w^\star)\big\}\,\epsilon\Big]^{\frac{n-1}{n}} .
\end{equation}
In particular, under (A2), which gives $C_2<\infty$,
$v\leq4C_2^{1/2}\{1+\Lambda(w^\star)\}^{1/2}\epsilon^{1/2}$; and if
$\Vert F\Vert_\infty<\infty$ then
$v\leq4\Vert F\Vert_\infty^2\{1+\Lambda(w^\star)\}\,\epsilon$.
\end{lemma}

\begin{proof}
From \eqref{eq:phialt},
\[
    v \;\leq\; \E_{\pi\otimes q}\big[(1-\alpha)^2\Delta_F^2\big],
    \qquad \Delta_F:=F(X)-F(Y) .
\]
H\"older's inequality with exponents $\big(\tfrac{n}{n-1},n\big)$,
together with $0\leq1-\alpha\leq1$ and $\tfrac{2n}{n-1}\geq2$ (so that
$(1-\alpha)^{2n/(n-1)}\leq(1-\alpha)^2$), gives
\[
    v \;\leq\;
    \Big(\E\big[(1-\alpha)^2\big]\Big)^{\frac{n-1}{n}}
    \Big(\E\big[\Delta_F^{2n}\big]\Big)^{\frac1n} .
\]
Since $|\Delta_F|^{2n}\leq2^{2n-1}\{|F(X)|^{2n}+|F(Y)|^{2n}\}$ we get
$\E[\Delta_F^{2n}]\leq2^{2n}C_n$, i.e.\
$(\E[\Delta_F^{2n}])^{1/n}\leq4C_n^{1/n}$. Lemmas~\ref{lem:defect}
and~\ref{lem:revkl} bound the first factor by
$[\{1+\Lambda(w^\star)\}\epsilon]^{(n-1)/n}$, which proves
\eqref{eq:vKL}. The bounded case uses
$\Delta_F^2\leq4\Vert F\Vert_\infty^2$ directly, without
H\"older.
\end{proof}

\begin{proof}[Proof of Theorem 2.3]
The inequality $\sigma^2_F\leq\sigma^2_{IM}$ is part 4 of
Proposition~2.2. For the first inequality, combining
\eqref{eq:sigbound} with \eqref{eq:vKL} for $n=2$,
\[
    \sigma^2_{IMCV}
    \;\leq\;
    4\,\big(1+2w^\star\big)\,C_2^{1/2}
    \big\{1+\Lambda(w^\star)\big\}^{1/2}\,\epsilon^{1/2} ,
\]
so $\sigma^2_{IMCV}\leq\sigma^2_F$ holds whenever
\[
    \epsilon \;\leq\;
    \left(\frac{\sigma^2_F}
    {4\,(1+2w^\star)\,C_2^{1/2}\,\{1+\Lambda(w^\star)\}^{1/2}}
    \right)^{2} .
\]
\end{proof}

\subsection{Proof of Corollary 3.1}\label{corp1}
\begin{proof}
    Fix $i$ and write $q_i:=q_{\theta_i}$ for the $i$-th proposal. We first
    check that the $i$-th chain satisfies (A1)--(A3), so that
    Lemmas~\ref{lem:varbound} and~\ref{lem:vKL} apply to it.
    Assumption (A5) is (A1) for $\mathrm{IM}(P_i,\pi,q_i)$. By (A6),
    $w^\star_i\leq M<\infty$; if $w^\star_i=1$ then $q_i=\pi$, the
    acceptance probability is identically one and $\varphi\equiv\E_{q_i}[F]$
    is constant, so $\sigma^2_{IMCV,(i)}=0$ and the bound below holds
    trivially, and we may therefore assume $1<w^\star_i\leq M$, which is
    (A3). For (A2), (A7) provides $W\geq1$ with $F\in L^W_\infty$,
    $\E_\pi[W^4]<\infty$ and
    $\sup_{\theta\in\Theta}\E_{q_\theta}[W^4]<\infty$, so
    $W\in L^{\pi,q_i}_4$ uniformly in $i$.

    The same bounds make the constants of Lemma~\ref{lem:vKL} uniform:
    since $|F|\leq\Vert F\Vert_W W$,
    \[
        C^{(i)}_2:=\max\big\{\E_\pi[F^4],\,\E_{q_i}[F^4]\big\}
        \;\leq\;
        \Vert F\Vert_W^4\,
        \max\Big\{\E_\pi[W^4],\,
        \sup_{\theta\in\Theta}\E_{q_\theta}[W^4]\Big\}
        \;=:\;\bar C\;<\;\infty ,
    \]
    a bound independent of $i$. Applying \eqref{eq:sigbound} and
    \eqref{eq:vKL} with $n=2$ to the $i$-th chain, and using that
    $w\mapsto1+2w$ is nondecreasing and that $\Lambda$ is nondecreasing
    by Lemma~\ref{lem:revkl},
    \begin{align*}
        \sigma^2_{IMCV,(i)}
        &\;\leq\;
        4\,\big(1+2w^\star_i\big)\,\big(C^{(i)}_2\big)^{1/2}
        \big\{1+\Lambda(w^\star_i)\big\}^{1/2}\epsilon_i^{1/2} \\
        &\;\leq\;
        4\,(1+2M)\,\bar C^{1/2}\,\{1+\Lambda(M)\}^{1/2}\,
        \epsilon_i^{1/2} ,
    \end{align*}
    where $\mathbb{KL}(q_i\Vert\pi)\leq\epsilon_i$ was used in
    \eqref{eq:vKL}. The prefactor does not depend on $i$ and
    $\epsilon_i\to0$ by hypothesis, hence
    $\sigma^2_{IMCV,(i)}\to0$.
\end{proof}

\subsection{Proof of Theorem 3.2}\label{thmp3}

Throughout, batches are indexed by $i=1,\ldots,\ell$ and within-batch
samples by $j=1,\ldots,B$, and we write $n=\ell B$. Time indices are
the pairs $t=(i,j)$, ordered lexicographically, which is the order that the samples are generated. We write $t-1$ for the predecessor
of $t$ in this order, so that $(i,j)-1=(i,j-1)$ for $j\geq2$ and
$(i,1)-1=(i-1,B)$ for $i\geq2$, and $\sum_t$ for the sum over all
$n$ indices. The sample path is $Z_t=(X_t,Y_t)$, and $i(t)$ denotes
the batch containing $t$. The parameter $\theta_i$ used in batch $i$
is determined by the samples of batches $1,\ldots,i-1$ only, and is
held fixed within batch $i$. Within batch $i$ the proposal is $q_{\theta_i}$ and the acceptance
probability is $\alpha_{\theta_i}$, so that for two consecutive times
$t-1=(i,j-1)$ and $t=(i,j)$ in the same batch, $j\geq2$, the pair
$Z_t$ is drawn from $\widetilde P_{\theta_i}(Z_{t-1},\cdot)$, with
$\widetilde P_{\theta_i}=K_{\theta_i}\otimes Q_{\theta_i}$ the
augmented kernel of \eqref{eq:augkernel}. At a batch boundary, that
is, for $t-1=(i,B)$ and $t=(i+1,1)$ with $1\leq i<\ell$, the
acceptance step producing $X_t$ still uses $\alpha_{\theta_i}$,
because the proposal $Y_{t-1}$ was drawn from $q_{\theta_i}$ in batch
$i$, while the fresh proposal $Y_t$ is drawn from $q_{\theta_{i+1}}$,
therefore $Z_t\sim\widetilde Q_i(Z_{t-1},\cdot)$ with
$\widetilde Q_i:=K_{\theta_i}\otimes Q_{\theta_{i+1}}$. All quantities
carrying a subscript $\theta$ ($w^\star_\theta$, $\varphi_\theta$,
$P_\theta$, $K_\theta$, and so on) are the objects of
Section~\ref{sec:prelim} for the proposal $q_\theta$. We write $\lambda_\theta := \pi \otimes q_\theta$ for the invariant law of the augmented chain with proposal $q_\theta$, so that $\E_{\lambda_\theta} = \E_{ \pi \otimes q_\theta}$. We also write
$\bar F:=F-\E_\pi[F]$,  and let $\{\mathcal{G}_t\}$ denote the filtration generated by the
samples and the adaptation up to time $t$ with
$\mathcal G_0:=\sigma(X_{(1,1)},\theta_1)$, note
$\theta_i\in\mathcal G_{t-1}$ for every $t$ in batch $i$. Define
$\mathcal F_i:=\mathcal G_{(i,B)}$ for the $\sigma$-field generated by
the first $i$ batches, and
$\Phi(x,y):=\{1+W(x)+W(y)\}^2$. We also define $\rho:=1-1/M$ where $M$ is from (A6).

We prove the central limit theorem in the regime $\ell\to\infty$,
$B\to\infty$ with $\ell/B\to0$. The law of large numbers only requires
$B\to\infty$ and $n=\ell B\to\infty$. Throughout we assume that the initial state satisfies
$\E[W(X_{(1,1)})^4]<\infty$, which holds in particular when $X_{(1,1)}$
is drawn from $q_{\theta_1}$, by (A7), and more generally for any
deterministic starting point.

\begin{lemma}\label{lem:envelope}
Assume (A6) and (A7), and set
$c_0=\max\big\{1,\;\E_\pi[W],\;\sup_{\theta}\E_{q_\theta}[W]\big\}<\infty$. Then:
\begin{enumerate}
\item[(i)] For every $\theta$ and $k\geq0$, we have the decomposition
$P_\theta^k(x,\cdot)=(1-\rho^k)\,\Pi(\cdot)+\rho^k R_\theta^k(x,\cdot)$,
where $R_\theta:=\rho^{-1}(P_\theta-M^{-1}\Pi)$ is a Markov kernel with
$\pi R_\theta=\pi$.
\item[(ii)] The series $\hat g_\theta:=\sum_{k\geq1}P_\theta^k\bar F$
converges in $L^2(\pi)$, solves
$(I-P_\theta)\hat g_\theta=P_\theta\bar F$, and satisfies
\[
    |\hat g_\theta(x)|\;\leq\;C_g\,\big(1+W(x)\big),
\]
for all
$x\in\X$ and $\theta\in\Theta$, where $C_g:=(M^2+M-1)\,c_0\,\Vert F\Vert_W $.
\item[(iii)] For every measurable $G\geq0$ and every $t > (1,1)$,
\[
    \E\big[G(X_t)\big]
    \;\leq\;
    \sup_{\theta\in\Theta}\E_{q_\theta}[G]
    +\Big(1-\frac1M\Big)\,\E\big[G(X_{t-1})\big] ;
\]
consequently
$\sup_t\E[(1+W(X_t))^4]\leq
M\sup_\theta\E_{q_\theta}[(1+W)^4]+\E[(1+W(X_{(1,1)}))^4]<\infty$.
\end{enumerate}
\end{lemma}

\begin{proof}
(i) By (A6), $w^\star_\theta\leq M$, so Lemma~\ref{lem:contr}(i)
gives the minorisation $P_\theta\geq M^{-1}\Pi$; consequently
$R_\theta:=\rho^{-1}(P_\theta-M^{-1}\Pi)$ is a Markov kernel, and
$\pi R_\theta=\pi$ by the invariance $\pi P_\theta=\pi$, exactly as
in the proof of Lemma~\ref{lem:contr}(ii). Rearranging the definition
of $R_\theta$,
\[
    P_\theta = M^{-1}\Pi+\rho R_\theta ,
    \qquad\text{therefore}\qquad
    P_\theta-\Pi = \rho R_\theta-(1-M^{-1})\Pi = \rho\,(R_\theta-\Pi) .
\]
Using the identities $\Pi P_\theta=\Pi$,
$P_\theta\Pi=\Pi$ and $\Pi^2=\Pi$, by induction on $k$,
\[
    (P_\theta-\Pi)^{k+1}
    = (P_\theta^k-\Pi)(P_\theta-\Pi)
    = P_\theta^{k+1}-P_\theta^k\Pi-\Pi P_\theta+\Pi^2
    = P_\theta^{k+1}-\Pi ,
\]
that is, $(P_\theta-\Pi)^k=P_\theta^k-\Pi$ for all $k\geq1$, and
likewise $(R_\theta-\Pi)^k=R_\theta^k-\Pi$. Then,
\[
    P_\theta^k-\Pi
    = (P_\theta-\Pi)^k
    = \rho^k(R_\theta-\Pi)^k
    = \rho^k\big(R_\theta^k-\Pi\big),
\]
and rearranging gives
$P_\theta^k=(1-\rho^k)\,\Pi+\rho^k R_\theta^k$.

(ii) The convergence in $L^2(\pi)$ and the identity
$(I-P_\theta)\hat g_\theta=P_\theta\bar F$ follow from
Lemma~\ref{lem:contr}(ii) exactly as in Section~\ref{thmp1}. For the
bound, since $|F|\leq\Vert F\Vert_W W$ and
$|\E_\pi F|\leq\Vert F\Vert_W\,\E_\pi[W]\leq\Vert F\Vert_Wc_0$,
\[
    |\bar F(x)|\;\leq\;\Vert F\Vert_W\,\big(c_0+W(x)\big).
\]
The density of the acceptance part of $P_\theta$ is
$\alpha_\theta(x,y)q_\theta(y)\leq q_\theta(y)$, and the rejection
probability satisfies $r_\theta(x)\leq1-M^{-1}=\rho$ by the
minorisation, so that
\[
    P_\theta W(x)
    \;\leq\;\E_{q_\theta}[W]+r_\theta(x)W(x)
    \;\leq\;c_0+\rho\,W(x),
\]
and therefore, dropping the nonpositive term $-M^{-1}\E_\pi[W]$,
\[
    R_\theta W
    = \rho^{-1}\big(P_\theta W-M^{-1}\E_\pi[W]\big)
    \;\leq\;\rho^{-1}c_0+W .
\]
By induction on $k$, starting from the bound
on $\bar F$ and using $R_\theta(a+W)\leq a+\rho^{-1}c_0+W$ for any
constant $a$, we obtain
\[
    R_\theta^k|\bar F|
    \;\leq\;\Vert F\Vert_W\,R_\theta^k\big(c_0+W\big)
    \;\leq\;\Vert F\Vert_W\,\big(c_0+k\rho^{-1}c_0+W\big).
\]
Since $\E_\pi[\bar F]=0$, part (i) gives
$P_\theta^k\bar F=\rho^kR_\theta^k\bar F$, and thus
\[
    |P_\theta^k\bar F|
    \;\leq\;\rho^k\,R_\theta^k|\bar F|
    \;\leq\;\Vert F\Vert_W\,\rho^k\big(c_0+k\rho^{-1}c_0+W\big).
\]
Summing over $k\geq1$ and using
$\sum_{k\geq1}\rho^k=\rho/(1-\rho)=M-1$ and
$\sum_{k\geq1}k\rho^{k-1}=(1-\rho)^{-2}=M^2$,
\[
    |\hat g_\theta|
    \;\leq\;\Vert F\Vert_W\big\{(M-1)c_0+M^2c_0+(M-1)W\big\}
    \;\leq\;(M^2+M-1)\,c_0\,\Vert F\Vert_W\,(1+W),
\]
where the last step uses $c_0\geq1$. This is the bound with $C_g$ as
defined.

(iii) Let $\theta$ be the parameter under which $Y_{t-1}$ was drawn, so
that $\theta=\theta_i$ when $t-1$ lies in batch $i$, including at a batch
boundary, where $t=(i+1,1)$ and the acceptance step still uses
$\alpha_{\theta_i}$. Conditionally on $\mathcal G_{t-1}$ the state $X_t$
equals $Y_{t-1}$ with probability $\alpha_\theta(X_{t-1},Y_{t-1})$ and
$X_{t-1}$ otherwise, so that
\[
    \E\big[G(X_t)\,\big|\,\mathcal G_{t-1}\big]
    =\alpha_\theta(X_{t-1},Y_{t-1})\,G(Y_{t-1})
    +\big\{1-\alpha_\theta(X_{t-1},Y_{t-1})\big\}\,G(X_{t-1}).
\]
Taking expectations and using that $Y_{t-1}$ is drawn from $q_\theta$
afresh, independently of $X_{t-1}$ given the past,
\begin{align*}
    \E\big[G(X_t)\big]
    &= \E\bigg[\int\min\Big\{q_\theta(y),
      \frac{\pi(y)}{w_{\theta,X_{t-1}}}\Big\}\,G(y)\,dy
    + r_\theta(X_{t-1})\,G(X_{t-1})\bigg]\\
    &\leq\;\sup_{\theta'\in\Theta}\E_{q_{\theta'}}[G]
      +\rho\,\E\big[G(X_{t-1})\big],
\end{align*}
using $\min\{q_\theta,\pi/w_{\theta,X_{t-1}}\}\leq q_\theta$ for the
first term and $r_\theta(X_{t-1})\leq1-M^{-1}=\rho$, obtained by
integrating \eqref{eq:doeblin} over $y$, for the second. This is the
recursion. Writing $s:=\sup_\theta\E_{q_\theta}[G]$ and $m_t:=\E[G(X_t)]$, the
recursion holds $m_t\leq s+\rho\,m_{t-1}$. Let $t$ be the $k$-th
index in the lexicographic order, so that $t-1$ is the $(k-1)$-th and
$(1,1)$ the first. Applying the recursion $k-1$ times,
\begin{align*}
    m_t
    &\;\leq\; s+\rho\,m_{t-1}
    \;\leq\; s+\rho\,s+\rho^2\,m_{t-2}
    \;\leq\;\cdots
    \;\leq\; s\sum_{r=0}^{k-2}\rho^r+\rho^{k-1}m_{(1,1)}\\
    &\;\leq\; \frac{s}{1-\rho}+m_{(1,1)}
    \;=\; Ms+m_{(1,1)} ,
\end{align*}
using $\sum_{r\geq0}\rho^r=(1-\rho)^{-1}=M$ and $\rho^{k-1}\leq1$.
The bound does not depend on $t$. For $G=(1+W)^4$, $s$ is finite by
(A7) and $m_{(1,1)}=\E[(1+W(X_{(1,1)}))^4]$ by the initial
condition, then the inequality holds.
\end{proof}

For $\theta\in\Theta$, let $\varphi_\theta$ denote
\eqref{eq:phi} with proposal $q_\theta$, and let
$\psi_\theta:=\varphi_\theta-\E_{\pi\otimes q_\theta}[\varphi_\theta]$
be its centred version, the function $\varphi_0$ of
Section~\ref{sec:prelim} for the proposal $q_\theta$. By
Lemma~\ref{lem:aug}(iii) applied to $q_\theta$,
$\E_{\pi\otimes q_\theta}[\varphi_\theta]=\E_\pi[F]$ for every
$\theta\in\Theta$, so that $\psi_\theta=\varphi_\theta-\E_\pi[F]$.
The stationary mean of the summand therefore does not depend on
$\theta$, and the analysis below can be centred at $\E_\pi[F]$
throughout.

The solution to the Poisson equation for $\psi_\theta$ on the
augmented chain is obtained by summing the series
$\sum_{k\geq0}\widetilde P_\theta^k\psi_\theta$. By the identity \eqref{eq:Pbar} and \eqref{eq:PK},
$\widetilde P_\theta^k\psi_\theta=K_\theta(P_\theta^k\bar F)$ for every $k\geq1$. Then, the series equals
$\psi_\theta+K_\theta\hat g_\theta$, with $\hat g_\theta$ as in
Lemma~\ref{lem:envelope}(ii). We therefore define
\[
    \hat h_\theta(x,y)
    := \psi_\theta(x,y)+K_\theta\hat g_\theta(x,y)
    = \psi_\theta(x,y)+\alpha_\theta(x,y)\,\hat g_\theta(y)
      +\{1-\alpha_\theta(x,y)\}\,\hat g_\theta(x) .
\]

\begin{lemma}\label{lem:poisson}
Assume (A6) and (A7). Then:
\begin{enumerate}
\item[(i)] $\hat h_\theta-\widetilde P_\theta\hat h_\theta
=\psi_\theta$, and for all $x,y\in\X$ and $\theta\in\Theta$,
\[
    |\hat h_\theta(x,y)|\leq C_h\{1+W(x)+W(y)\},
    \qquad
    C_h:=2c_0 \Vert F\Vert_W+C_g .
\]
\item[(ii)] Define
$S_\theta:=\widetilde P_\theta\big(\hat h_\theta^2\big)
-\big(\widetilde P_\theta\hat h_\theta\big)^2$, then
\[
    \E_{\pi\otimes q_\theta}\big[S_\theta\big]
    \;=\;\sigma^2_{IMCV}(\theta),
\]
the asymptotic variance of the IMCV estimator with the fixed
proposal $q_\theta$.
\end{enumerate}
\end{lemma}

\begin{proof}
(i) By \eqref{eq:Pbar},
\eqref{eq:PK} and $\hat g_\theta-P_\theta\hat g_\theta=P_\theta\bar F$ from
Lemma~\ref{lem:envelope}(ii), 
\[
    \hat h_\theta-\widetilde P_\theta\hat h_\theta
    =\psi_\theta
    +K_\theta\big\{\hat g_\theta-P_\theta\bar F
    -P_\theta\hat g_\theta\big\}
    =\psi_\theta.
\]
For the bound, the representation
\eqref{eq:phialt} gives
$|\psi_\theta(x,y)|\leq|F(x)|+|F(y)|+|\E_{q_\theta}F|+|\E_\pi F|
\leq\Vert F\Vert_W\{W(x)+W(y)+2c_0\}$, and since
$K_\theta\hat g_\theta(x,y)$ is a convex combination of
$\hat g_\theta(x)$ and $\hat g_\theta(y)$,
Lemma~\ref{lem:envelope}(ii) gives
$|K_\theta\hat g_\theta(x,y)|\leq \max\{|\hat g_\theta(x)|,|\hat g_\theta(y)|\} \leq C_g\{1+W(x)+W(y)\}$. Adding the
two bounds gives the claim with $C_h$ as stated.

(ii) Write $\lambda:=\lambda_\theta$ (drop the
subscript $\theta$). By invariance of $\lambda$ under $\widetilde P$
(Lemma~\ref{lem:aug}(i)),
$\E_\lambda[\widetilde P(\hat h^2)]=\E_\lambda[\hat h^2]$, so
\[
    \E_\lambda[S]
    =\E_\lambda[\hat h^2]-\E_\lambda[(\widetilde P\hat h)^2]
    =\E_\lambda[\hat h^2]-\E_\lambda[(\hat h-\psi)^2]
    =2\,\E_\lambda[\hat h\,\psi]-\E_\lambda[\psi^2].
\]
The series $\hat h=\sum_{k\geq0}\widetilde P^k\psi$ converges in
$L^2(\lambda)$ by \eqref{eq:Pkid}, Lemma~\ref{lem:Kcontr} and
    Lemma~\ref{lem:contr}(ii) follow the same proof as Proposition 2.2. Then, we have
\[
    \E_\lambda[\hat h\,\psi]
    = \sum_{k\geq0}\E_\lambda\big[\psi\,\widetilde P^k\psi\big]
    = \sum_{k\geq0}\gamma_k ,
\]
with $\gamma_k=\E_\lambda[\psi(Z_0)\psi(Z_k)]$ the autocovariances of
the proof of Proposition~2.2 (Section~\ref{propp1}) and
$\gamma_0=\E_\lambda[\psi^2]$. Therefore
\[
    \E_\lambda[S]
    = 2\sum_{k\geq0}\gamma_k-\gamma_0
    = \gamma_0+2\sum_{k\geq1}\gamma_k
    = \sigma^2_{IMCV}(\theta),
\]
where the last equality is the representation of the asymptotic
variance established in the same proof.
\end{proof}

\begin{lemma}\label{lem:sigmacont}
Assume (A6), (A7) and (A8). Then, for constant $C_\sigma$
\[
    \sup_{\theta\in\Theta}\big|\sigma^2_{IMCV}(\theta)\big|
    \;\leq\;C_\sigma\;<\;\infty ,
\]
and $\theta\mapsto\sigma^2_{IMCV}(\theta)$ is continuous at
$\theta^\star$.
\end{lemma}

\begin{proof}
Writing
$\gamma_k(\theta)=\E_{\lambda_\theta}[\psi_\theta\,
\widetilde P_\theta^{\,k}\psi_\theta]$, the proof of
Lemma~\ref{lem:poisson}(ii) gives
$\sigma^2_{IMCV}(\theta)=\gamma_0(\theta)+2\sum_{k\geq1}\gamma_k(\theta)$.

Under (A6) and (A7) every $q_\theta$ satisfies (A2) and (A3), so
\eqref{eq:covdecayB} in the proof of Proposition~2.2 applies at each
fixed $\theta$ and gives
$|\gamma_k(\theta)|\leq(1-1/w^\star_\theta)^{k-1}
\Vert\psi_\theta\Vert^2_{L^2(\lambda_\theta)}$ for $k\geq1$, while
$\gamma_0(\theta)=\Vert\psi_\theta\Vert^2_{L^2(\lambda_\theta)}$. The
right-hand sides involve only $w^\star_\theta$ and
$\Vert\psi_\theta\Vert^2_{L^2(\lambda_\theta)}$, bounded uniformly in
$\theta$ by (A6) and (A7) respectively (see the proof of Proposition 2.2), so that
\[
    |\gamma_k(\theta)|\;\leq\;C\rho^{\,k-1}
    \quad (k\geq1),
    \qquad
    \sup_{\theta\in\Theta}\big|\sigma^2_{IMCV}(\theta)\big|
    \;\leq\;C_\sigma,
\]
which gives the uniform bound. Next, we prove the continuity property. Since $q_\theta\rightarrow
q_{\theta^\star}$ a.e. by (A8) and these are probability
densities, Scheff\'e's lemma gives
$\int \vert q_\theta - q_{\theta^\star}\vert dx\rightarrow0$. If
$|h|\leq c\,W^2$, then $|h|\leq cK^2$ on $\{W\leq K\}$ and
$|h|\leq cK^{-2}W^4$ on $\{W>K\}$ for any constant $K>0$, so that
\begin{align*}
    \Big|\int h\,(q_\theta-q_{\theta^\star})\,dx\Big|
    \; &\leq \int_{\{W\leq K\}} |h|\vert q_\theta - q_{\theta^\star}\vert dx +\leq \int_{\{W> K\}} |h|\vert q_\theta - q_{\theta^\star}\vert dx\\
    & \leq\; cK^2\,\int \vert q_\theta - q_{\theta^\star}\vert dx
    +\frac{c}{K^2}\int W^4\,(q_\theta+q_{\theta^\star})\,dx ,
\end{align*}
the last integral being bounded uniformly in $\theta$ by (A7).
Letting $\theta\rightarrow\theta^\star$ and then $K\rightarrow\infty$
gives $\int h\,q_\theta\rightarrow\int h\,q_{\theta^\star}$, and
taking $h=W$, $h=W^2$ and $h=F$ yields
\begin{equation}\label{eq:momconv}
    \int W^sq_\theta\rightarrow\int W^sq_{\theta^\star}
    \quad(s=1,2),
    \qquad
    \E_{q_\theta}[F]\rightarrow\E_{q_{\theta^\star}}[F].
\end{equation}

Fix $k\geq 1$, given $x_0$ and
proposals $y_0,\dots,y_{k-1}$, the first coordinate of the chain after
$m$ transitions is determined by the acceptance indicators
$b=(b_1,\dots,b_k)\in\{0,1\}^k$ through
\[
  x_0(b):=x_0,\qquad
  x_m(b):=\begin{cases} y_{m-1}, & b_m=1,\\ x_{m-1}(b), & b_m=0,\end{cases}
  \qquad m=1,\dots,k,
\]
so that $x_k(b)\in\{x_0,y_0,\dots,y_{k-1}\}$ for every $b$, and the
probability of the history $b$ is
\[
  \omega_\theta(b)
  :=\prod_{m=1}^{k}\Big[
      b_m\,\alpha_\theta\big(x_{m-1}(b),y_{m-1}\big)
      +(1-b_m)\big\{1-\alpha_\theta\big(x_{m-1}(b),y_{m-1}\big)\big\}
    \Big].
\]
The weights $\omega_\theta(b)$ are nonnegative and
$\sum_{b}\omega_\theta(b)=1$, since at each step the two alternatives
have probabilities $\alpha_\theta$ and $1-\alpha_\theta$. Define
\[
  g_{k,\theta}(x_0,y_0,\dots,y_k)
  :=\psi_\theta(x_0,y_0)\sum_{b\in\{0,1\}^k}
    \omega_\theta(b)\,\psi_\theta\big(x_k(b),y_k\big),
\]
which is the conditional expectation of
$\psi_\theta(Z_0)\psi_\theta(Z_k)$ given $X_0=x_0$ and
$Y_0=y_0,\dots,Y_k=y_k$. Since, under $\lambda_\theta$, $X_0\sim\pi$ and
$Y_0,\dots,Y_k$ are i.i.d.\ $q_\theta$ and independent of $X_0$ by the
product form \eqref{eq:augkernel} of $\widetilde P_\theta$,
\[
  \gamma_k(\theta)
  =\int_{\X^{k+2}} g_{k,\theta}(x_0,y_0,\dots,y_k)\,
    \pi(x_0)\prod_{r=0}^{k}q_\theta(y_r)\,dx_0\,dy_0\cdots dy_k .
\]

 By (A6), $q_\theta\geq M^{-1}\pi$, so
$w_\theta=\pi/q_\theta$ is finite, and by (A8), $q_\theta\to
q_{\theta^\star}$ a.e., hence $w_\theta\to w_{\theta^\star}$ and
$\alpha_\theta\to\alpha_{\theta^\star}$ a.e.\ on $\X\times\X$. Since
$\psi_\theta$ depends on $\theta$ only through
$\alpha_\theta$ and $\E_{q_\theta}[F]$, \eqref{eq:momconv} gives
$\psi_\theta\to\psi_{\theta^\star}$ a.e., and since each
$\omega_\theta(b)$ is a finite product of values of $\alpha_\theta$ and
$1-\alpha_\theta$, also $\omega_\theta(b)\to\omega_{\theta^\star}(b)$
a.e.\ for every $b$. Therefore $g_{k,\theta}\to g_{k,\theta^\star}$
a.e.\ on $\X^{k+2}$.

The bound
$|\psi_\theta(x,y)|\leq \|F\|_W\{W(x)+W(y)+2c_0\}$ established in the
proof of Lemma~\ref{lem:poisson}(i) is uniform in $\theta$. As
$x_k(b)\in\{x_0,y_0,\dots,y_{k-1}\}$ and $W\geq 1$,
\[
  \big|\psi_\theta\big(x_k(b),y_k\big)\big|
  \le C\Big\{1+W(x_0)+\sum_{r=0}^{k}W(y_r)\Big\}
\]
for every $b$, and the same bound holds for $|\psi_\theta(x_0,y_0)|$.
Averaging over $b$ with the weights $\omega_\theta(b)$ and multiplying,
\[
  |g_{k,\theta}(x_0,y_0,\dots,y_k)|
  \le C\Big\{1+W(x_0)+\sum_{r=0}^{k}W(y_r)\Big\}^2,
\]
with $C=C(k,\|F\|_W,c_0)$ not depending on $\theta$. Hence the integrand
of $\gamma_k(\theta)$ is dominated in absolute value by
\[
  G_\theta(x_0,y_0,\dots,y_k)
  :=C\Big\{1+W(x_0)+\sum_{r=0}^{k}W(y_r)\Big\}^2
    \pi(x_0)\prod_{r=0}^{k}q_\theta(y_r).
\]
The function $G_\theta$ converges a.e.\ to $G_{\theta^\star}$ by (A8).
Then, by \eqref{eq:momconv}, 
$\int G_\theta\to\int G_{\theta^\star}<\infty$. Generalised dominated
convergence \citep{Pratt1960} therefore gives
$\gamma_k(\theta)\to\gamma_k(\theta^\star)$ as $\theta\to\theta^\star$.
For $k=0$ the same argument applies with
$g_{0,\theta}(x_0,y_0):=\psi_\theta(x_0,y_0)^2$ and
$G_\theta:=C\{1+W(x_0)+W(y_0)\}^2\pi(x_0)q_\theta(y_0)$.

Finally, since
$|\gamma_k(\theta)|\le C\rho^{\,k-1}$ uniformly in $\theta$ and
$\sum_{k\ge1}\rho^{\,k-1}<\infty$, dominated convergence for series
gives $\sum_{k\geq1}\gamma_k(\theta)\to\sum_{k\geq1}\gamma_k(\theta^\star)$,
and hence
$\sigma^2_{IMCV}(\theta)=\gamma_0(\theta)+2\sum_{k\geq1}\gamma_k(\theta)
\to\sigma^2_{IMCV}(\theta^\star)$.
\end{proof}

\begin{lemma}\label{lem:batchavg}
Assume (A6) and (A7). Let $f$ be an $\mathcal F_{i-1}$-measurable
random function with $|f|\leq c\,\Phi$. Then
\[
    \Big|\sum_{j=2}^{B-1}\E\big[f(Z_{(i,j)})\mid\mathcal F_{i-1}\big]
    -(B-2)\,\E_{\pi\otimes q_{\theta_i}}[f]\Big|
    \;\leq\;C\,c\,\Phi\big(Z_{(i-1,B)}\big) ,
\]
with $C$ not depending on $i$, $c$, $\ell$ or $B$, and for $i=1$ the
same bound holds with $\Phi(Z_{(0,B)})$ replaced by
$\{1+W(X_{(1,1)})\}^2$.
\end{lemma}

\begin{proof}
Let $\mathcal H_i:=\mathcal F_{i-1}\vee\sigma(X_{(i,1)})$. The
parameter $\theta_i$ is $\mathcal F_{i-1}$-measurable, hence
$\mathcal H_i$-measurable, and conditionally on $\mathcal H_i$ the
state $X_{(i,j)}$ is obtained from $X_{(i,1)}$ by $j-1$ transitions
of the single kernel $P_{\theta_i}$. For $ 2\leq j\leq B$, Lemma~\ref{lem:envelope}(i) gives the decomposition
\[
    \Pr\big(X_{(i,j)}\in\cdot\;\big|\;\mathcal H_i\big)
    = P_{\theta_i}^{\,j-1}\big(X_{(i,1)},\cdot\big)
    = \big(1-\rho^{\,j-1}\big)\,\Pi(\cdot)
    + \rho^{\,j-1}\,R_{\theta_i}^{\,j-1}\big(X_{(i,1)},\cdot\big).
\]
Given $\mathcal H_i$ and $X_{(i,j)}$ the proposal $Y_{(i,j)}$ is a
draw from $q_{\theta_i}$, so that
$\E[f(Z_{(i,j)})\mid\mathcal H_i]
=\E[\bar f(X_{(i,j)})\mid\mathcal H_i]$ with
$\bar f(x):=\int f(x,y)q_{\theta_i}(y)\,dy$ and
$\E_\pi[\bar f]=\E_{\pi\otimes q_{\theta_i}}[f]$. Notice that
$\Phi(x,y) = (1+W(x)+W(y))^2\leq(1+W(x))^2(1+W(y))^2$, then define
$C_1:=\sup_\theta\E_{q_\theta}[(1+W)^2]$ is finite by (A7), therefore
\[
    |\bar f(x)|
    \;\leq\;c\,(1+W(x))^2\int(1+W(y))^2q_{\theta_i}(y)\,dy
    \;\leq\;c\,C_1(1+W(x))^2 .
\]
The argument of Lemma~\ref{lem:envelope}(iii) with $G=(1+W)^2$ gives
$P_\theta(1+W)^2\leq C_1+\rho(1+W)^2$, hence
$R_\theta(1+W)^2\leq\rho^{-1}C_1+(1+W)^2$, and since $R_\theta$ is a
Markov kernel, iterating gives
$R^{\,k}_\theta(1+W)^2\leq(1+W)^2+k\rho^{-1}C_1$.

Integrating $\bar f$ against the decomposition, summing over $j$, and
using the finiteness of $\sum_{k\geq1}\rho^k$ and
$\sum_{k\geq1}k\rho^{k}$,
\begin{align*}
    \sum_{j=2}^{B-1}\Big|\E\big[\bar f(X_{(i,j)})\mid\mathcal H_i\big]
    -\E_\pi[\bar f]\Big|
    &=\sum_{j=2}^{B-1}\rho^{\,j-1}
      \Big|R^{\,j-1}_{\theta_i}\bar f\big(X_{(i,1)}\big)
      -\E_\pi[\bar f]\Big|\\
    &\leq c\,C_1\sum_{k\geq1}\rho^{k}
      \Big\{\big(1+W(X_{(i,1)})\big)^2+k\rho^{-1}C_1
      +\E_\pi\big[(1+W)^2\big]\Big\}\\
    &\leq\;C\,c\,\big\{1+W(X_{(i,1)})\big\}^2 ,
\end{align*}
where $\E_\pi[(1+W)^2]<\infty$ by (A7), the last two terms are
constants absorbed using $W\geq1$, and $C$ depends only on $M$,
$C_1$ and $\E_\pi[(1+W)^2]$.

The state $X_{(i,1)}$ equals either $X_{(i-1,B)}$ or $Y_{(i-1,B)}$,
so $\{1+W(X_{(i,1)})\}^2\leq\Phi(Z_{(i-1,B)})$ pointwise. Since
$\theta_i$ is $\mathcal F_{i-1}$-measurable, the quantity
$\E_{\pi\otimes q_{\theta_i}}[f]$ is $\mathcal F_{i-1}$-measurable,
and so is $\Phi(Z_{(i-1,B)})$. Taking conditional expectations given
$\mathcal F_{i-1}$ in the display above and using the law of total expectation
gives the stated bound. For $i=1$ the same argument applies with
$\mathcal H_1$ replaced by $\mathcal G_0$ and
$\Phi(Z_{(i-1,B)})$ by $\{1+W(X_{(1,1)})\}^2$.
\end{proof}

The proof of Theorem~3.2 follows the martingale-approximation
approach to the central limit theorems for adaptive MCMC developed by
\citet{Andrieu2006OnAlgorithms} and, we follow the idea by
\citet[Theorem~2.2]{Fort2014CLTAdaptive}. In this approach the
centred sum is written, through the Poisson equation, as a
martingale plus remainder terms, and the martingale central limit
theorem of \citet[Corollary~3.1]{HallHeyde1980} is applied. Our
setting differs from that of \citet{Fort2014CLTAdaptive} in two
respects. First, the independence sampler satisfies the uniform
minorisation of Lemma~\ref{lem:contr}(i), which replaces their drift
and small-set conditions and makes all constants explicit. Second,
the batch structure of Algorithm~2 keeps the parameter fixed within
each batch, so that the remainder consists of one term per change of
batch rather than one per iteration. This allows the regularity
condition on consecutive Poisson solutions, their assumption A3-a,
to be replaced by the condition $\ell/B\to0$ on the batch length.
The convergence of the conditional variances, which is an assumption
(A4) in \citet{Fort2014CLTAdaptive}, is established here in
Lemma~\ref{lem:wlln}.

The decomposition is as follows. The Poisson equation of
Lemma~\ref{lem:poisson}(i) writes each summand as
$\psi_\theta(Z_t)=\hat h_\theta(Z_t)-\widetilde P_\theta\hat h_\theta(Z_t)$,
and the term $\hat h_\theta(Z_t)$ can be split into its conditional
expectation given the past and a martingale difference. Summing over
$t$, the conditional expectations telescope against the terms
$\widetilde P_\theta\hat h_\theta(Z_{t-1})$ inside each batch, and
what remains is a martingale plus a small number of boundary terms,
i.e. the terms at the change of batches. This reduces the central limit theorem
to a martingale central limit theorem, in which the dependence
between batches and the changing kernels are absorbed by the
filtration. To make this precise, define for $t=(i,j)$
\[
    d_t :=
    \begin{cases}
        \hat h_{\theta_i}(Z_t)
        -\widetilde P_{\theta_i}\hat h_{\theta_i}(Z_{t-1}),
        & j\geq2,\\[2pt]
        \hat h_{\theta_i}(Z_t)
        -\widetilde Q_{i-1}\hat h_{\theta_i}(Z_{t-1}),
        & j=1,\ i\geq2,\\[2pt]
        \hat h_{\theta_1}(Z_t)
        -\E\big[\hat h_{\theta_1}(Z_t)\mid\mathcal G_0\big],
        & t=(1,1),
    \end{cases}
\]
and, for $1\leq i<\ell$, the boundary residuals
\[
    R_i:=\big\{\widetilde Q_i\hat h_{\theta_{i+1}}
    -\widetilde P_{\theta_i}\hat h_{\theta_i}\big\}(Z_{(i,B)}) ,
    \qquad
    \mathcal B_n
    :=\E\big[\hat h_{\theta_1}(Z_{(1,1)})\mid\mathcal G_0\big]
    -\widetilde P_{\theta_\ell}\hat h_{\theta_\ell}(Z_{(\ell,B)})
    +\sum_{i=1}^{\ell-1}R_i .
\]

\begin{lemma}\label{lem:mds}
Assume (A6) and (A7). Then:
\begin{enumerate}
\item[(i)] $\{d_t\}$ is a martingale-difference sequence with
respect to $\{\mathcal G_t\}$, and $\sup_t\E[d_t^4]<\infty$.
\item[(ii)] $\sum_{t}\psi_{\theta_{i(t)}}(Z_t)=\sum_{t}d_t+\mathcal B_n$,
and $\E|\mathcal B_n|\leq C\ell$ for a constant $C$ not depending on
$\ell$ and $B$.
\end{enumerate}
\end{lemma}

\begin{proof}
(i) For $t=(i,j)$ with $j\geq2$, the conditional law of $Z_t$ given
$\mathcal G_{t-1}$ is $\widetilde P_{\theta_i}(Z_{t-1},\cdot)$, and
for $j=1$ it is $\widetilde Q_{i-1}(Z_{t-1},\cdot)$. In both cases
$\theta_i\in\mathcal G_{t-1}$, thus $\hat h_{\theta_i}$ can be
treated as a fixed function, and the subtracted term in the
definition of $d_t$ equals
\[
    \E\big[\hat h_{\theta_i}(Z_t)\mid\mathcal G_{t-1}\big]
    = \int\hat h_{\theta_i}(z)\,\widetilde P_{\theta_i}(Z_{t-1},dz)
    = \widetilde P_{\theta_i}\hat h_{\theta_i}(Z_{t-1})
\]
for $j\geq2$, and likewise
$\widetilde Q_{i-1}\hat h_{\theta_i}(Z_{t-1})$ for $j=1$. At
$t=(1,1)$ the subtracted term is
$\E[\hat h_{\theta_1}(Z_t)\mid\mathcal G_0]$ by definition.
Therefore $\E[d_t\mid\mathcal G_{t-1}]=0$ for every $t$, and $d_t$
is $\mathcal G_t$-measurable.

For the moments, Lemma~\ref{lem:poisson}(i) gives
$|\hat h_\theta(x,y)|\leq C_h\{1+W(x)+W(y)\}$. The kernels
$\widetilde P_\theta$ and $\widetilde Q_i$ accept  $x$ or $y$ for the first coordinate and draw the second from some $q_{\theta'}$, so that
\[
    \big|\widetilde P_\theta\hat h_\theta(x,y)\big|
    \;\leq\;C_h\big\{1+W(x)+W(y)+\E_{q_{\theta'}}[W]\big\}
    \;\leq\;C_h\big\{1+c_0+W(x)+W(y)\big\},
\]
and likewise for $\widetilde Q_i\hat h_{\theta}$ and for the
conditional expectation at $t=(1,1)$. Consequently,
\[
    |d_t|
    \;\leq\;C_h(2+c_0)\big\{1+W(X_t)+W(Y_t)+W(X_{t-1})+W(Y_{t-1})\big\},
\]
where at $t=(1,1)$ the terms $W(X_{t-1})$ and $W(Y_{t-1})$ are
absent. By
convexity of $u\mapsto u^4$, Lemma~\ref{lem:envelope}(iii) and (A7), $\E[d_t^4]$ is bounded by a constant
times $1+\E[W(X_t)^4]+\E[W(Y_t)^4]+\E[W(X_{t-1})^4]+\E[W(Y_{t-1})^4]$,
which gives
$\sup_t\E[d_t^4]<\infty$.

(ii) Write $c_t:=\E[\hat h_{\theta_{i(t)}}(Z_t)\mid\mathcal G_{t-1}]$,
so that $d_t=\hat h_{\theta_{i(t)}}(Z_t)-c_t$. By
Lemma~\ref{lem:poisson}(i),
$\psi_{\theta_{i(t)}}(Z_t)
=\hat h_{\theta_{i(t)}}(Z_t)
-\widetilde P_{\theta_{i(t)}}\hat h_{\theta_{i(t)}}(Z_t)
=d_t+c_t-\widetilde P_{\theta_{i(t)}}\hat h_{\theta_{i(t)}}(Z_t)$.
Summing over $t$ and regrouping the compensators,
\[
    \sum_t\psi_{\theta_{i(t)}}(Z_t)
    = \sum_t d_t
    + c_{(1,1)}
    -\widetilde P_{\theta_\ell}\hat h_{\theta_\ell}(Z_{(\ell,B)})
    +\sum_{t\neq(\ell,B)}
    \Big\{c_{t+1}
    -\widetilde P_{\theta_{i(t)}}\hat h_{\theta_{i(t)}}(Z_t)\Big\}.
\]
For $t=(i,j)$ with $j<B$,
the successor lies in the same batch and
$c_{t+1}=\widetilde P_{\theta_i}\hat h_{\theta_i}(Z_t)$, so the
corresponding summand vanishes. For $t=(i,B)$ with $i<\ell$, the
successor is $(i+1,1)$ and
$c_{t+1}=\widetilde Q_i\hat h_{\theta_{i+1}}(Z_{(i,B)})$, so the
summand equals $R_i$. The right-hand side is therefore
$\sum_td_t+\mathcal B_n$. Finally, the bounds of part (i) give
$|R_i|\leq2C_h(1+c_0)\{1+W(X_{(i,B)})+W(Y_{(i,B)})\}$, and the two
edge terms satisfy bounds of the same form. By
Lemma~\ref{lem:envelope}(iii) and (A7), the expectations of the
absolute values of all $\ell+1$ terms of $\mathcal B_n$ are
uniformly bounded by a constant, and the inequality
$\E|\mathcal B_n|\leq C\ell$ follows.
\end{proof}

\begin{proof}[Proof of Theorem 3.2]

(ergodicity) We check the conditions of Theorem~1 of
\citet{Roberts2007CouplingMCMC} for the family
$\{IM(P_{\theta_i},\pi,q_{\theta_i})\}$ equipped with the adaptation of
Algorithm~2. Simultaneous uniform ergodicity holds because, by (A6) and
\eqref{eq:doeblin}, every kernel obeys the common minorisation
$P_\theta\geq M^{-1}\Pi$, thus
$\Vert P_\theta^k(x,\cdot)-\Pi\Vert_{\mathrm{TV}}\leq(1-1/M)^k$
uniformly in $\theta$ and $x$. Diminishing adaptation holds by (A8).
The adaptive chain is therefore ergodic with target $\pi$.

(law of large numbers) Since
$\psi_\theta=\varphi_\theta-\E_\pi[F]$ for every $\theta$, the
estimator satisfies
$\mu_{\ell,B,IMCV}-\E_\pi[F]=n^{-1}\sum_t\psi_{\theta_{i(t)}}(Z_t)$,
and Lemma~\ref{lem:mds}(ii) gives
\[
    \mu_{\ell,B,IMCV}-\E_\pi[F]
    = \frac1n\sum_t d_t+\frac1n\mathcal B_n .
\]
We show that both terms converge to zero in probability. For the
martingale term, martingale differences are orthogonal in $L^2$:
for $s<t$, $\E[d_sd_t]=\E\big[d_s\,\E[d_t\mid\mathcal G_{t-1}]\big]=0$
because $d_s$ is $\mathcal G_{t-1}$-measurable. Therefore
\[
    \E\Big[\Big(\sum_t d_t\Big)^2\Big]
    = \sum_t\E[d_t^2]
    \;\leq\; n\,\sup_t\E[d_t^2]
    \;\leq\; Cn ,
\]
where $\sup_t\E[d_t^2]\leq\big(\sup_t\E[d_t^4]\big)^{1/2}<\infty$ by
Lemma~\ref{lem:mds}(i). By Chebyshev's inequality, for every
$\eta>0$,
\[
    \Pr\Big(\Big|\frac1n\sum_td_t\Big|>\eta\Big)
    \;\leq\;\frac{\E[(\sum_td_t)^2]}{\eta^2n^2}
    \;\leq\;\frac{C}{\eta^2n}
    \;\longrightarrow\;0
\]
as $n=\ell B\to\infty$. For the boundary term, Markov's inequality
and Lemma~\ref{lem:mds}(ii) give
\[
    \Pr\Big(\Big|\frac1n\mathcal B_n\Big|>\eta\Big)
    \;\leq\;\frac{\E|\mathcal B_n|}{\eta n}
    \;\leq\;\frac{C\ell}{\eta\,\ell B}
    \;=\;\frac{C}{\eta B}
    \;\longrightarrow\;0
\]
as $B\to\infty$. Hence $\mu_{\ell,B,IMCV}\to\E_\pi[F]$ in
probability.

(central limit theorem) By Lemma~\ref{lem:mds}(ii),
\[
    \sqrt n\,\big\{\mu_{\ell,B,IMCV}-\E_\pi[F]\big\}
    =\frac{1}{\sqrt n}\sum_{t}d_t+\frac{1}{\sqrt n}\,\mathcal B_n .
\]
The remainder vanishes in probability, since by Markov's inequality
and Lemma~\ref{lem:mds}(ii), for every $\eta>0$,
\[
    \Pr\Big(\Big|\frac{1}{\sqrt n}\mathcal B_n\Big|>\eta\Big)
    \;\leq\;\frac{\E|\mathcal B_n|}{\eta\sqrt n}
    \;\leq\;\frac{C\ell}{\eta\sqrt{\ell B}}
    \;=\;\frac{C}{\eta}\sqrt{\frac{\ell}{B}}
    \;\longrightarrow\;0
\]
as $\ell/B\to0$. By Slutsky's lemma it therefore suffices to prove the central limit
theorem for $n^{-1/2}\sum_td_t$. The array $\{d_t/\sqrt n\}$ is a square-integrable
martingale-difference array with respect to $\{\mathcal G_t\}$, so by
the central limit theorem for such arrays
\citep[Corollary~3.1 and the remarks following it]{HallHeyde1980}, it suffices to verify the conditional Lindeberg condition and the
convergence of the conditional variances. The nested $\sigma$-field
condition of that corollary is not required here, since the limit
$\sigma^2_{IMCV}(\theta^\star)$ is a deterministic constant and only
distributional convergence is claimed.

\emph{Lindeberg.} Fix $\varepsilon>0$. On the event $\{|d_t|>\varepsilon\sqrt n\}$ we
have $d_t^2/(\varepsilon^2n)>1$ and
therefore
$d_t^2\leq d_t^2\cdot d_t^2/(\varepsilon^2n)=d_t^4/(\varepsilon^2n)$,
so that
\[
    \sum_t\E\Big[\frac{d_t^2}{n}\,
    \mathbb{I}\big\{|d_t|>\varepsilon\sqrt n\big\}\,\Big|\,
    \mathcal G_{t-1}\Big]
    \;\leq\;
    \frac{1}{\varepsilon^2n^2}\sum_t\E\big[d_t^4\mid\mathcal G_{t-1}\big] .
\]
The right-hand side is nonnegative with expectation bounded by
$n\sup_t\E[d_t^4]/(\varepsilon^2n^2)\leq C/(\varepsilon^2 n)$ by
Lemma~\ref{lem:mds}(i), and therefore converges to zero in
probability by Markov's inequality. This is the conditional
Lindeberg condition for the array $\{d_t/\sqrt n\}$.

\emph{Conditional variances.} For $t=(i,j)$ with $j\geq2$, the
conditional law of $Z_t$ given $\mathcal G_{t-1}$ is
$\widetilde P_{\theta_i}(Z_{t-1},\cdot)$, so that
$\E[d_t^2\mid\mathcal G_{t-1}]=S_{\theta_i}(Z_{t-1})
=S_{\theta_i}(Z_{(i,j-1)})$ with $S_\theta$ as in
Lemma~\ref{lem:poisson}(ii). As $j$ ranges over $2,\ldots,B$ the
argument ranges over the first $B-1$ states of batch $i$, and the
parameter and the argument belong to the same batch. At the $\ell$
remaining indices, those with $j=1$, the conditional variance is
bounded by
$\E[d_t^2\mid\mathcal G_{t-1}]\leq\E[d_t^4\mid\mathcal G_{t-1}]^{1/2}$,
whose expectation is bounded uniformly in $t$ by
Lemma~\ref{lem:mds}(i), so that for every $\eta>0$, by Markov's
inequality,
\[
    \Pr\Big(\frac1n\sum_{t:\,j=1}\E\big[d_t^2\mid\mathcal G_{t-1}\big]
    >\eta\Big)
    \;\leq\;\frac{\ell\,C}{\eta\,n}
    \;=\;\frac{C}{\eta\,B}
    \;\longrightarrow\;0 .
\]
It therefore remains to prove
\begin{equation}\label{eq:condvar}
    \frac1n\sum_{i=1}^{\ell}\sum_{j=1}^{B-1}
    S_{\theta_i}\big(Z_{(i,j)}\big)
    \;\xrightarrow{\;P\;}\;
    \sigma^2_{IMCV}(\theta^\star) ,
\end{equation}
which is Lemma~\ref{lem:wlln}. Since the limit is a constant,
\citet[Corollary~3.1]{HallHeyde1980} gives
$n^{-1/2}\sum_td_t\xrightarrow{d}
\mathcal N\big(0,\sigma^2_{IMCV}(\theta^\star)\big)$, and with the
remainder bound above,
$\sqrt n\{\mu_{\ell,B,IMCV}-\E_\pi[F]\}\xrightarrow{d}
\mathcal N\big(0,\sigma^2_{IMCV}(\theta^\star)\big)$.
\end{proof}

\begin{lemma}\label{lem:wlln}
Under (A4)--(A8), the convergence \eqref{eq:condvar}
holds as $\ell\to\infty$, $B\to\infty$.
\end{lemma}

\begin{proof}
Throughout this proof $C$ denotes a finite constant whose value may
change from one occurrence to the next, depending only on $M$, on
$\Vert F\Vert_W$, on the moments in (A7) and on the initial
condition, and never on $i$, $j$, $\theta$, $\ell$ or $B$. By
Lemma~\ref{lem:poisson}(i) and (A7), the bound $|S_\theta|\leq C\Phi$
holds uniformly in $\theta$, and by Lemma~\ref{lem:poisson}(ii),
$\E_{\pi\otimes q_\theta}[S_\theta]=\sigma^2_{IMCV}(\theta)$. Write
\begin{align*}
    A_n
    &:=\frac1n\sum_{i=1}^{\ell}\sum_{j=1}^{B-1}
      S_{\theta_i}\big(Z_{(i,j)}\big)\\
    &\;=\frac1n\sum_{i=1}^{\ell}S_{\theta_i}\big(Z_{(i,1)}\big)
      +\frac1n\sum_{i=1}^{\ell}\Xi_i
      +\frac1n\sum_{i=1}^{\ell}\sum_{j=2}^{B-1}
      \E\big[S_{\theta_i}\big(Z_{(i,j)}\big)\,\big|\,
      \mathcal F_{i-1}\big]
\end{align*}
for the average on the left of \eqref{eq:condvar}, where
\[
    \Xi_i:=\sum_{j=2}^{B-1}\Big\{S_{\theta_i}\big(Z_{(i,j)}\big)
    -\E\big[S_{\theta_i}\big(Z_{(i,j)}\big)\,\big|\,
    \mathcal F_{i-1}\big]\Big\} .
\]
It differs from an ergodic average in two respects. The function
$S_{\theta_i}$ changes from batch to batch, and within a batch the
chain is not started from stationarity. We treat both by
conditioning on $\mathcal F_{i-1}$, under which $\theta_i$ is fixed,
and handle the three terms in turn.

The state $X_{(i,1)}$
is produced by the acceptance step $K_{\theta_{i-1}}$ acting on the
pair $Z_{(i-1,B)}$, which does not satisfy the minorisation of
Lemma~\ref{lem:envelope}(i), so Lemma~\ref{lem:batchavg} is available
only for $j\geq2$. The corresponding $\ell$ terms with $j=1$ are
discarded directly. By $|S_\theta|\leq C\Phi$,
Lemma~\ref{lem:envelope}(iii) and (A7), $\sup_t\E[\Phi(Z_t)]<\infty$,
so that $\E|S_{\theta_i}(Z_{(i,1)})|\leq C$ for every $i$, and
Markov's inequality gives
\[
    \Pr\Big(\frac1n\sum_{i=1}^{\ell}
    \big|S_{\theta_i}\big(Z_{(i,1)}\big)\big|>\eta\Big)
    \;\leq\;\frac{C\ell}{\eta\,n}\;=\;\frac{C}{\eta\,B}
    \;\longrightarrow\;0 .
\]

 For the second term, each $\Xi_i$ is
$\mathcal F_i$-measurable with $\E[\Xi_i\mid\mathcal F_{i-1}]=0$, so
$\{\Xi_i\}$ is a martingale-difference sequence and
$\E[\Xi_i\Xi_{i'}]=0$ for $i\neq i'$. Since conditional expectation
does not increase $\Vert\cdot\Vert_2$, the bound
$|S_\theta|\leq C\Phi$ and Minkowski's inequality give
\[
    \Vert\Xi_i\Vert_2
    \;\leq\;C\sum_{j=2}^{B-1}
    \big\Vert\Phi\big(Z_{(i,j)}\big)\big\Vert_2 ,
\]
and each term is bounded uniformly in $j$ by
Lemma~\ref{lem:envelope}(iii) and (A7), so that
$\E[\Xi_i^2]\leq CB^2$. Then,
$\E[(\sum_i\Xi_i)^2]=\sum_i\E[\Xi_i^2]\leq C\ell B^2$ and
$\E[(n^{-1}\sum_i\Xi_i)^2]\leq C/\ell\rightarrow0$; the geometric
decay within a batch would improve this to $C/n$, which is not
needed.

For the remaining term, $\theta_i$ is $\mathcal F_{i-1}$-measurable,
so Lemma~\ref{lem:batchavg} applies with $f=S_{\theta_i}$ and gives
\[
    \frac1n\sum_{i=1}^{\ell}\sum_{j=2}^{B-1}
    \E\big[S_{\theta_i}\big(Z_{(i,j)}\big)\,\big|\,
    \mathcal F_{i-1}\big]
    = \frac{B-2}{B}\cdot\frac1\ell\sum_{i=1}^{\ell}
      \sigma^2_{IMCV}(\theta_i)+\Delta_n ,
\]
where $\E|\Delta_n|\leq C\ell/n=C/B\rightarrow0$ by
Lemma~\ref{lem:envelope}(iii) and (A7), and $(B-2)/B\rightarrow1$
while the averaged summands are bounded by $C_\sigma$ by
Lemma~\ref{lem:sigmacont}.

Finally, we show that
\begin{equation}\label{eq:cesaro}
    \E\left|\frac1\ell\sum_{i=1}^{\ell}
    \big\{\sigma^2_{IMCV}(\theta_i)-\sigma^2_{IMCV}(\theta^\star)
    \big\}\right|
    \;\longrightarrow\;0 ,
\end{equation}
so that, by Markov's inequality,
$\ell^{-1}\sum_i\sigma^2_{IMCV}(\theta_i)\rightarrow
\sigma^2_{IMCV}(\theta^\star)$ in probability. Write
$V_i:=\sigma^2_{IMCV}(\theta_i)$ and
$V^\star:=\sigma^2_{IMCV}(\theta^\star)$, so that $|V_i|\leq C_\sigma$
for every $i$ and every $B$, and $|V^\star|\leq C_\sigma$, by
Lemma~\ref{lem:sigmacont}.

Let $\epsilon,\delta>0$. By the continuity of
$\theta\mapsto\sigma^2_{IMCV}(\theta)$ at $\theta^\star$ there is
$\eta>0$ such that $\Vert\theta-\theta^\star\Vert\leq\eta$ implies
$|\sigma^2_{IMCV}(\theta)-\sigma^2_{IMCV}(\theta^\star)|\leq\epsilon$.
By (A4) the convergence $\theta_i\rightarrow\theta^\star$ in
probability holds uniformly in the number $B$ of samples per
adaptation step, so there is $i_0=i_0(\eta,\delta)$, not depending on
$B$ and hence not on $\ell$, with
$\Pr\{\Vert\theta_i-\theta^\star\Vert>\eta\}\leq\delta$ for all
$i\geq i_0$. Consequently
$\Pr\{|V_i-V^\star|>\epsilon\}\leq\delta$ for $i\geq i_0$. That $i_0$
does not depend on $B$ is what makes the argument valid along the
joint limit $\ell\rightarrow\infty$, $B\rightarrow\infty$, in which
$\{\theta_i\}=\{\theta^{(B)}_i\}$ forms a triangular array.

Splitting on the event $\{|V_i-V^\star|>\epsilon\}$ and bounding the
integrand by $2C_\sigma$ there,
$\E|V_i-V^\star|\leq\epsilon+2C_\sigma\delta$ for $i\geq i_0$, while
$\E|V_i-V^\star|\leq2C_\sigma$ for the finitely many $i<i_0$.
Therefore
\[
    \E\left|\frac1\ell\sum_{i=1}^{\ell}(V_i-V^\star)\right|
    \;\leq\;\frac1\ell\sum_{i=1}^{\ell}\E|V_i-V^\star|
    \;\leq\;\frac{2C_\sigma(i_0-1)}{\ell}
    +\epsilon+2C_\sigma\delta ,
\]
and letting $\ell\rightarrow\infty$ and then
$\epsilon,\delta\downarrow0$ gives \eqref{eq:cesaro}. Only the
uniform bound $|V_i|\leq C_\sigma$ is used here; a union bound over
the $\ell$ terms would contribute $\ell\delta$ and would not suffice.

Collecting the estimates, $\E[(n^{-1}\sum_i\Xi_i)^2]\leq C/\ell
\rightarrow0$, $\E|\Delta_n|\leq C/B\rightarrow0$,
$(B-2)/B\rightarrow1$, and the $\ell$ discarded terms
$S_{\theta_i}(Z_{(i,1)})$ are negligible. Together with
\eqref{eq:cesaro} these give \eqref{eq:condvar}.
\end{proof}

\section{Further variance reduction with optimal coefficients}\label{sec:coef}

Consider estimator (12). The correlation between the function $F$ and
its control variate is not perfect, so we can introduce coefficients
$c_1,c_2$ in front of the control variate terms to further reduce the
estimator variance:
\[
    \mu_{n,IMCV}^{(c_1,c_2)}
    = \frac{1}{n}\sum_{i=1}^{n}
    \Big\{F(X_i)+c_1\big\{\alpha(X_i,Y_i)(F(Y_i)-F(X_i))
    -c_2\,(F(Y_i)-\E_q[F])\big\}\Big\} .
\]
Write $V_1(x,y):=\alpha(x,y)\{F(y)-F(x)\}$ and
$V_2(y):=F(y)-\E_q[F]$, and recall from Lemma~\ref{lem:aug}(iii)
that $\E_{\pi\otimes q}[V_1]=\E_{\pi\otimes q}[V_2]=0$ exactly. The estimator is therefore unbiased in
the stationary regime whatever $(c_1,c_2)$, and the coefficients only
affect its variance.

The coefficient $c_2$ is chosen as the least-squares projection of
$V_1$ on $V_2$ in the stationary law, that is,
\[
    c_2^\star
    := \frac{\mathrm{Cov}_{\pi\otimes q}(V_1,V_2)}
            {\mathrm{Var}_{\pi\otimes q}(V_2)},
\]
estimated by
\begin{equation}\label{c2}
    \hat{c}_2
    = \frac{\sum_{i=1}^n\alpha(X_i,Y_i)(F(Y_i)-F(X_i))
    (F(Y_i)-\E_q[F])}{\sum_{i=1}^n(F(Y_i)-\E_q[F])^2} ,
\end{equation}
see, for example, \cite{owen2013monte}. Given $c_2^\star$, define the
sampled control variate
\[
    H(x,y)
    := F(x)+\alpha(x,y)\{F(y)-F(x)\}
    -c_2^\star\,\{F(y)-\E_q[F]\},
\]
which satisfies $\E_q[H(x,Y)]=PF(x)$, so that $H$ is a one-sample
unbiased estimate of $PF(x)$ based on the proposal drawn at $x$. When
$PF$ is available exactly, the coefficient minimising the asymptotic
variance of the estimator built from $F+c(PF-F)$ over a reversible
chain is given by Theorem~3 of \cite{dellaportas2012control}. Our
setting differs in that $PF$ is replaced by the sampled $H$, so that
the relevant state process is the augmented chain $Z_i=(X_i,Y_i)$,
which is not reversible, we therefore do not claim optimality and
define the coefficient directly as the probability limit of its
estimator,
\[
    c_1^\star
    := \frac{\E_\pi\big[(F-\E_\pi[F])\,(F+PF)\big]}
            {\E\big[\{F(X_1)-H(X_0,Y_0)\}^2\big]} ,
\]
where the expectation in the denominator is taken in stationarity,
$(X_0,Y_0)\sim\pi\otimes q$ and $X_1\sim K((X_0,Y_0),\cdot)$. The coefficient
is estimated by
\begin{equation}\label{c1}
    \hat{c}_1
    =\frac{\sum_{i=1}^{n}\{F(X_i)(F(X_i)+H_i)\}
    -n^{-1}\sum_{i=1}^{n}F(X_i)\sum_{i=1}^{n}(F(X_i)+H_i)}
    {\sum_{i=2}^{n}\{F(X_i)-H_{i-1}\}^2},
\end{equation}
where $H_i:=F(X_i)+\alpha(X_i,Y_i)(F(Y_i)-F(X_i))
-\hat c_2(F(Y_i)-\E_q[F])$ uses \eqref{c2}. Thus, we obtain the
estimator with coefficients \eqref{c1} and \eqref{c2}:
\begin{equation}\label{imcv_coef}
    \mu_{n,IMCV}^{(\hat{c}_1,\hat{c}_2)}
    = \frac{1}{n}\sum_{i=1}^{n}
    \Big\{F(X_i)+\hat c_1\big\{\alpha(X_i,Y_i)(F(Y_i)-F(X_i))
    -\hat c_2\,(F(Y_i)-\E_q[F])\big\}\Big\} .
\end{equation}

Also, for IM with coupling (25), we can add a coefficient $\hat c$ in
front of the control variate. The estimation of $\hat c$ is analogous
to \eqref{c1} with $H_i$ replaced by $F(X_i)-(F(Y_{i-1})-\E_q[F])$.
The IM estimator with coupling becomes
\begin{equation}\label{est:cimcoeff}
    \mu_{n,CIM}^{(\hat{c})}
    = \frac{1}{n}\sum_{i=1}^{n}
    \big\{F(X_i)-\hat{c}\,(F(Y_{i-1})-\E_q[F])\big\} .
\end{equation}
For the ISCV estimator (27), the coefficient can be obtained using a
method analogous to \eqref{c2}. Then the ISCV estimator becomes
\begin{equation}\label{eq:ISCVcoef}
    \mu_{n,ISCV}^{(\hat{c})}
    = \frac{\sum_{i=1}^n F(X_i)w(X_i)}{\sum_{i=1}^n w(X_i)}
    - \frac{\hat{c}}{n}\sum_{i=1}^n\big(F(X_i)-\E_q[F]\big) .
\end{equation}
Moreover, the coefficients for estimator (13) can be obtained by
replacing $F(y)-\E_q[F]$ by $\tilde F(y)-\E_q[\tilde F]$ in
\eqref{c1} and \eqref{c2}.

Similarly, for the batch estimator (23), we can reduce the variance by
adding coefficients for each batch:
\begin{equation}\label{eq:batchimcvcoef}
\scalebox{0.90}{$
    \mu_{\ell,B,IMCV}^{(\hat{c}_1,\hat{c}_2)}
    = \frac{1}{\ell}\sum_{i=1}^{\ell}
    \Big\{\frac{1}{B}\sum_{j=1}^{B}
    \Big\{F(X_{ij})+\hat{c}_{1i}
    \big[\alpha(X_{ij},Y_{ij})(F(Y_{ij})-F(X_{ij}))
    -\hat{c}_{2i}\,(F(Y_{ij})-\E_{q_{\theta_i}}[F])\big]\Big\}\Big\}
$}
\end{equation}
with
\[
    \hat{c}_{1i}
    =\frac{\sum_{j=1}^B\{F(X_{ij})(F(X_{ij})+H_{ij})\}
    -\frac{1}{B}\sum_{j=1}^B F(X_{ij})
    \sum_{j=1}^B\{F(X_{ij})+H_{ij}\}}
    {\sum_{j=2}^{B}\{F(X_{ij})-H_{i(j-1)}\}^2}
\]
and
\[
    \hat{c}_{2i}
    = \frac{\sum_{j=1}^B\alpha(X_{ij},Y_{ij})(F(Y_{ij})-F(X_{ij}))
    (F(Y_{ij})-\E_{q_{\theta_i}}[F])}
    {\sum_{j=1}^B(F(Y_{ij})-\E_{q_{\theta_i}}[F])^2},
\]
where $H_{ij}$ is defined as $H_i$ above with proposal
$q_{\theta_i}$. A similar strategy is applicable to the batch version
of the ISCV estimator (28):
\begin{equation}\label{eq:aiscvcoef}
    \mu_{\ell,B,ISCV}^{(\hat{c})}
    = \frac{1}{\ell}\sum_{i=1}^{\ell}
    \Bigg\{\frac{\sum_{j=1}^B F(X_{ij})\,w_i(X_{ij})}
    {\sum_{j=1}^B w_i(X_{ij})}
    - \frac{\hat{c}_i}{B}\sum_{j=1}^B
    \big(F(X_{ij})-\E_{q_{\theta_i}}[F]\big)\Bigg\} .
\end{equation}

Within each batch the proposal is held fixed at $q_{\theta_i}$, so
for each fixed $\theta$ Theorem~\ref{thm:coef}(i) applies to the
sampler $IM(P,\pi,q_\theta)$ and gives the corresponding limits
$c_1^\star(\theta)$, $c_2^\star(\theta)$. We do not pursue the joint
limit in $\ell$ and $B$ for the coefficients; the theory of
Section~\ref{thmp3} concerns the estimator (23) without coefficients,
and the batch versions above are used in the numerical illustrations
as practical refinements.

The convergence of the coefficient estimators is described in the
following theorem. Note that $V_2$ depends on $y$ alone, so that
$\mathrm{Var}_{\pi\otimes q}(V_2)=\mathrm{Var}_q(F)$, and the two
nondegeneracy conditions below are exactly the requirement that the
denominators of $c_2^\star$ and $c_1^\star$ be positive.

\begin{theorem}\label{thm:coef}
Assume (A1)--(A3), $\mathrm{Var}_q(F)>0$ and
$\E[\{F(X_1)-H(X_0,Y_0)\}^2]>0$. Then, as $n\rightarrow\infty$,
\begin{enumerate}
\item[(i)] $(\hat c_1,\hat c_2)\rightarrow(c_1^\star,c_2^\star)$
almost surely;
\item[(ii)] $\sqrt n\{\mu^{(\hat c_1,\hat c_2)}_{n,IMCV}-\E_\pi[F]\}$
and $\sqrt n\{\mu^{(c_1^\star,c_2^\star)}_{n,IMCV}-\E_\pi[F]\}$ have
the same limiting normal distribution.
\end{enumerate}
\end{theorem}

\begin{proof}
(i) Every numerator and denominator in \eqref{c2} and \eqref{c1} is
an ergodic average of a fixed function of the consecutive pair
$(Z_{i-1},Z_i)$. By Lemma~\ref{lem:aug}(ii) the augmented chain
$\{Z_i\}$ is uniformly ergodic, so is $\{(Z_{i-1},Z_i)\}$, and
its stationary law is $(\pi\otimes q)\otimes\widetilde P$. Since
$0\leq\alpha\leq1$, each of the functions appearing in \eqref{c2} and
\eqref{c1} is a product of at most two factors bounded by
$C\{|\E_q[F]|+|F(x)|+|F(y)|\}$, so each is integrable under that law
because $F$ has four moments under $\pi$ and under $q$ by (A2). The
strong LLN for Harris chains
\citep[Theorem~17.0.1(i)]{Meyn2009MarkovEdition} applies to each
average and identifies its limit with the corresponding stationary
moment.

For \eqref{c2} this gives
$\hat c_2\rightarrow\mathrm{Cov}_{\pi\otimes q}(V_1,V_2)/
\mathrm{Var}_q(F)=c_2^\star$ almost surely, the denominator limit
being positive by assumption. The quantity $H_i$ in \eqref{c1}
uses $\hat c_2$ rather than $c_2^\star$; writing
$H_i=H(Z_i)+(c_2^\star-\hat c_2)V_2(Y_i)$, the numerator and
denominator of \eqref{c1} differ from their counterparts with
$c_2^\star$ by terms bounded by
$|c_2^\star-\hat c_2|$ times ergodic averages that converge almost
surely, and these differences vanish almost surely. Applying the
strong LLN to the averages with $c_2^\star$ in place of $\hat c_2$,
the numerator of \eqref{c1} converges to
$\E_\pi[(F-\E_\pi[F])(F+PF)]$, where $\E_q[H(x,Y)]=PF(x)$ has been
used to integrate out the proposal, and the denominator converges to
$\E[\{F(X_1)-H(X_0,Y_0)\}^2]$, positive by assumption. Hence
$\hat c_1\rightarrow c_1^\star$ almost surely.

(ii) The estimator is linear in $(c_1,c_1c_2)$, so that
\[
    \mu^{(\hat c_1,\hat c_2)}_{n,IMCV}
    -\mu^{(c_1^\star,c_2^\star)}_{n,IMCV}
    = (\hat c_1-c_1^\star)\,\bar V_1
    -(\hat c_1\hat c_2-c_1^\star c_2^\star)\,\bar V_2 ,
    \qquad
    \bar V_k:=\frac1n\sum_{i=1}^{n}V_k(Z_i) .
\]
Both $V_1$ and $V_2$ are centred under $\pi\otimes q$ by
Lemma~\ref{lem:aug}(iii) and lie in $L^2(\pi\otimes q)$ under (A2),
so the argument of Proposition~2.2 applied to $V_k$ in place of
$\varphi_0$ shows that $\sqrt n\,\bar V_k$ converges in distribution
to a centred normal density. In particular, for every $\epsilon>0$ there
is $A<\infty$ with
$\sup_n\Pr\{|\sqrt n\,\bar V_k|>A\}\leq\epsilon$. By part (i) and the
continuity of $(a,b)\mapsto(a,ab)$, the differences
$\hat c_1-c_1^\star$ and $\hat c_1\hat c_2-c_1^\star c_2^\star$
converge to zero almost surely. Hence, for every $\eta>0$,
\[
    \Pr\Big(\sqrt n\,\big|\mu^{(\hat c_1,\hat c_2)}_{n,IMCV}
    -\mu^{(c_1^\star,c_2^\star)}_{n,IMCV}\big|>\eta\Big)
    \;\longrightarrow\;0 ,
\]
since each product is a null sequence times a sequence that is
bounded in probability. Slutsky's theorem then gives the claim.
\end{proof}

\section{Further experiments}
\subsection{Comparison with importance sampler with control variates}\label{sec:iscv}
In this experiment, we assume that the target density $\pi(x)$ is the standard normal $\mathcal{N}(x|0,1)$ and the quantity of interest is the expected value of the sigmoid function $F(x) = 1/(1+\exp\{-x\})$. The proposal density $q(y)$ is  $\mathcal{N}(y|0.1,1.01)$. Since the target density is symmetric, the true expected value of $F$ is known to be $0.5$. To approximate this expectation under the proposal, we approximate $F$ by $\tilde{F} = \Phi(\sqrt{\pi/8}x)$, where $\Phi(\cdot)$ denotes the standard normal CDF. This approximation admits a closed-form expectation $\Phi(0.1/\sqrt{1+(8/\pi)\times1.01})$.  We perform $T = 100$ independent replications across varying sample sizes to compare the mean squared error (MSE) of the estimators defined in equations \eqref{imcv_coef} and \eqref{eq:ISCVcoef}.
Note that both algorithms were evaluated using exactly the same set of samples.

\begin{figure}[!ht]
\centering

\includegraphics[width=0.8\textwidth]{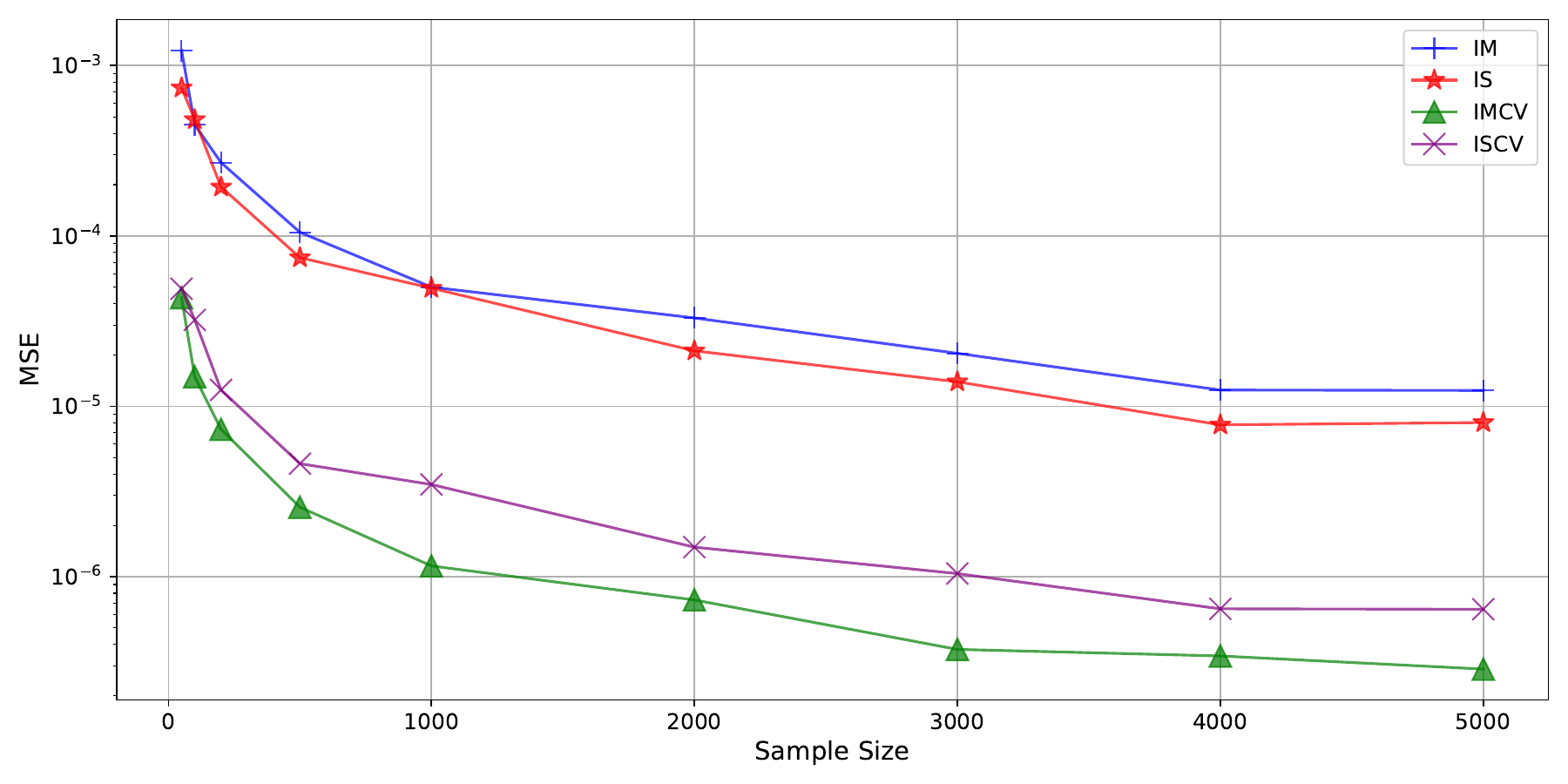}
\caption{Comparison of the MSE of the estimators $\mu_{n,IMCV}^{(\hat{c}_1,\hat{c}_2)}$ and $\mu_{n,ISCV}^{(\hat{c})}$ for different sample sizes. The values are shown on a logarithmic scale.}
\label{fig:mse}

\end{figure}
Fig.~\ref{fig:mse} illustrates that  the IS estimator exhibits a lower MSE compared to the IM estimator. The increased variance in IM is primarily attributed to the positive correlation among its dependent samples. When control variates are applied to both IM and IS, a noticeable reduction in variance is observed in each case. Notably, the IMCV method exhibits both lower bias and lower variance compared to ISCV. Furthermore, as the sample size grows, the mean squared error (MSE) of IMCV decreases more rapidly than that of ISCV.

\subsection{Comparison with adaptive importance sampler with control variates}\label{sec:aiscv}
We consider an adaptation algorithm applied to a two-dimensional problem. The target is chosen to be standard Gaussian $\mathcal{N}(x|\boldsymbol{0},I_2)$, while the initial proposal is set to $\mathcal{N}(x|\mu,\Sigma)$, where $\mu = (0.2,0.2)^\top$ and $\Sigma = ((1.2,0.2),(0.2,1.2))$. The quantity of interest is the expected value of the function $F(x) = 1/(1+\exp\{x_1 + x_2\})$. Due to the symmetry of the target, the true expected value is 0.5. We implement Algorithm~$\ref{alg:sl}$ from Section~\ref{sec:adp-algo} to adapt the proposal and perform 100 independent replications to compare the MSE of the estimators defined in equations \eqref{eq:batchimcvcoef} and \eqref{eq:aiscvcoef}.  Both algorithms were evaluated using exactly the same set of samples.

\begin{figure}[!ht]
\centering

\subfigure[]{
    \includegraphics[width=0.8\textwidth]{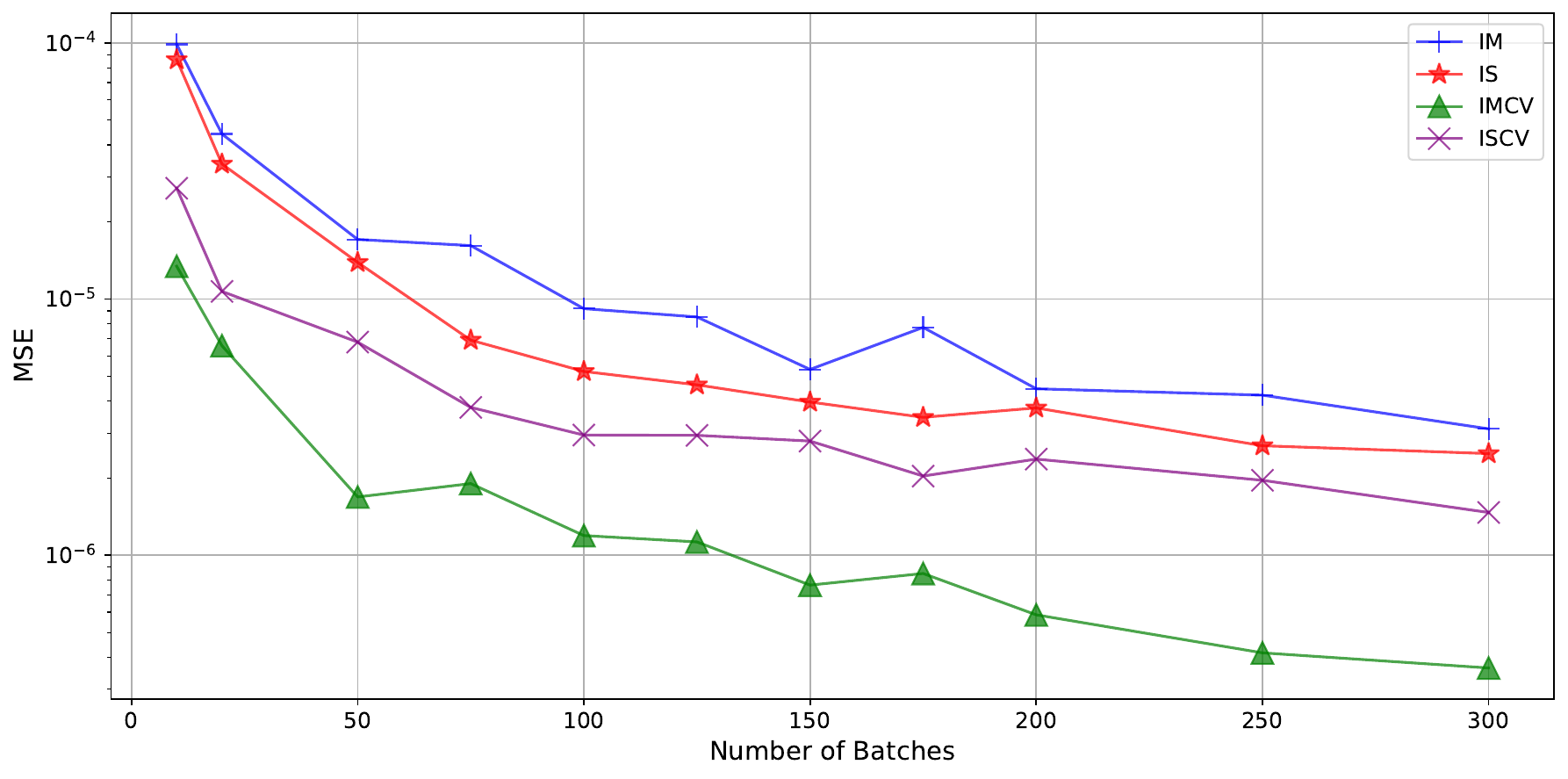}
    \label{fig:aiscva}
}
\subfigure[]{
    \includegraphics[width=0.8\textwidth]{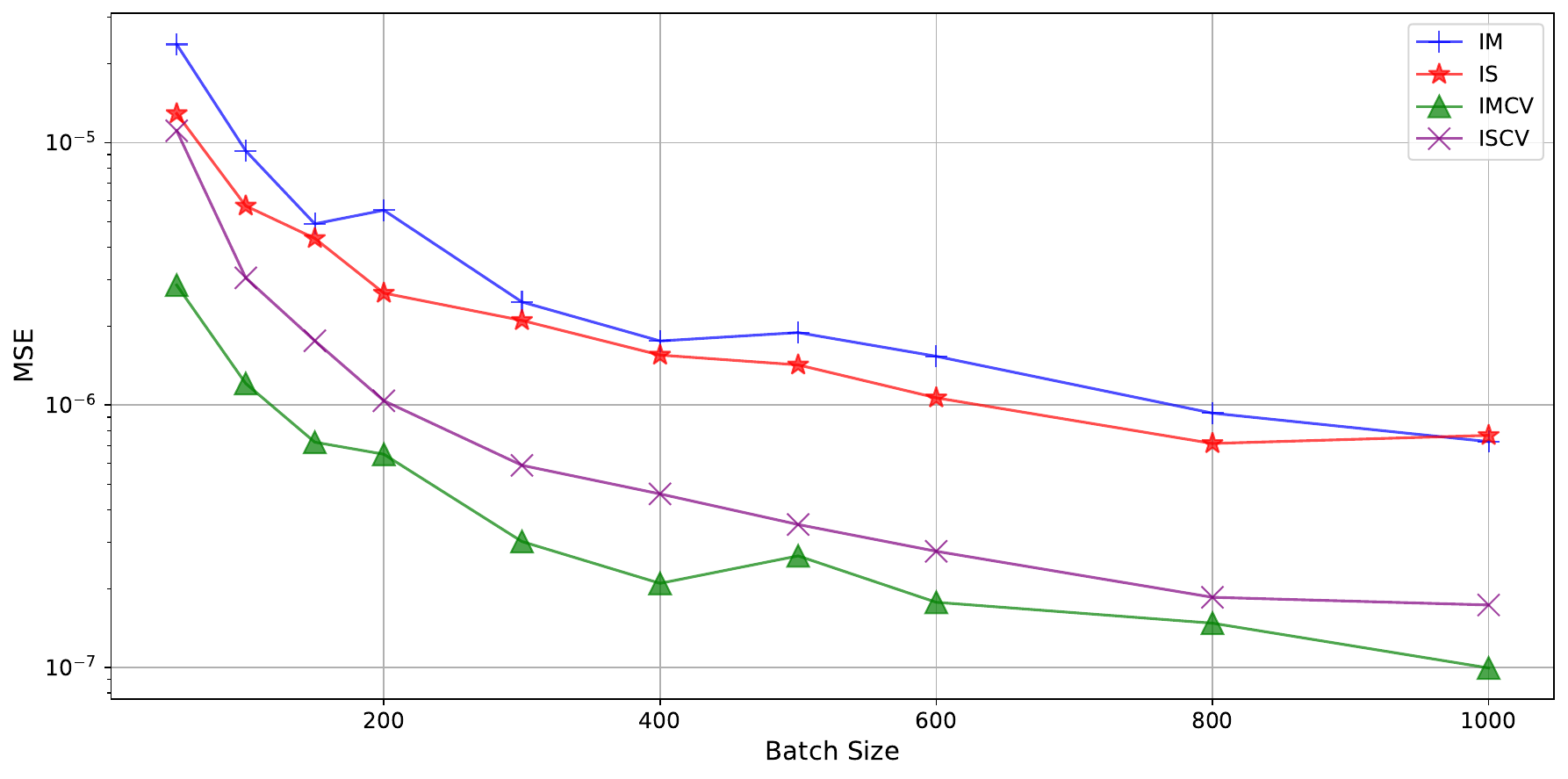}
    \label{fig:aiscvb}
}
\caption{Comparison of the MSE of the estimators $\mu_{\ell,B,IMCV}^{(\hat{c}_1,\hat{c}_2)}$ and $\mu_{\ell,B,ISCV}^{(\hat{c})}$ for different number of adaptation batches with fixed batch size $B = 100$ (top) and different batch size with fixed number of batches $\ell = 100$ (bottom). The values are shown on a logarithmic scale.}
\label{fig:aimcv_vs_aiscv}

\end{figure}
Fig.~\ref{fig:aimcv_vs_aiscv} shows that although IS is preferable to IM,  the IMCV estimator exhibits a lower MSE within each batch, which consequently leads to a lower overall MSE across the entire adaptation trajectory. Furthermore, as the number of batches and the batch size increase, the MSE of IMCV decreases more rapidly than that of the other estimators.

\section{Analytics for examples}
\subsection{Analytical calculations for Section 2.3}\label{ap:1dn}
We derive the theoretical bounds plotted in Section 2.3, synthetic example part. In both
examples $F(x)=x$, so $\sigma_F=1$, and we bound $\sigma^2_{IMCV}$ by
the minimum of \eqref{eq:sigbound} and \eqref{eq:sigbound2},
\begin{equation}\label{eq:minbound}
    \sigma^2_{IMCV}
    \;\leq\;
    \min\big\{(1+2w^\star)\,v,\;
    v+2(w^\star-1)\sqrt{v}\,\big\},
    \qquad v=\mathrm{Var}_{\pi\otimes q}(\varphi),
\end{equation}
so that it remains to compute $w^\star$ and $v$ for each example.
Since $\E_{\pi\otimes q}[\alpha(X,Y)\{F(Y)-F(X)\}]=0$ by
Lemma~\ref{lem:aug}(iii) and $\E_\pi[X]=\E_q[Y]=0$, the
representation \eqref{eq:phialt} gives
$\E_{\pi\otimes q}\big[\{1-\alpha(X,Y)\}\{F(X)-F(Y)\}\big]=0$ and
hence
\begin{equation}\label{eq:v-toy}
    v \;=\;
    \E_{\pi\otimes q}\big[\{1-\alpha(X,Y)\}^2\,(Y-X)^2\big] .
\end{equation}
\begin{itemize}
\item
The target density is $\pi(x)=\mathcal N(x\,|\,0,1)$ and the proposal
density is $q(x)=\mathcal N(x\,|\,0,\sigma^2)$ with $\sigma>1$. Then
\[
    w_x = \frac{\pi(x)}{q(x)}
    = \sigma\exp\Big\{-\frac12\Big(1-\frac{1}{\sigma^2}\Big)x^2\Big\},
    \qquad
    \alpha(x,y)
    = \min\Big(1,\,
    \exp\Big\{-\frac12\Big(1-\frac{1}{\sigma^2}\Big)(y^2-x^2)\Big\}\Big).
\]
Since $\sigma>1$, the exponent is strictly decreasing in $x^2$, so
$w^\star=\sigma$, and $\alpha(x,y)<1$ exactly on $\{|y|>|x|\}$.
Restricting \eqref{eq:v-toy} to this set and using the symmetry of
the integrand under $(x,y)\mapsto(-x,-y)$,
\begin{equation}\label{1dnormalv}
    v = \frac{1}{\pi\sigma}
    \int_0^{\infty}\!\!\int_{-y}^{y}
    \Big(\exp\Big\{-\frac12\Big(1-\frac{1}{\sigma^2}\Big)(y^2-x^2)\Big\}
    -1\Big)^{\!2}(y-x)^2\,
    e^{-\frac{x^2}{2}-\frac{y^2}{2\sigma^2}}
    \,dx\,dy ,
\end{equation}
where the prefactor collects
$\pi(x)q(y)=e^{-x^2/2-y^2/(2\sigma^2)}/(2\pi\sigma)$ and the factor
$2$ from the symmetry reduction. In Section 2.3.1 we evaluate
\eqref{1dnormalv} by numerical integration and plot the bound
\eqref{eq:minbound} with $w^\star=\sigma$; for the values of
$\sigma$ considered there, $\sqrt{v}>1-1/\sigma$, so the minimum in
\eqref{eq:minbound} is attained by \eqref{eq:sigbound2}.
\item
The target density is $\pi(x)=\mathcal N(x\,|\,0,1)$ and the proposal
density is a Student-t distribution with $\nu$ degrees of freedom,
$q(x)=t_\nu(x)$, with p.d.f.\ given by
\[
    t_{\nu}(x)
    = \frac{\Gamma(\frac{\nu+1}{2})}{\sqrt{\nu\pi}\,
    \Gamma(\frac{\nu}{2})}
    \Big(1+\frac{x^2}{\nu}\Big)^{-\frac{\nu+1}{2}} ,
\]
where $\Gamma(\cdot)$ denotes the gamma function. Then
\[
    w_x = \frac{\pi(x)}{q(x)}
    = \sqrt{\frac{\nu}{2}}\,
    \frac{\Gamma(\frac{\nu}{2})}{\Gamma(\frac{\nu+1}{2})}
    \Big(1+\frac{x^2}{\nu}\Big)^{\frac{\nu+1}{2}}
    \exp\Big\{-\frac{x^2}{2}\Big\}.
\]
To find $w^{\star}$, we solve $dw_x/dx=0$ and check
$d^2w_x/dx^2<0$, resulting in
\[
    w^{\star} = w_{x=1}
    = \sqrt{\frac{\nu}{2}}\,
    \frac{\Gamma(\frac{\nu}{2})}{\Gamma(\frac{\nu+1}{2})}
    \Big(1+\frac{1}{\nu}\Big)^{\frac{\nu+1}{2}}
    \exp\Big\{-\frac{1}{2}\Big\}.
\]
Moreover,
\[
    \alpha(x,y)
    = \min\Big(1,\;
    \Big(\frac{\nu+y^2}{\nu+x^2}\Big)^{\frac{\nu+1}{2}}
    \exp\Big\{-\frac12(y^2-x^2)\Big\}\Big).
\]
The variance \eqref{eq:v-toy} is estimated by Monte Carlo as
\begin{equation}\label{1dtv}
    v \;\approx\;
    \frac{1}{N}\sum_{i=1}^{N}
    \big\{1-\alpha(X_i,Y_i)\big\}^2\,(Y_i-X_i)^2 ,
\end{equation}
where $\{X_i\}_{i=1}^N$ are samples from $\pi$ and
$\{Y_i\}_{i=1}^N$ are independent samples from $q$, and the bound
\eqref{eq:minbound} is then evaluated with the closed-form
$w^\star$ above; for the values of $\nu$ considered in Section
2.3, the minimum is again attained by \eqref{eq:sigbound2}.
\end{itemize}

\subsection{Integrals for the marginal likelihood estimation of Section 2.3}

From (18), the prior is
$$
    \beta_m | g, m \sim \mathcal{N}(0,
    g(\boldsymbol{X}_m^\top \boldsymbol{X}_m)^{-1}),~~
    p(g) = (1+g)^{-2}, ~~g>0.
$$
    Assume $\beta_m$ is a $d$-dimensional vector. The marginal likelihood (17) is
    \begin{align}\label{truele}
        &f(y | m)
        = \int f(y | m, \beta_m) 
        f(\beta_m | m) d \beta_m  \nonumber\\ 
        & = \int f(y | m, \beta_m) 
        \mathcal{N}(0,
        g(\boldsymbol{X}_m^\top \boldsymbol{X}_m)^{-1})p(g) d \beta_m dg \nonumber\\ 
        & = \int (2\pi \sigma^2)^{-\frac{N}{2}}\exp \left\{-\frac{1}{2\sigma^2}\sum_{i=1}^{N}(y_i-x_i^\top \beta_m)^2\right \}\mathcal{N}(0,
        g(\boldsymbol{X}_m^\top \boldsymbol{X}_m)^{-1})p(g)dgd\beta_m \nonumber \\
        & = C \int p(g) \int (2\pi g)^{-\frac{d}{2}} \Big \vert(\boldsymbol{X}_m^\top \boldsymbol{X}_m)^{-1}\Big \vert^{-\frac{1}{2}} \exp\left\{\frac{1}{\sigma^2}\sum_{i=1}^N y_i x_i^\top\beta_m - \frac{1}{2}(\frac{1}{\sigma^2}+\frac{1}{g})\beta^\top(\boldsymbol{X}_m^\top \boldsymbol{X}_m)\beta_m  \right\}d\beta_m dg \nonumber \\
        & = C \int (1+\frac{g}{\sigma^2})^{-\frac{d}{2}}\exp\left\{\frac{1}{2\sigma^2}\frac{g}{g+\sigma^2}\Big(\sum_{i=1}^Ny_ix_i^\top\Big)(\boldsymbol{X}_m^\top \boldsymbol{X}_m)^{-1} \Big(\sum_{i=1}^Ny_ix_i^\top\Big)^\top\right\}(1+g)^{-2}dg
    \end{align}
    where $C = (2\pi \sigma^2)^{-\frac{N}{2}} \exp\{-\frac{1}{2\sigma^2}\sum_{i=1}^N y_i^2\}$ is a constant. Notice that \eqref{truele} is a univariate integral, so we use numerical integration to estimate it.
    Furthermore, in the static control variate term in (12), we need to compute the analytical expectation under the proposal density $q$. In this example, the proposal $q$ is a discrete mixture of normals; see (19) and (20). Therefore,
    \begin{align}
        &\E_{q(\beta_m) }[f(y | m, \beta_m) ]
        = \int f(y | m, \beta_m) 
        q(\beta_m) d \beta_m  \nonumber\\ 
        & = \sum_{i=1}^K w_i\int f(y | m, \beta_m) \mathcal{N}(0, g_i (\boldsymbol{X}_m^\top \boldsymbol{X}_m)^{-1}) d \beta_m  \nonumber\\ 
        & = \sum_{i=1}^K w_i\int (2\pi \sigma^2)^{-\frac{N}{2}}\exp \left\{-\frac{1}{2\sigma^2}\sum_{i=1}^{N}(y_i-x_i^\top \beta_m)^2 \right\} \mathcal{N}(0, g_i (\boldsymbol{X}^\top \boldsymbol{X})^{-1}) d \beta_m  \nonumber\\ 
        & = C \sum_{i=1}^K w_i\int (2\pi g_i)^{-\frac{d}{2}} \Big \vert(\boldsymbol{X}_m^\top \boldsymbol{X}_m)^{-1}\Big \vert^{-\frac{1}{2}} \exp\left\{\frac{1}{\sigma^2}\sum_{i=1}^N y_i x_i^\top\beta_m - \frac{1}{2}(\frac{1}{\sigma^2}+\frac{1}{g_i})\beta_m^\top(\boldsymbol{X}_m^\top \boldsymbol{X}_m)\beta_m  \right\}d\beta_m \nonumber \\
        & = C \sum_{i=1}^K w_i \Big(1+\frac{g_i}{\sigma^2}\Big)^{-\frac{d}{2}}\exp\left\{\frac{1}{2\sigma^2}\frac{g_i}{g_i+\sigma^2}\Big(\sum_{i=1}^Ny_ix_i\Big)^\top(\boldsymbol{X}_m^\top \boldsymbol{X}_m)^{-1} \Big(\sum_{i=1}^Ny_ix_i\Big)\right\}
        \label{eq:bayesianRegress_StaticInt}
    \end{align}
Moreover, to estimate the MCMC ratio, we need to calculate the pseudo-marginal prior by Monte Carlo with sample size $L$ (in our experiment set $L$ to 100) during each MCMC iteration,
\begin{align*}
    f(\beta_m | m) &= \int_0^\infty \mathcal{N}(0,
    g(\boldsymbol{X}_m^\top \boldsymbol{X}_m)^{-1})p(g)dg\\
    & \approx \frac{1}{L}\sum_{i=1}^L \mathcal{N}(0,
    g_i(\boldsymbol{X}_m^\top \boldsymbol{X}_m)^{-1}), ~~g_i \sim p(g).
\end{align*}

\subsubsection{Computational cost} The computational cost for crude MC and MCMC estimators is as follows. The matrix multiplication $\boldsymbol{X}_m^\top \boldsymbol{X}_m$ costs $O(N d^2)$ while the Cholesky decomposition of this matrix costs $O(d^3)$. Both these operations are common to both estimators and are needed to be preformed once at the beginning. Then given these two precomputations, sampling from the Gaussian  
$\mathcal{N}(0,
g(\boldsymbol{X}_m^\top \boldsymbol{X}_m)^{-1})$
costs $O(d^2)$ while the likelihood evaluation $f(y | m, \beta_m)$  costs $O(N d)$ (i.e.\ for computing the quadratic form 
$\sum_{i=1}^{N}(y_i-x_i^\top \beta_m)^2$). This implies that the cost of the crude MC estimator (given the precomputations of $\boldsymbol{X}_m^\top \boldsymbol{X}_m$) is $O(n N d + n d^2)$ where $n$ is the number of samples. This cost for large number of data $N$ is dominated by  $O(n N d)$.    

For the proposed MCMC estimator the evaluation of
the expectation \eqref{eq:bayesianRegress_StaticInt} in the static control variable has small cost 
$O(K + N d + d^2)$ where $O(N d + d^2)$ comes from computing efficiently the quadratic form $\Big(\sum_{i=1}^Ny_ix_i\Big)^\top(\boldsymbol{X}_m^\top \boldsymbol{X}_m)^{-1} \Big(\sum_{i=1}^Ny_ix_i\Big)
$ using the precomputed  
Cholesky decomposition of 
$\boldsymbol{X}_m^\top \boldsymbol{X}_m $. During MCMC, 
sampling from the $K$-component mixture Gaussian proposal costs
$O(K + d^2)$, while the evaluation
of the M-H probability costs 
$O((L + K) d^2 + N d)$ where 
$O((L + K) d^2)$ is due to the evaluation of the pseudo-marginal prior and the mixture proposal. Given that we perform $n$ MCMC iterations to obtain the estimator 
the overall cost is dominated by 
$O(n (L + K) d^2 + n N d)$ which for large number of examples $N$ will be similar to crude MC.

 \section{Adaptive IM with different dimensional Gaussian target and Gaussian proposal}\label{apd:nd}

We consider a zero-mean and identity covariance $d$-dimensional Gaussian target  $\pi(x) = \mathcal{N}(x \vert 0, \bs{I}_d)$ and an adapted Gaussian proposal initialised as $q(x) = \mathcal{N}(x \vert 1_d, \bs{L} \bs{L}^\top)$ where $1_d$ is a vector of ones and the matrix $\bs{L}$
has elements 
$ \bs{L}(i,j) = 1$ if $i \geq j$ and $ \bs{L}(i,j) = 0$ if $i < j$.
We use Algorithm~\ref{alg:sl} in Section~\ref{sec:adp-algo} of this supplement for adaptation. Ranges of variance reduction factors for the estimates of the expected values of each coordinate of the target are presented in Table \ref{tab:5}. While uniform adaptation hyperparameters are used across all dimensions, practical implementations often require larger sample sizes for higher-dimensional problems to achieve a lower adaptation error.

\begin{table}
\centering
\caption{Sample collection after adaptation: range of VRFs for the coordinate estimates of a $d$-dimensional Gaussian target with Gaussian proposal. AP refers to the average acceptance probability. The numbers in the brackets refer to the average coefficients $\hat{c}_1,\hat{c}_2$ and $\hat{c}$ across all replications.}
\label{tab:5}
\begin{tabular}{@{}lcccc@{}}
\hline
  &  $\mu_{n,IMCV}$ &  $\mu_{n,IMCV}^{(\hat{c}_1,\hat{c}_2)}$ & $\mu_{n,RB}$ & $\mu_{n,CIM}^{(\hat{c})}$  \\
\hline
$d=5$    & 268.8 -- 678.0  & 257.2 -- 696.5  & 0.9 -- 1.2  & 20.9 -- 42.0  \\
AP: 0.98  & --  & (1.020, 0.982) & --  & (1.020)  \\ 
\hline
$d=10$   & 124.7 -- 239.5  & 120.4 -- 242.0  & 1.0 -- 1.2  & 10.1 -- 16.4  \\ 
AP: 0.97  & --  & (1.028, 0.967)  & --  & (1.028)  \\ 
\hline
$d=20$   & 40.0 -- 107.0   & 42.8 -- 114.1   & 1.0 -- 1.3  & 4.3 -- 14.8   \\ 
AP: 0.94  & --  & (1.052, 0.946)  & --  & (1.052)  \\ 
\hline
$d=50$   & 8.0 -- 22.0     & 8.6 -- 24.1     & 1.0 -- 1.5  & 2.2 -- 5.9    \\ 
AP: 0.88  & --  & (1.114, 0.873)  & --  & (1.117)  \\ 
\hline
$d=100$  & 2.1 -- 6.9      & 2.4 -- 7.3      & 1.0 -- 1.6  & 1.1 -- 3.0    \\ 
AP: 0.76  & --  & (1.231, 0.753)  & --  & (1.234)  \\ 
\hline
\end{tabular}
\end{table}

Table \ref{tab:5} illustrates that the VRF is inversely related to
dimensionality when all other parameters are held constant. We proceed
by analysing the sensitivity of the VRF to the dimension $d$. We model
the proposal after adaptation as a Gaussian perturbation of the target
in both mean and covariance, derive an exact expression for the
expected KL divergence and an explicit lower bound on the VRF for each
fixed $d$, and discuss which parts of the empirical decay in
Table~\ref{tab:5} and Figure~\ref{fig:allnormal} they explain.
Fig \ref{fig:allnormal} provides a more intuitive display. It can be
seen that, under the same optimisation algorithm and hyperparameters,
as $d$ increases both the acceptance rate and the VRF decrease.

\begin{figure}
\centering
\includegraphics[width=0.9\textwidth]{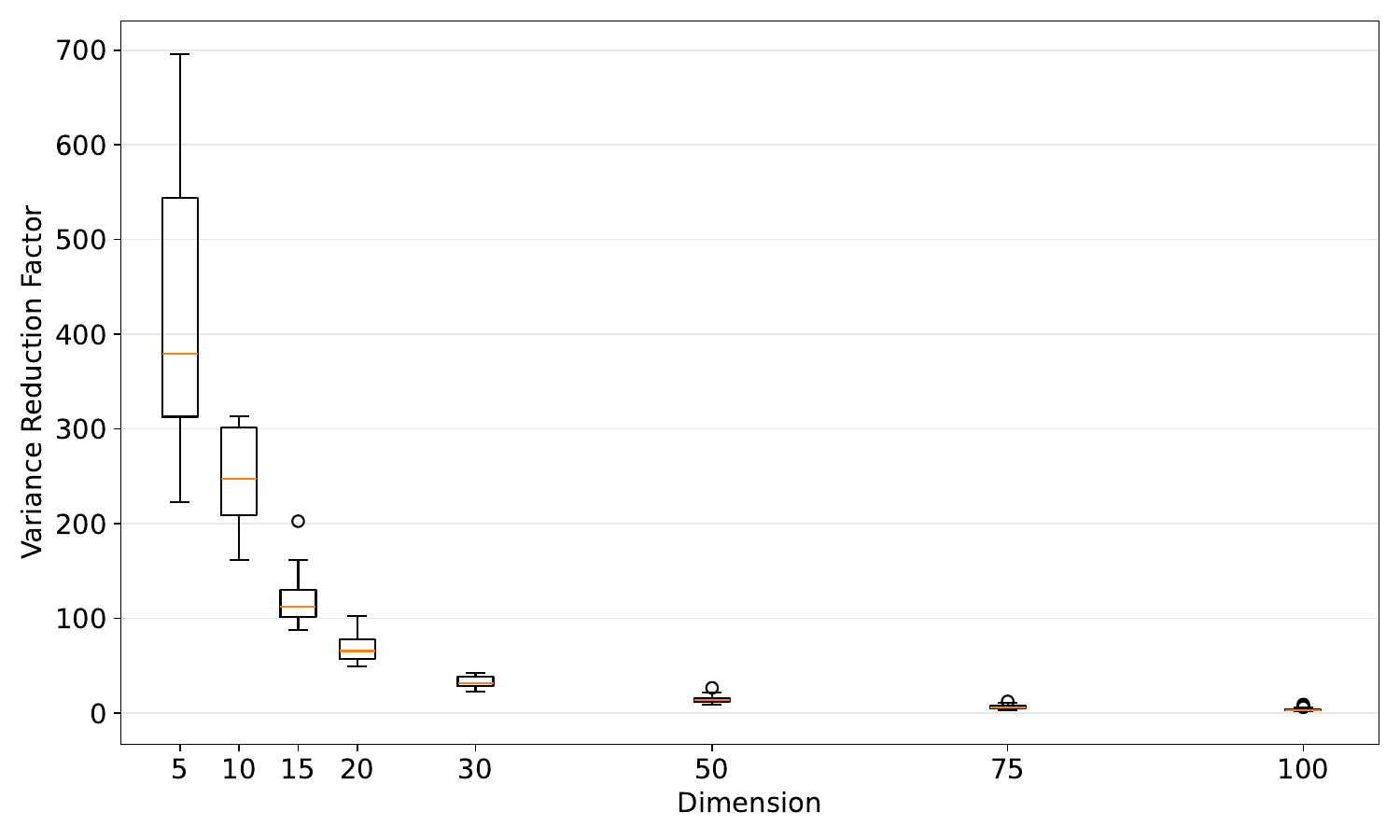}
\caption{\label{fig:allnormal} Boxplot of VRFs for the coordinate
estimates of different dimensional Gaussian target and Gaussian
proposal. The results are obtained from an independent run of the experiment in Table~\ref{tab:5}, so the values differ from those in Table~\ref{tab:5} due to Monte Carlo variability.}
\end{figure}

\begin{assumption}[A9]\label{asu5}
For a Gaussian target $\pi(x)=\mathcal{N}(x\,|\,0,\bs I_d)$, the
proposal after adaptation can be expressed as
$q(x)=\mathcal{N}\big(x\,|\,\Gamma,\,
\bs U(\bs I_d+\bs\Delta)\bs U^\top\big)$, where $\bs U$ is a deterministic orthogonal matrix,
$\Gamma$ has i.i.d.\ elements
$\gamma_i\sim\mathcal N(0,\sigma^2_\gamma)$, and $\bs\Delta$ is diagonal with i.i.d.\ elements $\delta_i$,
independent of $\Gamma$, such that $\delta_i\geq\delta_0$ almost
surely for some $\delta_0>0$ and $\E[\delta_1^2]<\infty$.
\end{assumption}

The condition $\delta_i\geq\delta_0>0$ is necessary for the
analysis. If $\delta_i<0$ for some $i$, the proposal has smaller
variance than the target along the $i$-th column of $\bs U$, so the
density ratio $\pi/q$ is unbounded in that direction and
$w^\star=\infty$; if $\delta_i=0$ and $\gamma_i\neq0$, the two
densities have equal variance but different means, and again
$\pi/q$ is unbounded. In both cases (A3) fails. A strictly larger
variance in every direction is therefore needed to absorb the mean
shift, as the following lemma quantifies.

\begin{lemma}\label{lem:logwstar}
Under (A9), (A3) holds almost surely, with
\begin{equation}\label{eq:logwstar}
    \log w^\star
    \;=\;
    \sum_{i=1}^{d}
    \Big\{\frac12\log(1+\delta_i)
    +\frac{\tilde\gamma_i^2}{2\,\delta_i}\Big\} ,
    \qquad \tilde\gamma:=\bs U^\top\Gamma ,
\end{equation}
where the coordinates $\tilde\gamma_i$ are i.i.d.\
$\mathcal N(0,\sigma^2_\gamma)$, independent of $\bs\Delta$.
Consequently
$\E[\log w^\star]
=\frac{d}{2}\big\{\E[\log(1+\delta_1)]
+\sigma^2_\gamma\,\E[\delta_1^{-1}]\big\}$, which is finite and
positive.
\end{lemma}

\begin{proof}
Let $z=\bs U^\top x$. Since $\bs U$ is orthogonal, the density ratio
$\pi/q$ is invariant under this change of variables, and in the new
coordinates $\pi$ is the density of $\mathcal N(0,\bs I_d)$ while
$q$ is the density of $\mathcal N(\tilde\gamma,\bs I_d+\bs\Delta)$,
whose covariance is diagonal. The density of $\tilde\gamma$ is still
$\mathcal N(0,\sigma^2_\gamma\bs I_d)$. Both densities factorise over
coordinates, and
\[
    \log\frac{\pi(z)}{q(z)}
    = \sum_{i=1}^{d}\ell_i(z_i),
    \qquad
    \ell_i(t)
    := \frac12\log(1+\delta_i)-\frac{t^2}{2}
    +\frac{(t-\tilde\gamma_i)^2}{2(1+\delta_i)} .
\]
Since the sum is separable,
$\log w^\star=\sum_i\sup_{t\in\Real}\ell_i(t)$. Each $\ell_i$ is a
quadratic in $t$ achieving finite maximum
at $t^\ast=-\tilde\gamma_i/\delta_i$. Then
\[
    \ell_i(t^\ast)
    = \frac12\log(1+\delta_i)
    -\frac{\tilde\gamma_i^2}{2\delta_i^2}
    +\frac{\tilde\gamma_i^2(1+\delta_i)}{2\delta_i^2}
    = \frac12\log(1+\delta_i)+\frac{\tilde\gamma_i^2}{2\delta_i} .
\]
Summing over $i$ gives \eqref{eq:logwstar}, and the finiteness of
$w^\star$ is (A3). Taking the expectation gives the expression for
$\E[\log w^\star]$, which is finite because
$\E[\log(1+\delta_1)]\leq\E[\delta_1]<\infty$ and
$\E[\delta_1^{-1}]\leq1/\delta_0$, and is positive because each
summand is.
\end{proof}

\begin{lemma}\label{P2}
Under (A9), for every $d$,
\begin{equation}\label{eq:klexact}
    \E_{\Gamma,\bs\Delta}\big[\mathbb{KL}(q\Vert\pi)\big]
    \;=\;
    \frac{d}{2}\,
    \Big(\sigma^2_\gamma
    +\E\big[\delta_1-\log(1+\delta_1)\big]\Big)
    \;\leq\;
    \frac{d}{2}\,
    \Big(\sigma^2_\gamma+\frac12\,\E\big[\delta_1^2\big]\Big) .
\end{equation}
\end{lemma}

\begin{proof}
Denote $\bs\Sigma:=\bs U\bs\Delta\bs U^\top$, so that
$q=\mathcal N(\Gamma,\bs I_d+\bs\Sigma)$ and
$\pi=\mathcal N(0,\bs I_d)$. For two Gaussian densities
$\mathcal N(m_1,\bs S_1)$ and $\mathcal N(m_2,\bs S_2)$ on $\Real^d$,
\[
    \mathbb{KL}\big(\mathcal N(m_1,\bs S_1)\,\Vert\,
    \mathcal N(m_2,\bs S_2)\big)
    = \frac12\Big\{\mathrm{tr}(\bs S_2^{-1}\bs S_1)
    +(m_2-m_1)^\top\bs S_2^{-1}(m_2-m_1)
    -d+\log\frac{\det\bs S_2}{\det\bs S_1}\Big\} .
\]
With $m_1=\Gamma$, $\bs S_1=\bs I_d+\bs\Sigma$, $m_2=0$ and
$\bs S_2=\bs I_d$ this gives
\[
    \mathbb{KL}(q\Vert\pi)
    = \frac12\Big\{\mathrm{tr}(\bs I_d+\bs\Sigma)
    +\Gamma^\top\Gamma-d-\log\det(\bs I_d+\bs\Sigma)\Big\}
    = \frac12\Big\{\Gamma^\top\Gamma+\mathrm{tr}\{\bs\Sigma\}
    -\log\det(\bs I_d+\bs\Sigma)\Big\} .
\]
Since $\bs U$ is orthogonal, $\mathrm{tr}\{\bs\Sigma\}
=\mathrm{tr}\{\bs\Delta\}=\sum_i\delta_i$ and
$\det(\bs I_d+\bs\Sigma)=\det(\bs I_d+\bs\Delta)=\prod_i(1+\delta_i)$,
so that
\[
    \mathbb{KL}(q\Vert\pi)
    = \frac12\sum_{i=1}^{d}
    \Big\{\gamma_i^2+\delta_i-\log(1+\delta_i)\Big\} ,
\]
where $\gamma_i$ are the coordinates of $\Gamma$ and
$\Gamma^\top\Gamma=\sum_i\gamma_i^2$. Taking expectations over the
i.i.d.\ pairs $(\gamma_i,\delta_i)$, with $\E[\gamma_i^2]
=\sigma^2_\gamma$, yields \eqref{eq:klexact}. Finally, for
$\delta\geq0$ the elementary inequalities
$\delta-\frac12\delta^2\leq\log(1+\delta)\leq\delta$ give
$0\leq\delta-\log(1+\delta)\leq\frac12\delta^2$, and the expectation
is finite by (A9).
\end{proof}

\begin{proposition}\label{prop:vrf}
Assume (A9) and let $F(x)=x_j$ for a fixed coordinate $j$,
so that $\sigma^2_F=1$. Then, for every $d$, almost surely
\begin{equation}\label{eq:vrfbound}
    VRF
    \;\geq\;
    \Big[4\,(1+2w^\star)\,C_2^{1/2}\,
    \{1+\Lambda(w^\star)\}^{1/2}\Big]^{-1}
    \mathbb{KL}(q\Vert\pi)^{-1/2}
    \;=:\;\mathrm{VRF}_d ,
\end{equation}
where $C_2=\max\{3,\;m_j^4+6m_j^2s_j^2+3s_j^4\}$ with
$m_j:=\Gamma_j$ and
$s_j^2:=\{\bs U(\bs I_d+\bs\Delta)\bs U^\top\}_{jj}$. Moreover, as
$d\to\infty$,
\[
    \frac{\log\mathrm{VRF}_d}{d}
    \;\longrightarrow\;
    -\frac12\Big(\E[\log(1+\delta_1)]
    +\sigma^2_\gamma\,\E[\delta_1^{-1}]\Big)
    \qquad\text{in probability},
\]
so that the guaranteed variance reduction factor decays
exponentially in the dimension.
\end{proposition}

\begin{proof}
By Lemma~\ref{lem:logwstar}, (A3) holds, so \eqref{eq:sigbound} and
Lemma~\ref{lem:vKL} with $n=2$ and $\epsilon=\mathbb{KL}(q\Vert\pi)$
give
$\sigma^2_{IMCV}\leq4(1+2w^\star)C_2^{1/2}
\{1+\Lambda(w^\star)\}^{1/2}\mathbb{KL}(q\Vert\pi)^{1/2}$, and
inverting yields \eqref{eq:vrfbound}. The expression for $C_2$ is
the fourth moment $m^4+6m^2s^2+3s^4$ of $\mathcal N(m,s^2)$, with
$(m,s^2)=(0,1)$ under $\pi$ and $(m_j,s_j^2)$ under $q$.

For the asymptotics, write
$\kappa:=\frac12\{\E[\log(1+\delta_1)]+\sigma^2_\gamma\E[\delta_1^{-1}]\}$,
which is finite and positive by (A9), and take logarithms
in \eqref{eq:vrfbound}:
\[
    \frac{\log\mathrm{VRF}_d}{d}
    = -\frac{\log4}{d}
    -\frac{\log(1+2w^\star)}{d}
    -\frac{\log C_2}{2d}
    -\frac{\log\{1+\Lambda(w^\star)\}}{2d}
    -\frac{\log\mathbb{KL}(q\Vert\pi)}{2d} .
\]
We show that the second term converges to $\kappa$ in probability
and the remaining terms to zero; all limits below are in
probability as $d\to\infty$.

By \eqref{eq:logwstar}, $\log w^\star$ is a sum of $d$ i.i.d.\
terms with mean $\kappa$, so $\log w^\star/d\to\kappa$ by the weak
law of large numbers. Since
$\log(1+2w^\star)=\log w^\star+\log(2+1/w^\star)$ with
$\log(2+1/w^\star)\in[\log2,\log3]$, also
$\log(1+2w^\star)/d\to\kappa$.

By the proof of Lemma~\ref{P2}, $\mathbb{KL}(q\Vert\pi)$ is a sum of
$d$ i.i.d.\ terms with finite positive mean $c$, so
$\mathbb{KL}(q\Vert\pi)/d\to c$, and since $\log$ is continuous at
$c>0$,
\[
    \frac{\log\mathbb{KL}(q\Vert\pi)}{d}
    = \frac{\log d}{d}
    + \frac{1}{d}\log\frac{\mathbb{KL}(q\Vert\pi)}{d}
    \;\longrightarrow\;0 .
\]

For the term involving $\Lambda$, note that
$1+\Lambda(w)=(w-1)\log w/(w-1-\log w)$, which extends continuously
to $w=1$ with value $2$. The ratio
$\{1+\Lambda(w)\}/(1+\log w)$ is therefore continuous on
$[1,\infty)$ and tends to $1$ as $w\to\infty$, so it is bounded by a
constant $K$, and
\[
    0\;\leq\;\frac{\log\{1+\Lambda(w^\star)\}}{d}
    \;\leq\;\frac{\log K}{d}
    +\frac{\log d}{d}
    +\frac1d\log\frac{1+\log w^\star}{d} .
\]
Since $\log w^\star/d\to\kappa>0$, the last logarithm converges to
$\log\kappa$, so all three terms vanish and
$\log\{1+\Lambda(w^\star)\}/d\to0$ in probability.

Finally, $m_j=\Gamma_j$ has the $\mathcal N(0,\sigma^2_\gamma)$
distribution for every $d$, and
$s_j^2=\sum_{k=1}^d U_{jk}^2(1+\delta_k)\leq1+\max_{k\leq d}\delta_k$
because $\sum_kU_{jk}^2=1$. The maximum exceeds $d$ only if at least
one $\delta_k$ does, so Markov's inequality gives
\begin{align*}
    \Pr\Big(\max_{k\leq d}\delta_k>d\Big)
    &\;\leq\;\sum_{k=1}^{d}\Pr(\delta_k>d)
    \;=\;d\,\Pr(\delta_1>d)\\
    &\;\leq\;\frac{d\,\E[\delta_1^2]}{d^2}
    \;=\;\frac{\E[\delta_1^2]}{d}
    \;\longrightarrow\;0 .
\end{align*}
On the event $\{s_j^2\leq1+d\}$, whose probability tends to one,
\[
    C_2\;\leq\;\max\big\{3,\;m_j^4+6m_j^2(1+d)+3(1+d)^2\big\} ,
\]
so that $\log C_2\leq2\log d+\log(1+m_j^4)+c$ for a constant $c$ and
all large $d$. Since $\log d/d\to0$ and $m_j$ has a fixed
distribution, $\log C_2/d\to0$ in probability. Collecting the five
terms gives the limit.

\end{proof}

Proposition~\ref{prop:vrf} separates two effects. The divergence
$\mathbb{KL}(q\Vert\pi)$ grows linearly in $d$ by
\eqref{eq:klexact}, so the factor $\mathbb{KL}^{-1/2}$ in
\eqref{eq:vrfbound} contributes a decay of order $d^{-1/2}$ only.
This is how the bound depends on the dimension when $w^\star$ stays
bounded, which is the situation in the adaptive setting under (A6).
Assumption (A9) implies (A3) for every fixed $d$, through the
support condition $\delta_i\geq\delta_0>0$, but it does not keep
$w^\star$ bounded across dimensions, indeed, by Lemma~\ref{lem:logwstar},
$\log w^\star$ grows linearly in $d$, so that $w^\star$ increases
exponentially. The divergence grows only linearly, so
$w^\star$ dominates in \eqref{eq:vrfbound} and the exponential rate
in Proposition~\ref{prop:vrf} comes entirely from the coefficient
$(1+2w^\star)$. Under the perturbation model, the guaranteed variance
reduction therefore deteriorates because the proposal itself
deteriorates with the dimension, not because the estimator depends
on $d$ at fixed proposal quality. Table~\ref{tab:5} and
Figure~\ref{fig:allnormal} are consistent with this picture. The
empirical decay of the VRF is faster than $d^{-1/2}$, and the
average acceptance probability falls from $0.98$ at $d=5$ to $0.76$
at $d=100$, indicating that in those experiments the quality of the
adapted proposal, i.e. $w^\star$, degrades with the dimension.

\section{Adaptive IM algorithm with KL divergence instances}\label{sec:adp-algo}
In all three experiments of Section 3.3 we use the adaptation algorithm combined with doubly stochastic variational inference (DSVI) \citep{Titsias2014DoublyInference} or sticking the landing algorithm \citep{Roeder17} with Adam optimisation \citep{kingma2014adam}. The specific algorithm is described here in Algorithm \ref{alg:DSV} and Algorithm \ref{alg:sl}.

\begin{algorithm}
\raggedright
\caption{Adaptive IM with DSVI and Adam}\label{alg:DSV}
\textbf{Inputs}: Gaussian proposal parameters $\theta = (\mu, \bs{L})$, log-pdf of target distribution $\log \pi(x)$, stepsizes for two parameters ($\xi^{(1)}, \xi^{(2)}$), batch size $B$, the number of batches $\ell$, objective function $F$.
\\
\textbf{Initialisation} Set $i \leftarrow 1, j \leftarrow 1$ and initialise \{$X$, $\theta = (\mu, \bs{L})$, $\xi^{(1)}$, $\xi^{(2)}$\} by \{$X_0$, $q_{\theta_0}=(\mu_0, \bs{L}_0)$, $\xi^{(1)}_0$, $\xi^{(2)}_0$\}.
\begin{algorithmic}[1]
    \While{$i \leq \ell$}
        \State Generate diagonal matrix: $\bs{\Delta}_{\bs{L}_i} \leftarrow \text{diag}(\text{diag}(\bs{L}_i))$.
        \While{$j \leq B$}
        \State Sample $Z_{ij}$ from $\mathcal{N}(z \vert 0, \bs{I}_d)$.
        \State Compute $Y_{ij} \leftarrow \mu_i + L_i Z_{ij}$.
        \State Save $\{\alpha(X_{ij},Y_{ij}), X_{ij}, Y_{ij}\}$ from independent Metropolis $IM(P_i,\pi,q_{\theta_i})$.
        \State Update function 
        $$
        \begin{cases}
            H_{\mu}(Y_{ij})\leftarrow \nabla \log \pi(Y_{ij}),\\
            H_{L}(Y_{ij},Z_{ij})\leftarrow \nabla \log \pi(Y_i) Z_{ij}^\top + \bs{\Delta}_{\bs{L}_i}.
        \end{cases}
        $$
        \State Set $j \leftarrow j + 1$.
        \EndWhile
        \State Adaptation update 
        $$
        \begin{cases}
            \mu_{i+1} \leftarrow \mu_i - \xi^{(1)}_i Adam( \sum_{j=1}^{B}H_{\mu}(Y_{ij})/B),\\
            L_{i+1} \leftarrow L_i - \xi^{(2)}_i Adam( \sum_{j=1}^{B}H_{L}(Y_{ij},Z_{ij})/B).
        \end{cases}
        $$
        \State Calculate analytical result for the batch $E_{q_{\theta_i}}[F]$.
        \State Update $i \leftarrow i + 1, j \leftarrow 1$.
        \EndWhile
    \item Calculate the estimator \eqref{batchest}.
\end{algorithmic}
\textbf{Returns}: An estimate for \eqref{batchest}.
\end{algorithm}

\begin{algorithm}
\raggedright
\caption{Adaptive IM with sticking the landing and Adam}\label{alg:sl}
\textbf{Inputs}: Gaussian proposal parameters $\theta = (\mu, \bs{L})$, log-pdf of target distribution $\log \pi(x)$, stepsizes for two parameters ($\xi^{(1)}$, $\xi^{(2)}$), batch size $B$, the number of batches $\ell$, objective function $F$.
\\
\textbf{Initialisation} Set $i \leftarrow 1, j \leftarrow 1$ and initialise \{$X$, $\theta = (\mu, \bs{L})$, $\xi^{(1)}$, $\xi^{(2)}$\} by \{$X_0$, $q_{\theta_0}=(\mu_0, \bs{L}_0)$, $\xi^{(1)}_0$, $\xi^{(2)}_0$\}.
\begin{algorithmic}[1]
    \While{$i \leq \ell$}
        \While{$j \leq B$}
        \State Sample $Z_{ij}$ from $\mathcal{N}(z \vert 0, \bs{I}_d)$.
        \State Compute $Y_{ij} \leftarrow \mu_i + L_i Z_{ij}$.
        \State Save $\{\alpha(X_{ij},Y_{ij}), X_{ij}, Y_{ij}\}$ from independent Metropolis $IM(P_i,\pi,q_{\theta_i})$.
        \State Update function 
        $$
        \begin{cases}
            H_{\mu}(Y_{ij})\leftarrow \nabla \log \pi(Y_{ij}),\\
            H_{L}(Y_{ij},Z_{ij})\leftarrow (\nabla \log \pi(Y_i) - \nabla \log q_{\theta_i}(Y_i) )Z_{ij}^\top.
        \end{cases}
        $$
        \State Set $j \leftarrow j + 1$.
        \EndWhile
        \State Adaptation update 
        $$
        \begin{cases}
            \mu_{i+1} \leftarrow \mu_i - \xi^{(1)}_i Adam( \sum_{j=1}^{B}H_{\mu}(Y_{ij})/B),\\
            L_{i+1} \leftarrow L_i - \xi^{(2)}_i Adam( \sum_{j=1}^{B}H_{L}(Y_{ij},Z_{ij})/B).
        \end{cases}
        $$
        \State Calculate analytical result for the batch $E_{q_{\theta_i}}[F]$.
        \State Update $i \leftarrow i + 1, j \leftarrow 1$.
        \EndWhile
    \item Calculate the estimator \eqref{batchest}.
\end{algorithmic}
\textbf{Returns}: An estimate for \eqref{batchest}.
\end{algorithm}

\end{document}